\documentclass[leqno,12pt]{article}
\usepackage{a4wide,amsmath,amssymb,amsfonts,mathrsfs,exscale,graphics,amsthm,colonequals,comment,bbold}

\usepackage[curve,matrix,arrow,cmtip]{xy}
\newdir^{ (}{{}*!/-3pt/\dir^{(}}   
\newdir_{ (}{{}*!/-3pt/\dir_{(}}   
\NoComputerModernTips

\def\LocFree{\mathop{\tt LF}\nolimits}
\def\GradLocFree{\mathop{\tt GrLF}\nolimits}
\def\DescFilLocFree{\mathop{\tt FilLF^\bullet}\nolimits}
\def\AscFilLocFree{\mathop{\tt FilLF_\bullet}\nolimits}
\def\Vec{\mathop{\tt Vec}\nolimits}
\def\GradVec{\mathop{\tt GrVec}\nolimits}
\def\DescFilVec{\mathop{\tt FilVec^\bullet}\nolimits}

\newcommand{\descfil}{{{\rm fil}^\bullet}}
\newcommand{\ascfil}{{{\rm fil}_\bullet}}
\newcommand{\FZip}{\mathop{\text{$F$-{\tt Zip}}}\nolimits}
\newcommand{\GZip}{\mathop{\text{$G$-{\tt Zip}}}\nolimits}
\newcommand{\GhatZip}{\mathop{\text{$\hat G$-{\tt Zip}}}\nolimits}
\newcommand{\GZipFunctor}{\mathop{\text{$G$-{\tt ZipFun}}}\nolimits}
\def\GhZipFunctor{\mathop{\text{\smash{$\Ghat$}-{\tt ZipFun}}}\nolimits}

\newcommand{\GLnZip}{\mathop{{\rm GL}_n\text{-{\tt Zip}}}\nolimits}
\newcommand{\SLnZip}{\mathop{{\rm SL}_n\text{-{\tt Zip}}}\nolimits}
\newcommand{\CSpnZip}{\mathop{{\rm CSp}_n\text{-{\tt Zip}}}\nolimits}
\newcommand{\SpnZip}{\mathop{{\rm Sp}_n\text{-{\tt Zip}}}\nolimits}
\newcommand{\OnZip}{\mathop{{\rm O}_n\text{-{\tt Zip}}}\nolimits}
\newcommand{\COnZip}{\mathop{{\rm CO}_n\text{-{\tt Zip}}}\nolimits}

\newcommand{\UnZip}{\mathop{{\rm U}_n\text{-{\tt Zip}}}\nolimits}
\newcommand{\CUnZip}{\mathop{{\rm CU}_n\text{-{\tt Zip}}}\nolimits}

\newcommand{\SLnRep}{\mathop{{\rm SL}_n\text{-{\tt Rep}}}\nolimits}
\def\forget{{\rm forg}}
\def\Rep{\mathop{\tt Rep}\nolimits}
\def\GRep{\mathop{\text{$G$-{\tt Rep}}}\nolimits}
\def\GhRep{\mathop{\text{${\hat G}$-{\tt Rep}}}\nolimits}

\newcommand{\descgr}{\gr^\bullet_C}
\newcommand{\ascgr}{\gr_\bullet^D}

\def\Alt{{\Lambda}}
\def\Sym{{\rm S}}

\def\CHom{\mathop{\mathcal Hom}\nolimits}
\def\UBOne{{\underline{\hbox{\rm1\kern-2.3ptl}}}}
\newcommand{\UCL}{{\underline{\CL}}}
\newcommand{\UCM}{{\underline{\CM}}}
\newcommand{\UCN}{{\underline{\CN}}}
\newcommand{\UI}{{\underline{I}}}
\newcommand{\Un}{{\underline{n}}}
\newcommand{\Ghat}{{\hat G}}
\newcommand{\Hhat}{{\hat H}}
\newcommand{\Lhat}{{\hat L}}
\newcommand{\Phat}{{\hat P}}
\newcommand{\What}{{\hat W}}

\newcommand{\kbar}{{\bar k}}
\def\dothatw{{\vphantom{\hat w}\smash{\dot{\hat w}}}}
\let\into\hookrightarrow

\newcommand{\Gm}[1]{\BG_{m,#1}}

\numberwithin{equation}{section}
\def\UseTheoremCounterForNextEquation{\setcounter{equation}{\value{theorem}}\addtocounter{theorem}{1}}

\theoremstyle{plain}
\newtheorem{theorem}{Theorem}[section]
\newtheorem{lemma}[theorem]{Lemma}
\newtheorem{corollary}[theorem]{Corollary}
\newtheorem{proposition}[theorem]{Proposition}

\theoremstyle{definition}
\newtheorem{definition}[theorem]{Definition}

\newtheorem{construction}[theorem]{Construction}
\newtheorem{remark}[theorem]{Remark}
\newtheorem{example}[theorem]{Example}

\newcommand\lto{\longrightarrow}
\newcommand\ltoover[1]{\mathrel{\smash{\overset{#1}{\lto}}}}
\newcommand\varto[1]{\mathrel{\hbox to #1pt{\rightarrowfill}}}

\newcommand{\bijective}{\leftrightarrow}

\newcommand{\sends}{\mapsto}

\newcommand{\iso}{\overset{\sim}{\to}}
\newcommand{\liso}{\overset{\sim}{\lto}}
\newcommand\lisoover[1]{\mathrel{\smash{\overset{#1}{\liso}}}}

\renewcommand{\implies}{\Rightarrow}

\let\longto\longrightarrow
\let\onto\twoheadrightarrow
\def\isoto{\stackrel{\sim}{\longto}}

\let\phi\varphi
\let\epsilon\varepsilon
\let\setminus\smallsetminus
\let\emptyset\varnothing

\let\ge\geqslant

\newcommand{\dodot}{{}_{\bullet}}

\newcommand{\set}[2]{\{\,#1\ ;\  #2\,\}}

\newcommand{\updot}{{}^{\bullet}}
\newcommand{\vdual}{{}^{\vee}}

\def\Norm{\mathop{\rm Norm}\nolimits}
\def\Aut{\mathop{\rm Aut}\nolimits}
\def\Cent{\mathop{\rm Cent}\nolimits}

\def\Frob{\mathop{\rm Frob}\nolimits}

\def\End{\mathop{\rm End}\nolimits}
\def\Gal{\mathop{\rm Gal}\nolimits}

\def\Spec{\mathop{\rm Spec}\nolimits}

\def\Coker{\mathop{\rm coker}\nolimits} 

\def\Image{\mathop{\rm im}\nolimits} 
\def\Ker{\mathop{\rm ker}\nolimits}

\def\Stab{\mathop{\rm Stab}\nolimits}

\def\Transp{\mathop{\rm Transp}\nolimits}

\def\Lie{\mathop{\rm Lie}\nolimits}
\def\Ad{\mathop{\rm Ad}\nolimits}

\def\GL{\mathop{\rm GL}\nolimits}
\def\SL{\mathop{\rm SL}\nolimits}
\def\groupCO{\mathop{\rm CO}\nolimits}
\def\groupO{\mathop{\rm O{}}\nolimits}
\def\SO{\mathop{\rm SO}\nolimits}

\def\Sp{\mathop{\rm Sp}\nolimits}
\def\CSp{\mathop{\rm CSp}\nolimits}
\def\groupU{\mathop{\rm U{}}\nolimits}
\def\SU{\mathop{\rm SU}\nolimits}
\def\groupCU{\mathop{\rm CU}\nolimits}

\def\Sch{{\bf Sch}}

\def\UHom{\mathop{\underline{\rm Hom}}\nolimits}
\def\UIsom{\mathop{\underline{\rm Isom}}\nolimits}
\def\UAut{\mathop{\underline{\rm Aut}}\nolimits}

\def\sgn{\mathop{\rm sgn}\nolimits}

\def\codim{\mathop{\rm codim}\nolimits}

\def\proj{\mathop{\rm pr}\nolimits}
 
\def\intaut{\mathop{\rm int}\nolimits}

\def\gr{{\rm gr}}

\def\id{{\rm id}}

\def\opp{{\rm op}}

\newcommand{\BD}{{\mathbb{D}}}

\newcommand{\BF}{{\mathbb{F}}}
\newcommand{\BG}{{\mathbb{G}}}

\newcommand{\BQ}{{\mathbb{Q}}}
\newcommand{\BR}{{\mathbb{R}}}

\newcommand{\BX}{{\mathbb{X}}}

\newcommand{\BZ}{{\mathbb{Z}}}
\newcommand{\Bone}{{\mathbb{1}}}

\newcommand{\Fz}{{\mathfrak{z}}}

\newcommand{\CC}{{\mathcal C}}
\newcommand{\CD}{{\mathcal D}}

\newcommand{\CF}{{\mathcal F}}

\newcommand{\CL}{{\mathcal L}}
\newcommand{\CM}{{\mathcal M}}
\newcommand{\CN}{{\mathcal N}}
\newcommand{\CO}{{\mathcal O}}

\newcommand{\CR}{{\mathcal R}}

\newcommand{\CT}{{\mathcal T}}
\newcommand{\CU}{{\mathcal U}}
\newcommand{\CV}{{\mathcal V}}

\newcommand{\CX}{{\mathcal X}}

\newcommand{\CZ}{{\mathcal Z}}

\newcommand{\Hline}{\underline{H}}

\newcommand{\Cscr}{{\mathscr C}}

\newcommand{\Hscr}{{\mathscr H}}

\newcommand{\Oscr}{{\mathscr O}}

\newcommand{\eps}{\varepsilon}

\newcommand{\defeq}{\colonequals}

\newcommand{\leftexp}[2]{{\vphantom{#2}}^{#1}{#2}}
\newcommand{\doubleexp}[3]{\leftexp{#1}{#2}^{#3}}

\theoremstyle{definition}

\newcommand{\BIGOP}[1]{\mathop{\mathchoice%
{\raise-0.22em\hbox{\huge $#1$}}%
{\raise-0.05em\hbox{\Large $#1$}}{\hbox{\large $#1$}}{#1}}}

\newcommand{\BIGboxplus}{\mathop{\mathchoice%
{\raise-0.35em\hbox{\huge $\boxplus$}}%
{\raise-0.15em\hbox{\Large $\boxplus$}}{\hbox{\large $\boxplus$}}{\boxplus}}}

\newcounter{listcounter}
\newcounter{deflistcounter}
\newcounter{equivcounter}

\newskip{\itemsepamount}
\newskip{\topsepamount}
\newenvironment{assertionlist}{%
  \begin{list}
    {\upshape (\arabic{listcounter})}
    {\setlength{\leftmargin}{18pt}
     \setlength{\rightmargin}{0pt}
     \setlength{\itemindent}{0pt}
     \setlength{\labelsep}{5pt}
     \setlength{\labelwidth}{13pt}
     \setlength{\listparindent}{\parindent}
     \setlength{\parsep}{0pt}
     \setlength{\itemsep}{\itemsepamount}
     \setlength{\topsep}{\topsepamount}
     \usecounter{listcounter}}}
  {\end{list}}

% ...

\newenvironment{simplelist}{%
  \begin{list}{}
     {\setlength{\leftmargin}{25pt}
      \setlength{\rightmargin}{0pt}
      \setlength{\itemindent}{0pt}
      \setlength{\labelsep}{5pt}
      \setlength{\listparindent}{\parindent}
      \setlength{\parsep}{0pt}
      \setlength{\itemsep}{0pt}
      \setlength{\topsep}{0pt}}}
  {\end{list}}

\renewenvironment{enumerate}{\begin{list}{\arabic{ctr}.}{\usecounter{ctr}\setlength{\topsep}{3mm}
  \setlength{\leftmargin}{12mm}\setlength{\labelsep}{3mm}\setlength{\labelwidth}{10mm}}}{\end{list}}

\makeatletter
\let\@fnsymbol\@arabic
\makeatother

\begin{document}
\newcounter{ctr} 

\title{$F$-zips with additional structure}

\author{%
Richard Pink\footnote{Dept. of Mathematics,
 ETH Z\"urich, 
 CH-8092 Z\"urich,
 Switzerland,
 \tt pink@math.ethz.ch} \hspace{1cm}
 Torsten Wedhorn\footnote{Dept. of Mathematics,
 University of Paderborn, 
 D-33098 Paderborn,
 Germany,
 \tt wedhorn@math.uni-paderborn.de} \hspace{1cm}
 Paul Ziegler\footnote{Dept. of Mathematics,
 ETH Z\"urich
 CH-8092 Z\"urich,
 Switzerland,
 \tt pziegler@math.ethz.ch}
}

\date{\today}
\maketitle
\abstract{
Let $\BF_q$ be a fixed finite field of cardinality~$q$. An $F$-zip over a scheme $S$ over $\BF_q$ is a certain object of semi-linear algebra consisting of a locally free sheaf of $\CO_S$-modules with a descending filtration and an ascending filtration and a $\Frob_q$-twisted isomorphism between the respective graded sheaves.
In this article we define and systematically investigate what might be called ``$F$-zips with a $G$-structure'', for an arbitrary reductive linear algebraic group~$G$ over~$\BF_q$. 

These objects come in two incarnations.
One incarnation is an exact $\BF_q$-linear tensor functor from the category of finite dimensional representations of $G$ over~$\BF_q$ to the category of $F$-zips over~$S$. Locally any such functor has a type~$\chi$, which is a cocharacter of $G_k$ for a finite extension $k$ of $\BF_q$ that determines the ranks of the graded pieces of the filtrations.
The other incarnation is a certain $G$-torsor analogue of the notion of $F$-zips.
We prove that both incarnations define stacks that are naturally equivalent to a quotient stack of the form $[E_{G,\chi}\backslash G_k]$ that was studied in our earlier paper \cite{PWZ1}. By the results of \cite{PWZ1} they are therefore smooth algebraic stacks of dimension $0$ over~$k$. Using \cite{PWZ1} we can also classify the isomorphism classes of such objects over an algebraically closed field, describe their automorphism groups, and determine which isomorphism classes can degenerate into which others.

For classical groups we can deduce the corresponding results for twisted or untwisted symplectic, orthogonal, or unitary $F$-zips, a part of which has been described before in \cite{moonwed}. 
The results can be applied to the algebraic de Rham cohomology of smooth projective varieties (or generalizations thereof) and to truncated Barsotti-Tate groups of level $1$.
In addition, we hope that our systematic group theoretical approach will help to understand the analogue of the Ekedahl-Oort stratification of the special fibers of arbitrary Shimura varieties.
}



\newpage
\section{Introduction}
\label{Intro}


\subsection{Background}
\label{background}

Let $X \to S$ be a smooth proper morphism of schemes in characteristic $p > 0$ whose Hodge spectral sequence degenerates and is compatible with base change. In~\cite{moonwed} Moonen and the second author showed that its relative De Rham cohomology $H^{\bullet}_{\rm DR}(X/S)$ carries the structure of a so-called \emph{$F$-zip} over $S$, namely: It is a locally free sheaf of $\CO_S$-modules of finite rank together with two filtrations (the ``Hodge'' and the ``conjugate'' filtration) and a Frobenius linear isomorphism between the associated graded vector spaces (the ``Cartier isomorphism''). They showed that the isomorphism classes of $F$-zips of fixed dimension $n$ and with a fixed type of Hodge filtration over an algebraically closed field are in natural bijection with the orbits under $\GL_{n,k}$ in a variant $Z'_I$ of the varieties $Z_I$ studied by Lusztig \cite{Lusztig:parsheaves1}, 
\cite{Lusztig:parsheaves2}. They studied the analogous varieties $Z_I'$ for arbitrary reductive groups $G$ defined over a finite field and determined the $G$-orbits in them as analogues of the $G$-stable pieces in~$Z_I$. By specializing $G$ to classical groups they deduced from this a classification of $F$-zips with certain additional structure, e.g., with a non-degenerate symmetric or alternating form. 

In~\cite{PWZ1} the present authors showed that the quotient stack $[G\backslash Z_I']$ is isomorphic to a quotient stack of the form $[E_{\chi} \backslash G]$, where $E_{\chi}$ is certain linear algebraic group depending on the choice of a cocharacter $\chi$ of $G$. We studied this quotient stack in detail, classifying the $E_\chi$-orbits in $G$ by a subset of the Weyl group of $G$ and describing their closure relation using a variant of the Bruhat order.


\subsection{Main idea}
\label{Mi}

The aim of this paper is to define and investigate what might be called ``$F$-zips with a $G$-structure'', for an arbitrary reductive linear algebraic group~$G$. 

As a guideline let us first review the analogous case of vector bundles. Recall that giving a vector bundle $E$ of constant rank $n$ on a manifold or a scheme $S$ over a field $k$ is equivalent to giving the associated $\GL_{n,k}$-torsor. For a subgroup $G \subset\GL_{n,k}$, the choice of a $G$-torsor $I$ within this $\GL_{n,k}$-torsor is called a $G$-structure on~$E$. The vector bundle $E$ can be recovered as the pushout of $I$ with the given $n$-dimensional representation of~$G$, so giving a vector bundle with a $G$-structure is really equivalent to giving a $G$-torsor~$I$. 

At this point we can disregard the special role of the original representation and form the pushout of $I$ with all finite dimensional representations of~$G$. This yields an exact $k$-linear tensor functor from the Tannakian category $\GRep$ of finite dimensional representations of $G$ over $k$ to the category of vector bundles on~$S$, which is known (for example by Nori \cite[Prop.~2.9]{Nori1976}) to be again equivalent to giving~$I$.
Altogether such a functor is therefore equivalent to giving a vector bundle with a $G$-structure on~$S$.

These observations suggest that objects with a $G$-structure in a more general exact $k$-linear tensor category $\CT$ should be equivalent to, or might even be defined as, exact $k$-linear tensor functors $\GRep\to\CT$, and that they should be equivalent to $G$-torsor analogues of the objects in~$\CT$. In the cases of graded or filtered vector bundles equivalences of this kind were in fact derived in some cases in Saavedra \cite[IV.1--2]{Saavedra} and in general by the third author \cite{ZieglerFFF}.
The principle was also applied to $F$-isocrystals with additional structure by Rapoport and Richartz \cite{RaRi1996}.

The program for the present paper is therefore to develop this approach for the category of $F$-zips and to classify the ensuing objects using the results of~\cite{PWZ1}.


\subsection{$F$-zips with a $G$-structure}
\label{FzwaGs}

Let $\BF_q$ be a fixed finite field with $q$ elements, and consider a scheme $S$ over~$\BF_q$. Recall from \cite{moonwed} that an \emph{$F$-zip over $S$} is a tuple $\UCM = (\CM,C^\bullet,D_\bullet,\phi_\bullet)$ consisting of a locally free sheaf of $\CO_S$-modules of finite rank $\CM$ on~$S$, a descending filtration $C^\bullet$ and an ascending filtration $D_\bullet$ of~$\CM$, and an $\CO_S$-linear isomorphism $\phi_i\colon (\gr_C^i\CM)^{(q)} \stackrel{\sim}{\to} \gr^D_i\CM$ for every $i\in\BZ$, where $(\ )^{(q)}$ denotes the pullback by the Frobenius morphism $x\mapsto x^q$. In a natural way (see Section~\ref{Fzips}) the $F$-zips over $S$ are the objects of an exact $\BF_q$-linear tensor category $\FZip(S)$.

Let $G$ be a reductive linear algebraic group over~$\BF_q$, and let $k$ be a finite extension of~$\BF_q$. In the body of the paper we consider not necessarily connected groups, 
but to simplify notations in this introduction we stick to a connected group~$G$.
For simplicity we also assume that $G$ splits over~$k$, so that every conjugacy class of cocharacters of $G$ over any extension field of~$k$ possesses a representative that is defined over~$k$. Let $\GRep$ denote the $\BF_q$-linear abelian tensor category of finite-dimensional rational representations of $G$ over~$\BF_q$. The role of ``$F$-zips with a $G$-structure'' is played by the following objects:

\begin{definition}[cf.\ Definition~\ref{ZipFunctorDefinition}]
\label{0GZipFuncDef}
For any scheme $S$ over~$k$, a \emph{$G$-zip functor over~$S$} is an exact $\BF_q$-linear tensor functor 
$$\Fz\colon \GRep\to\FZip(S).$$
\end{definition}

As $S$ varies, these objects form a category $\GZipFunctor$ fibered in groupoids over the category of schemes over $k$. 
It is not hard to show that $\GZipFunctor$ is a stack over~$k$ (see Proposition~\ref{GZipFStack}). It possesses a natural decomposition that is indexed by conjugacy classes of cocharacters of~$G$, defined as follows.

\medskip
Let $\chi$ be a cocharacter of the group $G_k$ obtained from $G$ by base change. Then $\chi$ induces a grading on $V_k := V\otimes_{\BF_q}k$ for every representation $V$ of~$G$ and thus an $\BF_q$-linear tensor functor $\gamma_\chi$ from $\GRep$ to the category of graded $k$-vector spaces. 
On the other hand, any $G$-zip functor $\Fz$ over~$S$ induces  an $\BF_q$-linear tensor functor from $\GRep$ to the category of graded locally free sheaves of $\CO_S$-modules on~$S$ which sends $V$ to $\descgr\circ\Fz(V)$.

\begin{definition}[cf.\ Definitions \ref{DefTypeChi} and \ref{FunctorType}]
\label{0GZipFuncTypeDef}
A $G$-zip functor $\Fz$ over $S$ is called \emph{of type $\chi$} if the graded fiber functors $\descgr\circ\Fz$ and $\gamma_\chi$ are fpqc-locally isomorphic. The substack of $\GZipFunctor$ of $G$-zip functors of type $\chi$ is denoted $\GZipFunctor^\chi$.
\end{definition}

\begin{theorem}[cf.\ Corollary~\ref{GZipFDecomp}]
\label{0GZipFuncType}
Every $G$-zip functor over a connected scheme has a type. 
Each $\GZipFunctor^\chi$ is an open and closed substack of $\GZipFunctor$.
\end{theorem}


In Section~\ref{Fadd} we work out equivalent but simpler descriptions of $G$-zip functors for certain classical groups. The rough idea in all cases is that any $G$-zip functor is already determined up to unique isomorphism by its restriction to a certain finite subcategory of $\GRep$ and that, conversely, any suitable functor from this subcategory to the category of $F$-zips extends to a $G$-zip functor on all of $\GRep$. For instance, giving a $\GL_n$-zip functor is equivalent to giving an $F$-zip $\UCM$ of constant rank~$n$, and giving an $\SL_n$-zip functor is equivalent to giving an $F$-zip $\UCM$ of constant rank~$n$ together with an isomorphism between its highest exterior power $\Alt^n\UCM$ and the unit object $\UBOne(0)$. Similarly, giving an $\Sp_n$-zip, resp.\ $\groupO_n$-zip functor is equivalent to giving a symplectic, resp.\ orthogonal $F$-zip of constant rank~$n$, by which we mean an $F$-zip $\UCM$ of constant rank~$n$ together with an epimorphism of $F$-zips $\Alt^2\UCM\onto \UBOne(0)$, resp.\ $\Sym^2\UCM\onto \UBOne(0)$, whose underlying pairing of locally free sheaves is nondegenerate everywhere. We also discuss the relation between $\groupU_{n}$-zip functors and unitary $F$-zips, as well as twisted versions of these equivalences associated to the groups of similitudes $\CSp_n$ and $\groupCO_n$ and $\groupCU_n$.


\subsection{$G$-zips}
\label{Gz}

To describe the stack of $G$-zip functors $\GZipFunctor^\chi$ in detail we use the following $G$-torsor analogue of $F$-zips. Let $G$ and $\chi$ be as above, and let $P_\pm = L\ltimes U_\pm$ be the associated pair of opposite parabolic subgroups of~$G_k$.

\begin{definition}[cf. Definition \ref{GZipDef}]\label{0GZipDef}
A \emph{$G$-zip of type $\chi$} over a scheme $S$ over $k$ is a tuple $\UI = (I,I_+,I_-,\iota)$ consisting of a right $G_k$-torsor $I$ over~$S$, a right $P_+$-torsor $I_+\subset\nobreak I$, a right $P^{(q)}_-$-torsor $I_-\subset I$, and an isomorphism of $L^{(q)}$-torsors $\iota\colon I^{(q)}_+/U^{(q)}_+ \stackrel{\sim}{\longto} I_-/U^{(q)}_-$.
\end{definition}

As $S$ varies, these objects form a category $\GZip^\chi$ fibered in groupoids over the category of schemes over $k$. It is not hard to show that $\GZip^\chi$ is a stack over~$k$ (see Proposition~\ref{GZipsStack}).

\medskip
To $G$ and $\chi$ we can also associate a natural \emph{algebraic zip datum} in the sense of~\cite{PWZ1} (see Definition \ref{ZipDatumDef}). The associated \emph{zip group} is the linear algebraic group 
\[
E_{G,\chi} \defeq \bigl\{(\ell u_+,\ell^{(q)}u_-) \bigm| \ell\in L,\ u_+\in U_+, u_-\in U_-^{(q)}\bigr\} 
\ \subset\ P_+ \times_k P_-^{(q)}
\]
which acts from the left hand side on $G_k$ by $(p_+,p_-)\cdot g \ \defeq\ p_+ g p_-^{-1}$. We can thus form the algebraic quotient stack $[E_{G,\chi}\backslash G_k]$.

\begin{theorem}[cf.\ Proposition~\ref{GZipStack}, Theorem~\ref{GZip=GZipFun}, and Corollary~\ref{SmoothStack}]
\label{0IdentifyGZips}
The stacks $\GZipFunctor^{\chi}$ and $\GZip^{\chi}$ and $[E_{G,\chi}\backslash G_k]$ are naturally equivalent. They are smooth algebraic stacks of dimension $0$ over~$k$.
\end{theorem}

In particular, it is equivalent to give a $G$-zip functor of type $\chi$ over $S$, or a $G$-zip of type $\chi$ over $S$, or a morphism $S\to [E_{G,\chi}\backslash G_k]$ over~$k$. The equivalences are in fact obtained by explicit constructions (see Subsections \ref{Gzips2} and \ref{FunctorsZips}).


\subsection{Classification}
\label{Classi}

Using the results of~\cite{PWZ1} we can now describe the stack of $G$-zips of type $\chi$ in detail. By Theorem \ref{0IdentifyGZips}, all the following statements hold equivalently for $G$-zip functors of type~$\chi$. By the results of Section~\ref{Fadd} they also hold for symplectic, orthogonal, resp.\ unitary $F$-zips, and so on.

Let $W$ be the Weyl group of $G$, let $I \subset W$ be the subset of simple reflections corresponding to $P_+$, and let $W_I$ be the subgroup of $W$ generated by $I$. 
Let ${}^IW$ be the set of elements $w \in W$ that are of minimal length in their right coset $W_Iw$. We endow ${}^IW$ with a certain partial order $\preceq$ which is somewhat complicated to describe (and in general strictly finer than the Bruhat order; see (\ref{DefPartialOrder}) and Example \ref{SplitExample}). This turns ${}^IW$ into a finite topological space (see Proposition \ref{TopPartial1}), which we can compare with the topological space underlying the algebraic stack $\GZip^{\chi}$ (see Subsection~\ref{Quot1}).


\begin{theorem}[cf.\ Theorem~\ref{TopGZip1}]
\label{0PropGZips1}
The topological space underlying $\GZip^{\chi}$ is naturally homeomorphic to ${}^IW$. 
In particular, there is a natural bijection between the set of isomorphism classes of $G$-zips of type $\chi$ over an algebraically closed field $K$ containing $k$ and the set~${}^IW$.
\end{theorem}

For any $G$-zip $\UI$ of type $\chi$ over a scheme $S$ over $k$ we thus obtain a finite stratification of $S$ by the isomorphism type of~$\UI$. This generalizes the $F$-zip stratification defined in~\cite{moonwed} in the case $G = \GL_{n,\BF_q}$, as well as the Ekedahl-Oort stratification of the moduli space of $g$-dimensional principally polarized abelian varieties in the case $G = \CSp_{2g,\BF_q}$. The partial order $\preceq$ yields information on the closure relations between these strata (see (\ref{GZipStratClos}) and Proposition \ref{Generizing1}). Using a result of Yatsyshyn \cite{Yatsyshyn} we can also deduce a purity result (see Proposition~\ref{Purity2}).
Furthermore, the description of point stabilizers in $E_{G,\chi}$ from \cite{PWZ1} Theorem~8.1 yields information on automorphism groups of $G$-zips; in particular:

\begin{theorem}[cf.\ Proposition~\ref{AutoGZip}]
\label{0PropGZips2}
The automorphism group scheme of the $G$-zip of type $\chi$ over an algebraically closed field $K$ containing $k$ corresponding to $w \in {}^IW$ is an extension of a finite group (see Proposition~\ref{AutoGZip} for its precise description) by a connected unipotent group of dimension $\dim(G/P_+) - \ell(w)$, where $\ell(\ )$ denotes the length function on the Coxeter group~$W$.
\end{theorem}


\subsection{Applications}
\label{Appli}

In Section~\ref{EO} we study the de Rham cohomology of a smooth proper Deligne-Mumford stack $\CX \to S$ whose Hodge spectral sequence degenerates and is compatible with arbitrary base change. 
For all $d \geq 0$ we obtain an $F$-zip $\Hline^d_{\rm DR}(\CX/S)$. If $\CX$ is a scheme, the cup product induces for all $d$, $e \geq 0$ a morphism of $F$-zips
\[
\cup\colon \Hline^d_{\rm DR}(\CX/S) \otimes \Hline^e_{\rm DR}(\CX/S) \longto \Hline^{d+e}_{\rm DR}(\CX/S).
\]
If $\CX \to S$ is a smooth proper morphism of schemes with geometrically connected fibers of dimension $n$, the cup product turns $\Hline^n_{\rm DR}(\CX/S)$ into a twisted symplectic or orthogonal $F$-zip, depending on the parity of~$n$.

\medskip
In Subsection~\ref{BT1} we attach an $F$-zip to any truncated Barsotti-Tate group of level $1$ over a scheme $S$ of characteristic $p$. This construction improves the one given in~\cite{moonwed} where $S$ was assumed to be perfect. 

\medskip
In addition, we hope that our results can be applied to the special fibers of arbitrary Shimura varieties, where $G$ is the reduction modulo $p$ of the connected reductive linear algebraic group over~$\BQ$ that gives rise to the Shimura variety. In that case our systematic group theoretical approach should prove especially valuable. For good reductions of Shimura varieties of Hodge type some progress has already been made. C.~Zhang has defined in \cite{zhang} a smooth morphism from the special fiber of Kisin's integral model to the stack of $G$-zips of a certain type $\chi$ yielding a description of Ekedahl-Oort strata for the special fiber. This has been used by D.~Wortmann in \cite{Wort_MuOrd} to prove that the conjectured candidate for the generic Newton stratum is indeed open and dense in the special fiber.


\subsection{Contents of the paper}
\label{Cotp}

As a preparation we begin by recalling some properties of quotient stacks in Section~\ref{Quot}. 

In Section~\ref{Gzips} we first introduce some general notation used throughout the rest of the article. We mostly work with a not necessarily connected linear algebraic group $\Ghat$ over $\BF_q$ whose identity component $G$ is reductive. Besides a cocharacter $\chi$ of~$G_k$, the basic data also requires the choice of a subgroup $\Theta$ of the group of connected components of the stabilizer of~$\chi$. We then define the general notion of $\Ghat$-zips of type $(\chi,\Theta)$ and prove that they form a smooth algebraic stack of dimension $0$ 
over~$k$ that is naturally isomorphic to a quotient stack of the form $[E_{\Ghat,\chi,\Theta}\backslash \Ghat_k]$ that was studied in~\cite{PWZ1}.
The remainder of Section~\ref{Gzips} contains an assortment of results on the topological space underlying this stack, on the associated stratification, and on automorphisms.

We use Section~\ref{Prelim} to collect some generalities concerning locally free sheaves, gradings, filtrations, and alternating and symmetric powers. In Section~\ref{FFFAC} we recall some results of the third author on filtered and graded fiber functors on the Tannakian category $\GhRep$.

In Section~\ref{Fzips} we endow the category of $F$-zips over a scheme $S$ with the structure of an exact rigid tensor category. Section~\ref{Functors} then contains the definition of $\Ghat$-zip functors. Here we use the results recalled in Section~\ref{FFFAC} to prove that every $\Ghat$-zip functor over a connected scheme $S$ has a type~$\chi$. With the unique maximal possible choice of $\Theta$ we then establish a natural isomorphism between the stack of $\Ghat$-zip functors of type $\chi$ and the stack of $\Ghat$-zips of type $(\chi,\Theta)$. Consequently the stack of $\Ghat$-zip functors of type $\chi$ is also a smooth algebraic stack of dimension $0$ that is naturally isomorphic to $[E_{\Ghat,\chi,\Theta}\backslash \Ghat_k]$.

In Section~\ref{Fadd} we go through a list of eight classical groups, in each case describing an equivalence between $\Ghat$-zip functors and $F$-zips of a given rank with a certain embellishment such as an alternating or symmetric or hermitian form within the category of $F$-zips.

In the last Section~\ref{EO} we discuss applications to the algebraic de Rham cohomology of certain Deligne-Mumford stacks and to truncated Barsotti-Tate groups of level $1$.

\bigskip

\textsc{Acknowledgment}: We thank Chao Zhang for pointing out a mistake in an earlier version of the proof of Lemma~3.5 and Eike Lau for explaining us Remark~\ref{TruncBTSmooth}. We are grateful to the referee for helpful remarks.



\section{General properties of quotient stacks}
\label{Quot}

As a preparation we discuss some general properties of algebraic group actions and quotient stacks.


\subsection{Closure relation for an algebraic group action}
\label{QuotTop}

First recall that a topological space $Z$ is called $T_0$ (or Kolmogorov) if for any two distinct points, at least one of them possesses a neighborhood that does not contain the other. Abbreviating the closure of a subset by $\overline{(\ )}$, as usual, this is equivalent to saying that for any $z',z\in Z$ we have $z'=z$ if and only if both $z'\in\overline{\{z\}}$ and $z\in\overline{\{z'\}}$. On the other hand, recall that a partial order $\preceq$ on a set $Z$ is a transitive binary relation which is antisymmetric in the sense that $z'=z$ if and only if both $z'\preceq z$ and $z\preceq z'$.
With these observations the following well-known fact is easy to prove:

\begin{proposition}\label{TopPartial1}
For any finite $T_0$ topological space $Z$ the relation $z'\preceq z$ $:\Leftrightarrow$ ${z'\in\overline{\{z\}}}$ is a partial order on~$Z$. Conversely, any partial order on a finite set $Z$ arises in this way from a unique $T_0$ topology on~$Z$.
Moreover, a map between finite $T_0$ spaces is continuous if and only if it preserves the associated partial orders.
\end{proposition}

Next consider a field $k$ with an algebraic closure $\kbar$ and the associated absolute Galois group $\Gamma := \Aut(\kbar/k)$. Let $X$ be a scheme  of finite type over~$k$, and let $H$ be a linear algebraic group over $k$ which acts on $X$ from the left by a morphism $H \times_k X \to X$. Then every $H(\kbar)$-orbit $\CO\subset X(\kbar)$ is locally closed for the Zariski topology on $X(\kbar)$. Moreover, its closure $\overline{\CO}$ is again $H(\kbar)$-invariant and therefore a union of orbits, and we have $\dim(\overline{\CO}\setminus\CO)<\dim(\CO)$. From this it follows that the set of orbits 
$$\Xi := H(\kbar)\backslash X(\kbar)$$
with the induced quotient topology is a $T_0$ space.

\medskip
Assume now that $\Xi$ is finite. Then by Proposition \ref{TopPartial1} the topology on $\Xi$ corresponds to the partial order $\preceq$ on $\Xi$ defined by $\CO'\preceq\CO$ $:\Leftrightarrow$ $\CO'\subset\overline{\CO}$.
Also, the Galois group $\Gamma$ acts on $\Xi$ preserving the topology and the partial order. Thus the quotient set $\Gamma\backslash\Xi$ again inherits a quotient topology which is $T_0$ and corresponds to a partial order described in the same fashion. The set $\Gamma\backslash\Xi$ is in natural bijection with the set of algebraic $H$-orbits in~$X$, that is, with the set of non-empty $H$-invariant locally closed reduced subschemes that do not possess non-empty $H$-invariant locally closed proper subschemes.


\subsection{Quotient stacks}
\label{Quot1}

This article will require a certain familiarity with the notion of stacks. Recall that a stack over a scheme $S$ is a category fibered in groupoids over the category $((\Sch/S))$ of schemes over $S$ which satisfies effective descent with respect to any fpqc morphism. The morphisms of stacks are functors, and so instead of equality of morphisms one often only has isomorphisms of functors. For the technical definition for a stack to be algebraic see Laumon-Moret-Bailly \cite{ChampsAlgebriques}, Definition 4.1. Let us recall only that every scheme can be considered as an algebraic stack, and for every algebraic stack $\CX$ there exists a smooth surjective morphism from a scheme $X\to\CX$. Many concepts and properties of schemes and morphisms of schemes have analogues for algebraic stacks. For example, there exist natural fiber products and pullbacks of algebraic stacks, and many properties of algebraic stacks and of morphisms of algebraic stacks are tested using pullbacks to schemes.

\medskip
Every algebraic stack $\CX$ possesses an underlying topological space $|\CX|$, defined in \cite{ChampsAlgebriques} Section~5. An element of $|\CX|$ is an equivalence class of morphisms $\Spec K\to\CX$ for fields~$K$, where two morphisms $\Spec K_1\to\CX$ and $\Spec K_2\to\CX$ are equivalent if and only if there exists a common field extension $K$ such that the composite morphisms $\Spec K\to\Spec K_i\to\CX$ are isomorphic. The open subsets of $|\CX|$ are the subsets $|\,\CU|$ for all open substacks $\CU\subset\CX$. If $\CX$ is represented by a scheme~$X$, then $|\CX|$ is homeomorphic to the topological space underlying~$X$.

\medskip
Consider now the situation of Subsection \ref{QuotTop}, where $H$ acts from the left hand side on a scheme $X$ over~$k$. The quotient stack $[H\backslash X]$ is then defined as follows. For any scheme $S$ over $k$ the category $[H\backslash X](S)$ has as objects the pairs consisting of a left $H$-torsor $T\to S$ and an $H$-equivariant morphism $f\colon T\to X$ over~$k$. A morphism $(T,f)\to(T',f')$ in $[H\backslash X](S)$ is a morphism $g\colon T\to T'$ of $H$-torsors over $S$ such that $f'\circ g=f$, and composition is defined in the obvious way. With the evident notion of pullback under morphisms $S'\to S$ this turns the whole collection of categories $[H\backslash X](S)$ into a stack over $k$ that is denoted $[H\backslash X]$. This is an algebraic stack by \cite{ChampsAlgebriques} Proposition 10.13.1, and it possesses a natural surjective morphism $X\to [H\backslash X]$.

\medskip
Assume that the set $\Xi$ of orbits over $\bar k$ is finite, so that it and its quotient  $\Gamma\backslash\Xi$ carry the natural topologies described in Subsection \ref{QuotTop}. As a reformulation of \cite{Wd_DimOort} (4.4) we then have:

\begin{proposition}\label{TopSpaceStack}
There is a natural homeomorphism $\Gamma\backslash\Xi \cong \bigl|[H \backslash X]\bigr|$. 
\end{proposition}


Now consider any algebraic $H$-orbit $Y \subset X$.
Then $Y$ is a locally closed reduced subscheme of~$X$, and so $[H\backslash Y]$ is a locally closed reduced substack of $[H\backslash X]$. Varying $Y$ we thus obtain a stratification 
\UseTheoremCounterForNextEquation
\begin{equation}\label{PointStack1}
[H\backslash X] \ \hbox{``$=$''}\ 
\bigsqcup_{\Gamma\backslash\Xi} \;[H\backslash Y]
\end{equation}
in the sense that for any scheme $S$ and any morphism $S\to[H\backslash X]$ we obtain a disjoint decomposition of $S$ into locally closed subschemes $S\times_{[H\backslash X]}[H\backslash Y]$.

\medskip
Moreover, assume that $k$ is perfect, and consider any point $y\in Y(\kbar)$. Then $Y_{\kbar}$ is again reduced and hence the disjoint union of the reduced $H_\kbar$-orbit $\CO(y)$ of $y$ and a (possibly empty) finite collection of $\Gamma$-conjugates thereof. Being a reduced orbit $\CO(y)$ is smooth over~$\kbar$; hence $Y$ is smooth over $k$, and so $[H\backslash Y]$ is a smooth algebraic stack over $k$ by definition (c.f. \cite{ChampsAlgebriques}, D\'efinition 4.14). Furthermore, as the smooth morphism $X \to [H \backslash X]$ preserves codimension, we have
\UseTheoremCounterForNextEquation
\begin{equation}\label{PointStack2}
\codim([H\backslash Y], [H \backslash X])
\ =\ \codim(Y,X) \ =\ \codim(\CO(y),X_{\kbar}).
\end{equation}


Also, the automorphisms of an object of a quotient stack can be described as follows. 
Consider a scheme $S$ over $k$, a point $x \in X(S)$, and let $\bar x \in [H \backslash X](S)$ denote the image of $x$ under the canonical morphism $X \to [H \backslash X]$. 
Denote by $\UAut(\bar x)$ the sheaf of groups on the category on schemes over $S$ that attaches to $S' \to S$ the group of automorphisms of the base change $\bar x_{S'}$ in the category $[H \backslash X](S')$. On the other hand let $\Stab_{H_S}(x)$ denote the closed subgroup scheme of $H_S := H \times_{\!k} S$ whose $S'$-valued points consist of those $h \in H(S')$ which satisfy $h\cdot x_{S'} = x_{S'}$. 

\begin{proposition}\label{AutoStack1}
There is a natural isomorphism $\UAut(\bar x) \cong \Stab_{H_S}(x)$.
\end{proposition}

\begin{proof}
In the construction of $[H\backslash X]$, the point $\bar x\in [H \backslash X](S)$ is represented by the trivial $H$-torsor $H_S\to S$ together with the morphism $H_S\to X$, $h\mapsto hx$. The automorphisms of the trivial left $H$-torsor $H_{S'}\to S'$ are precisely the morphisms $h\mapsto hg$ for all sections $g\in H(S')$. Thus $\UAut(\bar x)(S') = \Aut_{[H\backslash X](S')}(\bar x_{S'})$ is the group of automorphisms $h\mapsto hg$ of $H_{S'}\to S'$ with $g\in H(S')$, such that the two morphisms $H_{S'}\to X$ given by $h\mapsto hx$ and $h\mapsto hg\mapsto hgx$ coincide. But these conditions are equivalent to $g\in \Stab_{H_S}(x)(S')$.
\end{proof}




\section{$\Ghat$-zips}
\label{Gzips}


\subsection{General notation}
\label{notation1}

Let $\BF_q$ be a finite field of order~$q$, and let $k$ be a finite overfield of~$\BF_q$. By a linear algebraic group over $k$ we mean a reduced affine group scheme of finite type over~$k$. We do not generally assume it to be connected.
Throughout we denote a linear algebraic group over $k$ by~$\hat H$, its identity component by~$H$, and the finite \'etale group scheme of connected components by $\pi_0(\Hhat) := \Hhat/H$; and similarly for other letters of the alphabet. Note that the unipotent radical $R_uH$ of $H$ is a normal subgroup of~$\Hhat$. Any homomorphism of algebraic groups $\hat\phi\colon\Ghat\to\Hhat$ restricts to a homomorphism $\phi\colon G\to H$.

Let $S$ be a scheme over~$k$. By an $\Hhat$-torsor $I$ over $S$ we will mean a right $\Hhat$-torsor over $S$ for the fpqc-topology, unless mentioned otherwise. In other words $I$ is a scheme over $S$ together with a right action 
$I \times_{\!k} \Hhat \to I$ written $(i,h)\mapsto ih$, 
such that the morphism $I \times_{\!k} \Hhat \to I \times_S I$, $(i,h) \mapsto (i,ih)$ is an isomorphism and there exists an fpqc-covering $S' \to S$ such that $I(S') \ne \emptyset$. Any section in $I(S')$ then induces an isomorphism $\Hhat \times_{\!k} S' \stackrel{\sim}{\longto} I \times_S S'$ over~$S'$. By faithfully flat descent for $S' \to S$ we can therefore deduce that every $\Hhat$-torsor over $S$ is affine and faithfully flat over~$S$. Moreover, since $k$ is perfect, the reduced group scheme $\Hhat$ is automatically smooth, and hence $I$ is smooth over~$S$. Thus by \cite{EGAIV_4} (17.16.3) there already exists a surjective \'etale morphism $S' \to S$ such that $I(S') \ne \emptyset$.

Any scheme $S$ over $k$ possesses a natural $q^{\rm th}$ power Frobenius morphism $S \to S$, which is the identity on the underlying topological space and the map $x\mapsto x^q$ on the structure sheaf. The pullback of a scheme or a sheaf or a morphism over $S$ under this Frobenius morphism is denoted by $(\ )^{(q)}$. For example, the pullback of a linear algebraic group $\Hhat$ over $k$ is a linear algebraic group $\Hhat^{(q)}$ over~$k$, and the pullback of an $\Hhat$-torsor $I$ over $S$ is an $\Hhat^{(q)}$-torsor $I^{(q)}$ over $S$.


\subsection{The basic data}
\label{notation2}

Let $\Ghat$ be a linear algebraic group over $\BF_q$ such that $G$ is reductive, and let $\Ghat_k$ denote its base extension to~$k$. Let $\chi\colon\BG_{m,k}\to G_k$ be a cocharacter over~$k$, and let $L$ denote its centralizer in~$G_k$. There exist unique opposite parabolic subgroups $P_\pm = L\ltimes U_\pm \subset G_k$ with common Levi component $L$ and unipotent radicals~$U_\pm$, such that $\Lie U_+$ is the sum of the weight spaces of weights $>0$, and $\Lie U_-$ is the sum of the weight spaces of weights $<0$ in $\Lie G_k$ under $\Ad\circ\chi$. Note that the groups $L$ and $P_\pm$ and $U_\pm$ are all connected.

By definition we have $L = \Cent_{\Ghat_k}(\chi) \cap G_k$, and since $L$ is connected, we have a canonical inclusion $\pi_0(\Cent_{\Ghat_k}(\chi)) = \Cent_{\Ghat_k}(\chi)/L \into \pi_0(\Ghat_k)$. Let $\Theta$ be a subgroup scheme of $\pi_0(\Cent_{\Ghat_k}(\chi))$, and let $\Lhat$ denote its inverse image in $\Cent_{\Ghat_k}(\chi)$. Then $L$ is the identity component of $\Lhat$, and $\pi_0(\Lhat)=\Theta\subset\pi_0(\Ghat_k)$. Also, since $\chi$ is centralized by~$\Lhat$, and the subgroups $U_\pm$ depend only on~$\chi$, these subgroups are normalized by~$\Lhat$. Thus $\Phat_\pm := \Lhat\ltimes U_\pm$ are algebraic subgroups of $\Ghat_k$ with identity components $P_\pm$ and $\pi_0(\Phat_\pm)\cong\pi_0(\Lhat)=\Theta$.

This data will remain fixed throughout the article.

\medskip
For schemes $S$ over $k$ we are interested in (right) torsors over $S$ with respect to the above algebraic groups. Consider any $\Ghat_k$-torsor $I$ over~$S$. 
By a $\Phat_\pm$-torsor $I_\pm\subset I$ we mean a subscheme which is a $\Phat_\pm$-torsor with respect to the induced action of~$\Phat_\pm$. For any $\Phat_\pm$-torsor $I_\pm$ over~$S$, the quotient $I_\pm/U_\pm$ is a $\Phat_\pm/U_\pm$-torsor over~$S$, which we can view as an $\Lhat$-torsor under the canonical isomorphism $\Lhat \stackrel{\sim}{\to} \Phat_\pm/U_\pm$. 

On the other hand, the definition of $\Ghat_k$ as a base extension from $\BF_q$ induces a natural isomorphism $\Ghat_k^{(q)} \cong \Ghat_k$. Via this isomorphism we can consider $\chi^{(q)}$ again as a cocharacter of~$G_k$, with associated subgroups $\Phat_\pm^{(q)} = \Lhat^{(q)}\ltimes U_\pm^{(q)}$. 
Likewise, the $\Ghat_k^{(q)}$-torsor $I^{(q)}$ becomes a $\Ghat_k$-torsor in a natural way.
Moreover, the pullback of a $\Phat_\pm$-torsor $I_\pm\subset I$ is a $\Phat^{(q)}_\pm$-torsor $I_\pm^{(q)} \subset I^{(q)}$.


\subsection{The stack of $\Ghat$-zips}\label{Gzips1}

\begin{definition}\label{GZipDef}
Let $S$ be a scheme over~$k$.
\begin{itemize}
\item[(a)] A \emph{$\Ghat$-zip of type $(\chi,\Theta)$ over $S$} is a tuple $\UI = (I,I_+,I_-,\iota)$ consisting of a (right) $\Ghat_k$-torsor $I$ over~$S$, a $\Phat_+$-torsor $I_+\subset I$, a $\Phat^{(q)}_-$-torsor $I_-\subset I$, and an isomorphism of $\Lhat^{(q)}$-torsors $\iota\colon I^{(q)}_+/U^{(q)}_+ \stackrel{\sim}{\longto} I_-/U^{(q)}_-$.
\item[(b)] A \emph{morphism} $(I,I_+,I_-,\iota) \to (I',I_+',I_-',\iota')$ of $\Ghat$-zips of type $(\chi,\Theta)$ over $S$ consists of equivariant morphisms $I\to I'$ and $I_\pm\to I_\pm'$ that are compatible with the inclusions and the isomorphisms $\iota$ and~$\iota'$.
\item[(c)] The resulting \emph{category of $\Ghat$-zips of type $(\chi,\Theta)$ over $S$} is denoted $\GhatZip_k^{\chi,\Theta}(S)$. 
\end{itemize}
If $\Ghat$ is connected, we necessarily have $\Theta = 1$ and drop it from the notation, speaking simply of $\Ghat$-zips of type $\chi$ over $S$ and denoting their category by $\GhatZip_k^{\chi}(S)$.

With the evident notion of pullback the $\GhatZip_k^{\chi,\Theta}(S)$ form a fibered category over the category $(({\Sch}/k))$ of schemes over~$k$, which we denote $\GhatZip^{\chi,\Theta}_k$. 
\end{definition}

\begin{proposition}\label{GZipsStack}
$\GhatZip^{\chi,\Theta}_k$ is a stack.
\end{proposition}

\begin{proof}
Any morphism of $\Ghat$-zips is an isomorphism; hence $\GhatZip^{\chi,\Theta}_k$ is a category fibered in groupoids. As $\Ghat_k$ and $\Phat_\pm$ are affine over~$k$, the torsors $I$ and $I_\pm$ are affine over~$S$, and so the data in a $\Ghat$-zip satisfy effective descent with respect to any fpqc morphism $S'\to S$.
\end{proof}

\begin{remark}\label{WlogFinite}
For any finite field extension $k'/k$ the given data $\chi,\Theta,L,P_\pm,\ldots$ induces corresponding data $\chi_{k'},\Theta_{k'},\ldots$ over $k'$ by base change. The definition of $\Ghat$-zips then immediately implies that $\GhatZip^{\chi_{k'},\Theta_{k'}}_{k'}$ is just the pullback of $\GhatZip^{\chi,\Theta}_k$ under $\Spec k'\to\Spec k$.
One can use this to deduce properties of $\GhatZip^{\chi,\Theta}_k$ from the corresponding properties of $\GhatZip^{\chi_{k'},\Theta_{k'}}_{k'}$. In particular, one can apply this to a finite extension $k'/k$ for which $G_{k'}$ splits and $\pi_0(\Ghat_{k'})$ is a constant group scheme. Thus over $k'$ the added complexity induced by the Galois action in Subsections \ref{TopUnderZip} and \ref{StratZip} disappears.
%
\end{remark}


\subsection{Realization as a quotient stack}
\label{Gzips2}

\begin{construction}\label{GZipCons}
Let $S$ be a scheme over~$k$. To any section $g\in \Ghat(S)$ we associate a $\Ghat$-zip of type ${(\chi,\Theta)}$ over~$S$, as follows. Let $I_g\defeq S\times_{\!k}\Ghat_k$ and $I_{g,+}\defeq S\times_{\!k}\Phat_+\subset I_g$ be the trivial torsors. Then $I_g^{(q)} \cong S\times_{\!k}\Ghat_k = I_g$ canonically, and we define $I_{g,-}\subset I_g$ as the image of $S\times_{\!k} \Phat^{(q)}_- \subset S\times_{\!k} \Ghat_-$ under left multiplication by~$g$.
Then left multiplication by $g$ induces an isomorphism of $\Lhat^{(q)}$-torsors
$$\iota_g\colon I^{(q)}_{g,+}/U^{(q)}_+ = S\times_{\!k}\Phat^{(q)}_+/U^{(q)}_+ \cong S \times_{\!k} \Phat^{(q)}_-/U^{(q)}_- \stackrel{\sim}{\longto} g(S\times_{\!k} \Phat^{(q)}_-)/U^{(q)}_- = I_{g,-}/U^{(q)}_-.$$
We thus obtain a $\Ghat$-zip of type ${(\chi,\Theta)}$ over~$S$, which we denote by 
$$\UI_g \defeq (I_g,I_{g,+},I_{g,-},\iota_g).$$
\end{construction}

\begin{lemma}\label{GZipLocStandard}
Every $\Ghat$-zip of type ${(\chi,\Theta)}$ is \'etale locally isomorphic to one of the form~$\UI_g$.
\end{lemma}

\begin{proof}
Let $\UI = (I,I_+,I_-,\iota)$ be a $\Ghat$-zip of type ${(\chi,\Theta)}$ over~$S$. By Subsection~\ref{notation1}, after replacing $S$ by an \'etale covering there exist sections $i_\pm\in I_\pm(S)$. These sections induce two sections $i_-U^{(q)}_-$ and $\iota(i^{(q)}_+U^{(q)}_+)$ in $(I_-/U^{(q)}_-)(S)$; hence there exists a unique section $\ell\in \Lhat^{(q)}(S)$ such that $i_-U^{(q)}_-\cdot\ell = \iota(i^{(q)}_+U^{(q)}_+)$. After replacing $i_-$ by $i_-\ell$ we may therefore assume that the induced sections of $I_-/U^{(q)}_-$ coincide. Then $i_-$ and $i_+$ induce two sections of~$I$; hence there exists a unique $g\in \Ghat(S)$ such that $i_- = i_+g$. 

We claim that $\UI \cong \UI_g$. Indeed, using $i_+$ to trivialize $I_+$ and $I$, we may without loss of generality assume that $I_+ = I_{g,+} \subset I = I_g$ and that $i_+$ is the identity section. Then $i_-=i_+g$ corresponds to the section $g$ of~$I_g$. This implies that $I_- = i_-\Phat^{(q)}_- = I_{g,-}$. Furthermore, since the $\Lhat^{(q)}$-equivariant isomorphism $\iota\colon I^{(q)}_+/U^{(q)}_+ \stackrel{\sim}{\longto} I_-/U^{(q)}_-$ sends the section $i^{(q)}_+U^{(q)}_+ = U^{(q)}_+$ to the section $i_-U^{(q)}_- = g(S\times_{\!k}U^{(q)}_-)$, it must coincide with $\iota_g$. Thus we find that $\UI=\UI_g$ and are done.
\end{proof}

\begin{definition}\label{ZipDatumDef}
The \emph{algebraic zip datum} associated to $\Ghat$ and $(\chi,\Theta)$ is the tuple $\CZ_{\Ghat,\chi,\Theta} \defeq (\Ghat_k,\Phat_+,\Phat_-^{(q)},\hat\varphi)$ where $\hat\varphi$ is the composite isogeny
\vskip-7pt
$$\xymatrix@C-10pt{
\Phat_+/U_+ \ar@{}[r]|-\cong & 
\Lhat\ \ar[rr]^{\Frob_q} && 
\ \Lhat^{(q)} \ar@{}[r]|-\cong & 
\Phat_-^{(q)}/U_-^{(q)}\rlap{.} \\}$$
The associated \emph{zip group} is the linear algebraic group over $k$ 
\UseTheoremCounterForNextEquation
\begin{equation}\label{ZipGroupDef}
E_{\Ghat,\chi,\Theta} \defeq 
\set{(\ell u_+,\ell^{(q)}u_-)}{\ell\in \Lhat,\ u_+\in U_+, u_-\in U_-^{(q)}} 
\ \subset\ \Phat_+ \times_{\!k} \Phat_-^{(q)}.
\end{equation}
It acts from the left hand side on $\Ghat_k$ by the formula 
\UseTheoremCounterForNextEquation
\begin{equation}\label{ZipGroupActionDef}
(p_+,p_-)\cdot g \ \defeq\ p_+ g p_-^{-1}.
\end{equation}
If $\Ghat$ is connected and thus $\Theta = 1$, we abbreviate $\CZ_{\Ghat,\chi} := \CZ_{\Ghat,\chi,\Theta}$ and $E_{\Ghat,\chi} := E_{\Ghat,\chi,\Theta}$.
\end{definition}

\begin{remark}\label{ItsZipData}
In \cite{PWZ1}, Definition~10.1, we defined algebraic zip data over algebraically closed fields, whereas here $k$ is finite. But the natural base extension of the above tuple $\CZ_{\Ghat,\chi,\Theta}$ to an algebraic closure $\bar k$ of $k$ is an algebraic zip datum in the sense of [loc.~cit.], and the base extension of the above zip group $E_{\Ghat,\chi,\Theta}$ and its action on $\Ghat$ are those of [loc.~cit.], so all the results there have direct consequences here. For example, by \cite{PWZ1}, Proposition~7.3, the zip datum over~$\bar k$ is orbitally finite, and so the group $E_{\Ghat,\chi,\Theta}$ acts with only finitely many orbits on~$\Ghat_k$.
\end{remark}

\begin{lemma}\label{StandardGZipMorphisms}
For any two sections $g, g'\in \Ghat(S)$ there is a natural bijection between the transporter
$$\Transp_{E_{\Ghat,\chi,\Theta}(S)}(g,g')
\ :=\ \bigl\{(p_+,p_-)\in E_{\Ghat,\chi,\Theta}(S) \bigm|
p_+ g p_-^{-1} = g' \bigr\}$$
and the set of morphisms of $\Ghat$-zips $\UI_g\to \UI_{g'}$ over~$S$, under which $(p_+,p_-)$ corresponds to the morphisms $I_g \to I_{g'}$ and $I_{g,+}\to I_{g',+}$ given by left multiplication with $p_+$ and the morphism $I_{g,-}\to I_{g',-}$ given by left multiplication with $g'p_-g^{-1}$.
\end{lemma}

\begin{proof}
By definition a morphism $\UI_g\to\UI_{g'}$ consists of equivariant isomorphisms $f\colon I_g\to I_{g'}$ and $f_\pm\colon I_{g,\pm}\to I_{g',\pm}$ satisfying certain compatibilities, which we ana\-lyze in turn. First, since $I_{g,+} = S\times_{\!k}\Phat_+ = I_{g',+}$, the isomorphism $f_+$ must be left multiplication by a unique section $p_+\in \Phat_+(S)$. Next, since $I_g=S\times_{\!k} \Ghat = I_{g'}$, the compatibility with $f_+$ implies that $f$, too, is left multiplication by~$p_+$. 

On the other hand, since $I_{g,-} = g(S\times_{\!k} \Phat^{(q)}_-)$ and $I_{g',-} = g'(S\times_{\!k}\Phat^{(q)}_-)$ within $S\times_{\!k} \Ghat_-$, the isomorphism $f_-$ must be left multiplication by $g'p_-g^{-1}$ for a unique section $p_-\in \Phat^{(q)}_-(S)$. This isomorphism must be compatible with the isomorphism $f\colon I_g\stackrel{\sim}{\to}I_{g'}$, which is left multiplication by~$p_+$. The compatibility thus amounts to the equation $g'p_-g^{-1} = p_+$.

The last compatibility is the commutativity of the diagram of isomorphisms
\vskip-8pt
$$\xymatrix@C+20pt{
I^{(q)}_{g,+}/U^{(q)}_+ \ar[r]^{p^{(q)}_+} \ar[d]_{g} & 
I^{(q)}_{g',+}/U^{(q)}_+ \ar[d]_{g\smash{'}} \\
I_{g,-}/U^{(q)}_- \ar[r]^{g'p_-g^{-1}} &
I_{g',-}/U^{(q)}_- \rlap{,}\\}$$
where each arrow is defined as left multiplication by the indicated element. This amounts to the equation $p^{(q)}_+U^{(q)}_+ = p_-U^{(q)}_-$ in $\Phat^{(q)}_+/U^{(q)}_+ \cong \Lhat^{(q)} \cong \Phat^{(q)}_-/U^{(q)}_-$. That in turn is equivalent to $p_+ = \ell u_+$ and $p_- = \ell^{(q)} u_-$ with $\ell\in \Lhat(S)$, $u_+\in U_+(S)$, and $u_- \in U^{(q)}_-(S)$, or in other words to $(p_+,p_-)\in E_{\Ghat,\chi,\Theta}(S)$. 

Combined with the earlier relation $g' = p_+g p_-^{-1}$ this means that $(p_+,p_-)$ lies in the transporter $\Transp_{E_{\Ghat,\chi,\Theta}(S)}(g,g')$. Thus the map in the lemma defines a bijection between this transporter and the set of morphisms $\UI_g\to\UI_{g'}$, as desired.
\end{proof}

\begin{proposition}\label{GZipStack}
The stack $\GhatZip^{\chi,\Theta}_k$ of $\Ghat$-zips of type ${(\chi,\Theta)}$ is isomorphic to the algebraic quotient stack $[E_{\Ghat,\chi,\Theta}\backslash \Ghat_k]$. In particular, the isomorphism classes of $\Ghat$-zips of type $(\chi,\Theta)$ over any algebraically closed field $K$ containing $k$ are in bijection with the $E_{\Ghat,\chi,\Theta}(K)$-orbits on $\Ghat(K)$.
\end{proposition}

\begin{proof}
Consider the category $\CX$ fibered in groupoids over $(({\Sch}/k))$ defined as follows: For any scheme $S$ over $k$ the class of objects of $\CX(S)$ is the set $\Ghat(S)$, and for any elements $g,g'\in \Ghat(S)$ the set of morphisms from $g$ to $g'$ is the transporter $\Transp_{E_{\Ghat,\chi,\Theta}}(g,g')$, with composition given by the multiplication in $E_{\Ghat,\chi,\Theta}$. For any morphism $S'\to S$ of schemes over~$k$, the pullback of objects and morphisms is given by the canonical maps $\Ghat(S)\to \Ghat(S')$ and $E_{\Ghat,\chi,\Theta}(S) \to E_{\Ghat,\chi,\Theta}(S')$. Since $E_{\Ghat,\chi,\Theta}$ is a scheme, this is a prestack, i.e., it satisfies effective descent for morphisms. By \cite{ChampsAlgebriques} 3.4.3, the stackification (for this notion see \cite{ChampsAlgebriques} 3.2) of this prestack is the quotient stack $[E_{\Ghat,\chi,\Theta}\backslash \Ghat_k]$. 

As can be verified directly from its description, the bijection in Lemma~\ref{StandardGZipMorphisms} is compatible with pullback and composition and sends $1\in \Transp_{E_{\Ghat,\chi,\Theta}}(g,g)$ to the identity morphism $\id\colon \UI_g\to \UI_g$ for all $g\in \Ghat(S)$. Thus there is a fully faithful morphism $\CX\to \GhatZip^{\chi,\Theta}_k$ which sends $g\in \CX(S) = \Ghat(S)$ to $\UI_g$ and which acts on morphisms by the bijection of Lemma~\ref{StandardGZipMorphisms}. Lemma~\ref{GZipLocStandard} is then equivalent to saying that this morphism induces an isomorphism from the stackification of $\CX$ to $\GhatZip^{\chi,\Theta}_k$. Since the former is $[E_{\Ghat,\chi,\Theta}\backslash \Ghat_k]$, the proposition follows.
\end{proof}

\begin{corollary}\label{SmoothStack}
$\GhatZip^{\chi,\Theta}_k$ is a smooth algebraic stack of dimension $0$ over $k$.
\end{corollary}

\begin{proof}
The quotient stack $[E_{\Ghat,\chi,\Theta}\backslash \Ghat_k]$ it is algebraic by \cite{ChampsAlgebriques} Proposition 10.13.1, and 
the canonical morphism $\Ghat_k \to [E_{\Ghat,\chi,\Theta}\backslash \Ghat_k]$ is a torsor over the group scheme $E_{\Ghat,\chi,\Theta}$. As $\Ghat_k$ and $E_{\Ghat,\chi,\Theta}$ are smooth of the same dimension, this quotient stack is smooth of dimension $0$ over~$k$. The corollary thus follows from Proposition \ref{GZipStack}.
\end{proof}


\subsection{The topological space underlying $\GhatZip^{\chi,\Theta}_k$}
\label{TopUnderZip}

We recall some notation and facts from \cite{PWZ1}, especially from Sections 2.2, 6, and~10.

Choose an algebraic closure $\kbar$ of $k$ and let $\Gamma := \Gal(\kbar/k)$ be the corresponding Galois group of~$k$. Let $T\subset B\subset G_\kbar$ be a maximal torus, respectively a Borel subgroup of~$G_\kbar$. Consider the finite groups
\begin{eqnarray*}
     W &\defeq& \Norm_{ G   (\kbar)}(T(\kbar))/T(\kbar), \\
\hat W &\defeq& \Norm_{\Ghat(\kbar)}(T(\kbar))/T(\kbar), \\
\Omega &\defeq& \bigl(\Norm_{\Ghat(\kbar)}(T(\kbar)) \cap \Norm_{\Ghat(\kbar)}(B(\kbar))\bigr)\!\bigm/\!T(\kbar).
\end{eqnarray*}
The fact that $W$ acts simply transitively on the set of Borel subgroups containing $T_\kbar$ implies that $\hat W = W \rtimes \Omega$, and the fact that $G(\kbar)$ acts transitively on the set of all maximal tori of $G_\kbar$ implies that $\Omega \cong \hat W/W \cong \pi_0(\hat G)(\kbar)$.
Also, let 
$S\subset W$ be the set of simple reflections associated to the pair $(T,B)$. As this pair is unique up to conjugation by $G(\kbar)$, and $\Norm_{G(\kbar)}(T(\kbar)) \cap \Norm_{G(\kbar)}(B(\kbar)) = T(\kbar)$, the Coxeter system $(W,S)$ and the groups $\hat W$ and $\Omega$ are, up to unique isomorphism, independent of the choice of $T$ and~$B$. 
The inner automorphism of $\hat W$ induced by an element $x\in\hat W$ will be denoted by ${\rm int}(x)\colon \hat w \mapsto \leftexp{x}{\hat w} := x \hat w x^{-1}$.

Recall that the length of an element $w\in W$ is the smallest number $\ell(w)$ such that $w$ can be written as a product of $\ell(w)$ simple reflections. For any subsets $K,K' \subseteq S$, we denote by $W_K$ the subgroup of $W$ generated by $K$ and by $\leftexp{K}{W}$ (resp.\ $W^{K'}$, resp.\ $\doubleexp{K}{W}{K'}$) the set of $w \in W$ that are of minimal length in the left coset $W_Kw$ (resp.\ in the right coset $wW_{K'}$, resp.\ in the double coset $W_KwW_{K'}$).
We let $w_0 \in W$ denote the unique element of maximal length in $W$, and $w_{0,K}$ the unique element of maximal length in~$W_K$.

The Frobenius isogeny $\hat\varphi\colon \Ghat \to \Ghat$ relative to $\BF_q$ induces an automorphism $\bar\varphi$ of $\hat W$ which preserves $W$ and~$\Omega$. Its restriction to $W$ is an automorphism of Coxeter systems $(W,S) \iso (W,S)$. Therefore $\bar\varphi$ preserves the length of elements in $W$ and in particular satisfies $\bar\varphi(w_0) = w_0$. 


Let $I \subseteq S$ be the type of the parabolic subgroup~$P_+$, and $J\subseteq S$ the type of~$P_-^{(q)}$. The fact that $P_-$ is opposite to $P_+$ implies that $J = \bar\varphi(\leftexp{w_0}{I}) = \leftexp{w_0}{\bar\varphi(I)}$. We write these equations in the form $J = \bar\varphi(\leftexp{y}{I}) = \leftexp{x}{\bar\varphi(I)}$, where $x \in \doubleexp{J}{W}{\bar\varphi(I)}$ is the unique element of minimal length in $W_Jw_0W_{\bar\varphi(I)}$ and  $y \defeq \bar\varphi^{-1}(x)$. 
Then $\hat\psi \defeq {\rm int}(x) \circ \bar\varphi = \bar\varphi \circ {\rm int}(y)$ is an automorphism of $\hat W$ which induces an isomorphism of Coxeter systems
$$(W_I,I) \iso (W_J,J).$$
From \cite{PWZ1} Proposition 2.7 we can deduce that 
\UseTheoremCounterForNextEquation
\begin{equation}\label{w0w0formula}
y = w_0w_{0,I} = w_{0,\bar\varphi^{-1}(J)}w_0.
\end{equation}

Via the isomorphism $\pi_0(\hat G)(\kbar)\cong\Omega$ we can view $\Theta(\kbar)$ (resp. $\Theta^{(q)}(\kbar)$) as a subgroup of~$\Omega$, which by abuse of notation we will again denote by~$\Theta$ (resp. $\Theta^{(q)}$).
Note that, since $\Theta \subseteq \Norm_{\Ghat}(P_+)/P_+$, conjugation by elements of $\Theta$ preserves the type $I$ of $P_+$ and thus the subgroup~$W_I$. Therefore $W_I\Theta = W_I\rtimes\Theta$ is a subgroup of~$\hat W$. 
Since $\Theta^{(q)} \subseteq \Norm_{\Ghat}(P^{(q)}_-)/P^{(q)}_-$, the same observation holds for $W_J\Theta^{(q)} = W_J\rtimes\Theta^{(q)}$, and the automorphism $\hat\psi$ sends the subgroup $\Theta\subseteq\Omega$ to $\Theta^{(q)}\subseteq\Omega$.
By \cite{PWZ1} Lemma~10.4, the map
\UseTheoremCounterForNextEquation
\begin{equation}\label{ThetaAction}
(\theta, \hat w) \mapsto \theta{\hat w}\hat\psi(\theta)^{-1}
\end{equation}
defines a left action of $\Theta$ on the subset $\leftexp{I}{W}\Omega \subseteq\hat W$.

Recall that the Bruhat order $\leq$ on $W$ is defined by $w'\leq w$ if for some (and equivalently for any) expression of $w$ as a product of $\ell(w)$ simple reflections, by leaving out certain factors one can obtain an expression of $w'$ as a product of $\ell(w')$ simple reflections. We extend the Bruhat order to $\hat W$ by setting
\UseTheoremCounterForNextEquation
\begin{equation}\label{ExtendBruhat1}
\hbox{$w'\omega' \leq w\omega$\ \ if and only if\ \  $w' \leq w$ and $\omega' = \omega$}
\end{equation}
for any $w,w' \in W$ and $\omega,\omega' \in \Omega$.
Also, for any $\hat w, \hat w' \in \leftexp{I}{W}\Omega$ we write
\UseTheoremCounterForNextEquation
\begin{equation}\label{DefPartialOrder}
\hbox{$\hat w' \preceq \hat w$ \ if and only if there exists $\hat v \in W_I\Theta$ with $\hat v \hat w'\hat\psi(\hat v)^{-1} \leq \hat w$.}
\end{equation}
By \cite{PWZ1} Theorem~10.9, see also \cite{He:GStablePieces}, this defines a partial order on $\leftexp{I}{W}\Omega$.
We also extend the length function from $W$ to $\What$ by setting
\UseTheoremCounterForNextEquation
\begin{equation}\label{ExtendLength}
\ell(w\omega) := \ell(w)
\end{equation}
for any $w\in W$ and $\omega\in\Omega$.

\begin{lemma}\label{ThetaPartialOrder}
The action~\eqref{ThetaAction} preserves the extended Bruhat order $\leq$ on $\hat W$, the partial order $\preceq$ on $\leftexp{I}{W}\Omega$, and the extended length function $\ell$ on $\hat W$.
\end{lemma}

\begin{proof}
Consider any elements $\theta\in\Theta$ and $w'$, $w \in W$ and $\omega'$, $\omega \in \Omega$. First assume that $w'\omega' \leq w\omega$, in other words, that $w' \leq w$ and $\omega' = \omega$. Then $\theta w' \theta^{-1} \leq \theta w \theta^{-1}$ and $\theta \omega' \hat\psi(\theta)^{-1} = \theta \omega \hat\psi(\theta)^{-1}$, and the latter is again an element of~$\Omega$, because $\Theta^{(q)}=\hat\psi(\Theta)\subset\Omega$. By (\ref{ExtendBruhat1}) we therefore find that
\[
\theta w'\omega'\hat\psi(\theta)^{-1} 
= \theta w' \theta^{-1}\cdot\theta \omega' \hat\psi(\theta)^{-1} 
\leq \theta w \theta^{-1}\cdot\theta \omega \hat\psi(\theta)^{-1} 
= \theta w\omega \hat\psi(\theta)^{-1}.
\]
Thus the action \eqref{ThetaAction} preserves the extended Bruhat order $\leq$ on~$\hat W$.

Next, the last equality above and the fact that $\theta \omega \hat\psi(\theta)^{-1}\in\Omega$ also imply that the length of $\theta w\omega \hat\psi(\theta)^{-1}$ is equal to that of $\theta w \theta^{-1}$. Since $\theta\in\Omega$, that length is equal to the length of $w$ and hence of $w\omega$, proving that the action \eqref{ThetaAction} preserves the extended length function on~$\hat W$.

Now assume that $w$ ,$w' \in \leftexp{I}{W}$ and $w'\omega' \preceq w\omega$, which means that $\hat v w'\omega' \hat\psi(\hat v)^{-1} \leq w\omega$ for some  $\hat v \in W_I\Theta$. Then we have just shown that
\[
\theta\hat v\theta^{-1}\cdot\theta w'\omega' \hat\psi(\theta)^{-1} 
\cdot \hat\psi(\theta\hat v\theta^{-1})^{-1} 
= \theta\hat v w'\omega' \hat\psi(\hat v)^{-1} \hat\psi(\theta)^{-1} 
\leq \theta w\omega \hat\psi(\theta)^{-1}.
\]
Since $\theta$ normalizes~$W_I$, it follows that $\hat u := \theta\hat v\theta^{-1}$ is an element of $W_I\Theta$ which satisfies $\hat u\cdot\theta w'\omega' \hat\psi(\theta)^{-1} \cdot \hat\psi(\hat u)^{-1} \leq \theta w\omega \hat\psi(\theta)^{-1}$ and thus by (\ref{DefPartialOrder}) shows that $\theta w'\omega' \hat\psi(\theta)^{-1} \preceq \theta w\omega \hat\psi(\theta)^{-1}$. Therefore the action~\eqref{ThetaAction} preserves the partial order $\preceq$, and we are done.
\end{proof}

As a consequence of Lemma \ref{ThetaPartialOrder}, the partial order $\preceq$ from (\ref{DefPartialOrder}) induces a partial order on the set of $\Theta$-orbits
\UseTheoremCounterForNextEquation
\begin{equation}\label{DefTopSpace}
\Xi^{\chi,\Theta} \defeq \Theta\backslash\leftexp{I}{W}\Omega.
\end{equation}
By Proposition \ref{TopPartial1} this in turn defines a $T_0$ topology on the finite set $\Xi^{\chi,\Theta}$. 

\medskip
Now observe that since the subgroups $P_\pm\subset G_k$ and $\Phat_\pm\subset \Ghat_k$ are defined over~$k$, there is a natural continuous action of $\Gamma := \Gal(\kbar/k)$ on everything discussed above. In particular this action preserves the decomposition $\hat W = W \rtimes \Omega$, the subsets $S,I,J,\leftexp{I}{W},\ldots$, the partial orders $\leq$ and $\preceq$, the length function~$\ell$, the subgroup $\Theta$ and its action on $\leftexp{I}{W}\Omega$, and so it induces an action on the topological space $\Xi^{\chi,\Theta}$. 

\begin{theorem}\label{TopGZip1}
The topological space underlying $\GhatZip^{\chi,\Theta}_k$ is naturally homeomorphic to the quotient space $\Gamma\backslash\Xi^{\chi,\Theta}$. 
\end{theorem}

\begin{proof}
By Proposition~\ref{GZipStack} the stack $\GhatZip^{\chi,\Theta}_k$ is isomorphic to $[E_{\Ghat,\chi,\Theta}\backslash \Ghat_k]$. By \cite{PWZ1}, Proposition~7.3 (see also Remark~\ref{ItsZipData}) the zip datum $\CZ_{\Ghat,\chi,\Theta,\kbar}$ is orbitally finite, i.e., the number of $E_{\Ghat,\chi,\Theta}(\kbar)$-orbits in $\Ghat(\kbar)$ is finite. We can therefore apply Proposition~\ref{TopSpaceStack}. The description of the topological space now follows from the description of $E_{\Ghat,\chi,\Theta,\kbar}$-orbits in $\Ghat_\kbar$ and their closures from \cite{PWZ1} Theorems~10.9 and~10.10.
\end{proof}

\begin{remark}\label{StandardGZip}
If we replace $k$ by a suitable finite extension $k'$ within $\kbar$,
the Galois group $\Gamma$ is replaced by a subgroup which acts trivially on $\hat W$ and everything else above. Then Theorem~\ref{TopGZip1} asserts that the topological space underlying $\GhatZip^{\chi,\Theta}_{k'}$ is naturally homeomorphic to $\Xi^{\chi,\Theta}$. In particular, for any algebraically closed extension field $K$ of~$\kbar$ we obtain a natural bijection 
\UseTheoremCounterForNextEquation
\begin{equation}\label{DescribeGZips1}
\Xi^{\chi,\Theta} \ \stackrel{\sim}{\longrightarrow}\ 
\left\{\begin{matrix} \text{isomorphism classes of $\Ghat$-zips}\\
\text{of type $(\chi,\Theta)$ over $K$} \end{matrix}\right\}.
\end{equation}
By \cite{PWZ1} this can be made more explicit, as follows. As the choice of $(T,B)$ was arbitary, we may without loss of generality assume that $T\subset L_K$ and $B\subset P_{-,K}$. Then we may identify $W = \Norm_G(T)(K)/T(K)$ and $\hat W = \Norm_{\Ghat}(T)(K)/T(K)$. Choose a representative $g \in \Norm_{G}(T)(K)$ of the element $y = \bar\varphi^{-1}(x) \in W$. Then by \cite{PWZ1} Lemma~12.11 the triple $(B,T,g)$ is a frame of the connected zip datum $(G_K,P_{+,K},P_{-,K}^{(q)},\varphi\colon L_K \to L_K^{(q)})$ in the sense of \cite{PWZ1} Definition~3.6.
Also, for every element $\hat w \in \leftexp{I}{W}\Omega$ choose a representative $\dothatw\in\Norm_{\Ghat}(T)(K)$, and let $\UI_{g\dothatw}$ denote the $\Ghat$-zip of type $(\chi,\Theta)$ over $K$ attached to $g\dothatw \in \Ghat(K)$ by Construction~\ref{GZipCons}. Combining the isomorphism in Proposition \ref{GZipStack} with \cite{PWZ1} Theorem~10.10 then shows that the bijection (\ref{DescribeGZips1}) sends the orbit of $\hat w$ in $\Xi^{\chi,\Theta} = \Theta\backslash\leftexp{I}{W}\Omega$ to the isomorphism class of $\UI_{g\dothatw}$.
\end{remark}

\begin{example}\label{SplitExample}
Assume that $\Ghat = G$ is a connected split reductive group over~$\BF_q$. Then $\Theta = \Omega = 1$, and $\Xi^{\chi,\Theta} = \leftexp{I}{W}$ with the trivial action of $\Gal(\kbar/\BF_q)$. All the formulas then simplify accordingly. In particular, by Theorem \ref{TopGZip1} the topological space underlying $\GZip^\chi_k$ is naturally homeomorphic to $\leftexp{I}{W}$ for every finite extension $k$ of $\BF_q$.

Moreover, the automorphism $\bar\varphi$ induced by Frobenius is the identity on~$W$. Thus $\hat\psi = {\rm int}(x)$, where $x \in \doubleexp{J}{W}{I}$ is the unique element of minimal length in $W_Jw_0W_I$. The partial order $\preceq$ on $\leftexp{I}{W}$ is therefore given by 
\UseTheoremCounterForNextEquation
\begin{equation}\label{DefPartialOrderSplit1}
\hbox{$w' \preceq w$ \ if and only if there exists $v \in W_I$ with $v w'xv^{-1}x^{-1} \leq w$.}
\end{equation}
If in addition the Dynkin diagram of $G$ has no component of type $A_n$ with $n \geq 2$, of type $D_n$ with $n \geq 5$ odd, or of type~$E_6$, then $w_0$ is central in $W$, and so $I = J$. Also, by (\ref{w0w0formula}) we then have $x = w_0w_{0,I}$, and since $w_0$ is central and $w_{0,I}^{-1}=w_{0,I}$, the partial order can then be written equivalently in the form
\UseTheoremCounterForNextEquation
\begin{equation}\label{DefPartialOrderSplit2}
\hbox{$w' \preceq w$ \ if and only if there exists $v \in W_I$ with $vw'w_{0,I}v^{-1}w_{0,I} \leq w$.}
\end{equation}
\end{example}


\subsection{The stratification of $\GhatZip^{\chi,\Theta}_k$}
\label{StratZip}

For any orbit in $\Gamma\backslash\Xi^{\chi,\Theta} = \Gamma\backslash(\Theta\backslash\leftexp{I}{W}\Omega)$ represented by an element $\hat w\in \leftexp{I}{W}\Omega$, let $[\hat w]$ denote the corresponding point in the topological space underlying $\GhatZip^{\chi,\Theta}_k$ via the homeomorphism in Theorem \ref{TopGZip1}.

\begin{theorem}\label{TopGZip2}
The point $[\hat w]$ underlies a smooth locally closed substack of $\GhatZip^{\chi,\Theta}_k$ of pure codimension $\dim(G/P_+) - \ell(\hat w)$, where $\ell(\ )$ denotes the extended length function from \eqref{ExtendLength}.
\end{theorem}

\begin{proof}
As in the proof of Theorem \ref{TopGZip1} this translates into an assertion for the quotient stack $[E_{\Ghat,\chi,\Theta}\backslash \Ghat_k]$. Let $\Ghat_\kbar^{\hat w}$ denote the $E_{\Ghat,\chi,\Theta,\bar k}$-orbit in $\Ghat_\kbar$ corresponding to $\hat w$ by \cite{PWZ1} Theorem~10.10. Since $k$ is perfect, by the remarks following Proposition~\ref{TopSpaceStack} this determines a smooth locally closed substack of $[E_{\Ghat,\chi,\Theta}\backslash \Ghat_k]$ with underlying point~$[\hat w]$. By (\ref{PointStack2}) the codimension of this substack is equal to the codimension of $\Ghat_\kbar^{\hat w}$ in $\Ghat_\kbar$, which by \cite{PWZ1} Theorem 5.11 and Lemma 10.3 is given by the desired formula.
\end{proof}


Let $S$ be a scheme over~$k$, and let $\UI$ be a $\Ghat$-zip of type $(\chi,\Theta)$ over $S$. Then $\UI$ defines a classifying morphism
\UseTheoremCounterForNextEquation
\begin{equation}\label{ClassifyingMorphism}
\zeta\colon S \longrightarrow \GhatZip^{\chi,\Theta}_k.
\end{equation}
Let $S_\UI^{\hat w}$ denote the pullback under $\zeta$ of the substack corresponding to $[\hat w]$. This is a locally closed subscheme of~$S$. As $[\hat w]$ varies, these subschemes form a finite stratification of~$S$, in other words $S$ is the set-theoretic disjoint union
\UseTheoremCounterForNextEquation
\begin{equation}\label{GZipStrat}
S \ = \bigsqcup_{[\hat w] \in \Gamma\backslash\Xi^{\chi,\Theta}} S_\UI^{[\hat w]}.
\end{equation}
%
%
%
The description of the topology in Theorem \ref{TopGZip1} implies that for any $\hat w$ we have
\UseTheoremCounterForNextEquation
\begin{equation}\label{GZipStratClos}
\overline{S_\UI^{[\hat w]}} \ \subset 
\bigsqcup_{\substack{[\hat w'] \in \Gamma\backslash\Xi^{\chi,\Theta}\\ \hat w' \preceq \hat w}} S_\UI^{[\hat w]}
\end{equation}
For the next result recall that any open or flat morphism of schemes is generizing.

\begin{proposition}\label{Generizing1}
If the morphism $\zeta$ in \eqref{ClassifyingMorphism} is generizing,
the inclusion (\ref{GZipStratClos}) is an equality.
If in addition $S$ is locally noetherian, then $S_\UI^{[\hat w]}$ is of pure of codimension $\dim(G/P_+) - \ell(\hat w)$. If $\zeta$ is smooth, then $S_\UI^{[\hat w]}$ is smooth as a scheme over~$k$.
\end{proposition}

\begin{proof}
If $\zeta$ is generizing, then $\zeta^{-1}(\overline{\Upsilon}) = \overline{\zeta^{-1}(\Upsilon)}$ for any locally closed substack $\Upsilon$ of $\GhatZip^{\chi,\Theta}_k$, so the first assertion follows from Theorem~\ref{TopGZip1}. 
If in addition $S$ is locally noetherian, the codimension is well-defined and preserved by~$\zeta$; so the second assertion follows from Theorem~\ref{TopGZip2}. If $\zeta$ is smooth, then $S_\UI^{[\hat w]}$, being the pullback of a smooth stack under a smooth morphism, is smooth as a scheme over~$k$, proving the third assertion.
\end{proof}

Instead of the above construction, the subscheme $S_\UI^{[\hat w]}$ can also be characterized by a construction directly involving~$\UI$. For simplicity we discuss this only in a special case
(but compare Remark \ref{WlogFinite}):

\begin{proposition}\label{DescribeStrata}
Assume that $G$ splits over $k$ and that $\pi_0(\Ghat_k)$ is a constant group scheme. Then a morphism of schemes $f\colon S'\to S$ factors through $S_\UI^{[\hat w]}$ if and only if $f^*\UI$ is locally for the fppf-topology on $S'$ isomorphic to the constant $G$-zip $\UI_{g\dothatw} \times_{\!k} S$ with $\UI_{g\dothatw}$ as in Remark~\ref{StandardGZip}.
\end{proposition}

\begin{proof}
The assumptions imply that the $E_{\Ghat,\chi,\Theta}$-orbit used in the proof of Theorem \ref{TopGZip2} is really defined over $k$; let us denote it by $\Ghat^{\hat w}_k$. Define $S''$ and $g'$ by the cartesian diagram
$$\xymatrix{
S'' \ar@{->>}[d] \ar[rr]^-{g'} && \Ghat_k \ar@{->>}[d]\\
S' \ar[r]^f & S \ar[r]^-\zeta & [E_{\Ghat,\chi,\Theta}\backslash \Ghat_k] 
\rlap{$\ \cong\ \GhatZip^{\chi,\Theta}_k$} \\}$$
where the vertical morphisms are fppf. Then by the definition of the quotient stack and the construction of $S_\UI^{[\hat w]}$, the morphism $f$ factors through $S_\UI^{[\hat w]}$ if and only if $g'$ factors through $\Ghat^{\hat w}_k$. As the orbit $\Ghat^{\hat w}_k$ is smooth, the morphism $E_{\Ghat,\chi,\Theta} \to \Ghat^{\hat w}_k$, $e\mapsto e\cdot g\dothatw$ is fppf. Thus $g'$ factors through $\Ghat^{\hat w}_k$ if and only if there exists an fppf-covering $S'''\to S''$ and an $e\colon S'''\to E_{\Ghat,\chi,\Theta}$ such that $g' = e\cdot g\dothatw$. By Lemma \ref{StandardGZipMorphisms} the latter condition is equivalent to saying that the $\Ghat$-zip $\UI_{g'}$ is fppf-locally isomorphic to $\UI_{g\dothatw}$, or again that $f^*\UI$ is fppf-locally isomorphic to $\UI_{g\dothatw}$, as desired.
\end{proof}

\begin{remark}\label{Purity1}
In~\cite{Yatsyshyn} Yatsyshyn and the second author show that all $E_{\Ghat,\chi,\Theta}$-orbits in $\Ghat_k$ are affine. This implies that the inclusion into $\GhatZip^{\chi,\Theta}_k$ of the substack associated to $[\hat w]$ is an affine morphism, and so the inclusion $S_\UI^{[\hat w]} \into S$ is an affine morphism. In particular this implies the following purity result:
\end{remark}

\begin{proposition}\label{Purity2}
Let $S$ be a locally noetherian scheme over~$k$, and let $Z$ be a closed subscheme of codimension $\geq 2$. Assume that $Z$ contains no embedded component of~$S$ (which is automatic if $S$ is reduced). Let $\UI$ be a $\Ghat$-zip 
over~$S$ whose restriction  to $S \setminus Z$ is fppf-locally constant. Then $\UI$ is fppf-locally constant.
\end{proposition}

\begin{proof}
By Proposition~\ref{DescribeStrata} there exists $[\hat{w}]$ such that the open immersion ${S \setminus Z \into S}$ factors through the subscheme $S_\UI^{[\hat w]}$. By assumption $S \setminus Z$ and hence $S_\UI^{[\hat w]}$ is schematically dense in~$S$; being locally closed $S_\UI^{[\hat w]}$ is therefore an open subscheme of~$S$. On the other hand its complement $Z'$ is of codimension $\geq 2$. Since the inclusion $S^{[\hat{w}]}_I \into S$ is affine, this implies that $Z' = \emptyset$.
\end{proof}


\subsection{Automorphisms of $G$-zips}
\label{AutoGzipSec}

Let $K$ be an algebraically closed extension field of~$\kbar$, and let $\UI$ be a $\Ghat$-zip of type $(\chi,\Theta)$ over~$K$. Let $T,B,g,\dothatw$ be as in Remark~\ref{StandardGZip}. 
Then $\UI$ is isomorphic to $\UI_{g\dothatw}$ for some $\hat w \in \leftexp{I}{W}\Omega$. Its automorphism group scheme is therefore $\UAut(\UI) \cong \UAut(\UI_{g\dothatw})$. By Proposition \ref{AutoStack1} the latter is isomorphic to the stabilizer $\Stab_{E_{\Ghat,\chi,\Theta,K}}(g\dothatw)$.

Since the results on stabilizers in \cite{PWZ1} were formulated only for connected zip data, we now assume that $\Ghat = G$ is connected. Then $\Theta = \Omega = 1$, and we can write $\hat w=w\in \leftexp{I}{W}$ and $\dothatw=\dot w\in \Norm_G(T)(K)$. As in \cite{PWZ1}, Subsection 5.1, let $H_w$ be the Levi subgroup of $G_K$ containing $T$ whose set of simple reflections is the unique largest subset $K_w$ of $J \cap \leftexp{w^{-1}}{I}$ such that $(\intaut(x) \circ \bar\varphi \circ \intaut(w))(K_w) = K_w$.

\begin{proposition}\label{AutoGZip}
\begin{itemize}
\item[(a)] The identity component of $\UAut(\UI_{g\dot{w}})$ is a unipotent group scheme of dimension $\dim(G/P_+) - \ell(w)$.
\item[(b)] Let $v$ be the unique element of minimal length in the double coset $W_IwW_J$. Then the Lie algebra of $\UAut(\UI_{g\dot{w}})$ has dimension $\dim(G/P_+) - \ell(v)$.
\item[(c)] The group of connected components of $\UAut(\UI_{g\dot{w}})$ is isomorphic to the constant group scheme over $K$ associated to the finite group
$$\Pi := \set{h \in H_w(\kbar)}{h = \varphi(g\dot{w}h(g\dot{w})^{-1})}.$$
\end{itemize}
\end{proposition}

\begin{proof}
By \cite{PWZ1} Theorem~8.1 the group $A := \UAut(\UI_{g\dot w}) \cong \Stab_{E_{G,\chi,\Theta,K}}(g\dot w)$ is isomorphic to a semi-direct product of the group $\Pi$ in (c) with a connected unipotent group scheme~$U$. As the zip datum is orbitally finite, the group $\Pi$ is finite by \cite{PWZ1}~Proposition~7.1. This shows~(c) and that the identity component of $A$ is unipotent. Moreover, the orbit $o(g\dot{w}) \subset G_K$ of $g\dot{w}$ has dimension $\dim P_+ + \ell(w)$ by \cite{PWZ1}~Theorem~7.5. As the definition of $E_{G,\chi,\Theta}$ implies that $\dim E_{G,\chi,\Theta}=\dim G$, it follows that 
$$\dim A\ =\ 
\dim E_{G,\chi,\Theta,K} - \dim o(g\dot{w}) = \dim G - \dim(P_+) - \ell(w),$$
proving~(a). Assertion~(b) follows from \cite{PWZ1}~Theorem~8.5.
\end{proof}

\begin{remark}\label{AutoGZip1}
Since a group scheme is smooth if and only if its dimension is equal to the dimension of its Lie algebra, Proposition \ref{AutoGZip} (a)~and~(b) imply that $\UAut(\UI_{g\dot w})$ is smooth if and only if $w$ is of minimal length in its double coset $W_IwW_J$. This condition will often not be satisfied.
\end{remark}



\section{Generalities on filtrations}
\label{Prelim}

In this section we briefly review some standard definitions and notations for gradings and filtrations of locally free sheaves of finite rank.


\subsection{Locally free sheaves of finite rank}
\label{PreLocFree}

Let $S$ be a scheme over a ring~$k$. The category of locally free sheaves of $\CO_S$-modules of finite rank on $S$ with all $\CO_S$-linear homomorphisms between them is denoted $\LocFree(S)$. It is a $k$-linear additive category, but in general not abelian. A homomorphism in $\LocFree(S)$ is called \emph{admissible} 
if its image in the category of sheaves of finite rank is a locally direct summand.
This notion turns $\LocFree(S)$ into an exact category in the sense of Quillen. It is also idempotent complete, i.e., any endomorphism $f\colon\CM\to\CM$ in $\LocFree(S)$ satisfying $f^2=f$ is admissible and corresponds to a direct sum decomposition $\CM = \Ker(f) \oplus \Image(f)$ within $\LocFree(S)$.

For a useful overview of exact categories see B\"uhler \cite{Buehler}. Every admissible homo\-mor\-phism in an exact category has a kernel and a cokernel, and they satisfy a number of axioms: see \cite{Buehler}. An additive functor between exact categories is exact if it sends admissible homomorphisms to admissible homomorphisms and commutes with their kernels and cokernels.

Endowed with the usual tensor product of sheaves of finite rank $\CM\otimes\CN$, the usual dual $\CM^\vee \defeq \CHom(\CM,\CO_S)$, and the usual associativity and commutativity constraints $\LocFree(S)$ is a rigid tensor category in the sense of Saavedra \cite{Saavedra} I.5.1. Moreover, the tensor product and the dual define exact functors in the indicated sense.


\subsection{Gradings}
\label{PreGrad}

By a \emph{graded locally free sheaf of finite rank on $S$} we mean a locally free sheaf of finite rank $\CM$ together with a decomposition $\CM = \bigoplus_{i\in\BZ}\CM^i$, whose \emph{graded pieces} $\CM^i$ vanish for almost all~$i$. A \emph{homomorphism} of graded locally free sheaves of finite rank $f\colon\CM\to\CN$ is a homomorphism of the underlying sheaves that satisfies $f(\CM^i)\subset \CN^i$ for all $i\in\BZ$. The category of graded locally free sheaves of finite rank on $S$ is denoted $\GradLocFree(S)$. It is a $k$-linear additive category, but in general not abelian. 

A homomorphism in $\GradLocFree(S)$ is \emph{admissible} if it is admissible in each degree. This turns $\GradLocFree(S)$ into an exact category that is idempotent complete.

The \emph{tensor product} of graded locally free sheaves of finite rank is the usual tensor product of sheaves with the grading $(\CM\otimes\CN)^i \defeq \bigoplus_{j\in\BZ} \; \CM^j \otimes \CN^{i-j}$. The \emph{dual} of a graded locally free sheaf of finite rank $\CM$ is the usual dual sheaf with the grading $(\CM^\vee)^i \defeq (\CM^{-i})^\vee$. These notions turn $\GradLocFree(S)$ into a rigid tensor category together with a (forgetful) exact tensor functor $\forget\colon\GradLocFree(S) \to \LocFree(S)$ sending graded locally free sheaves to their underlying locally free sheaves.


\subsection{Descending filtrations}
\label{PreDescFilt}

By a \emph{descending filtration} $C^\bullet$ of a locally free sheaf of finite rank $\CM$ on~$S$ we mean a family of quasi-coherent subsheaves $C^i\CM$ for $i\in\BZ$, which are locally direct summands and satisfy $C^{i+1}\CM\subset C^i\CM$ for all~$i$ and $C^i\CM=0$ for all $i\gg0$ and $C^i\CM=\CM$ for all $i\ll0$. A homomorphism of sheaves of finite rank $f\colon\CM\to\CN$ endowed with a descending filtration is \emph{compatible with the filtrations} if it satisfies $f(C^i\CM)\subset C^i\CN$ for all $i\in\BZ$. This defines a category of locally free sheaves of finite rank on $S$ endowed with a descending filtration, which we denote $\DescFilLocFree(S)$. It is a $k$-linear additive category, but in general not abelian. It possesses an evident forgetful functor $\forget\colon\DescFilLocFree(S) \to \LocFree(S)$.

The assumptions imply that the subquotients $\gr_C^i\CM := C^i\CM/C^{i+1}\CM$ are again locally free sheaves of finite rank on~$S$ which vanish for almost all~$i$. Also, any homomorphism $f\colon\CM\to\CN$ in $\DescFilLocFree(S)$ induces natural homomorphisms $\gr^i_Cf\colon \gr_C^i\CM\to\gr_C^i\CN$. Together this defines a natural functor $\gr_C^\bullet\colon \DescFilLocFree(S) \to \GradLocFree(S)$.

Reciprocally, any graded locally free sheaf of finite rank $\CM$ carries a natural descending filtration $C^i\CM \defeq \bigoplus_{j\ge i}\CM^j$, which defines a natural functor $\descfil\colon\GradLocFree(S) \to \DescFilLocFree(S)$.

A homomorphism $f\colon\CM\to\CN$ in $\DescFilLocFree(S)$ is called \emph{admissible} if for all $i$ the sheaf $f(C^i\CM)$ is equal to $f(\CM)\cap C^i\CN$ and a locally direct summand of~$\CN$. This is equivalent to saying that locally on $S$, the morphism possesses a factorization of the form $\CM\cong\CM'\oplus\CL\onto\CL\into\CL\oplus\CN'\cong\CN$ in the category of filtered locally free sheaves of finite rank. With this notion $\DescFilLocFree(S)$ is an exact category that is idempotent complete.

Descending filtrations of $\CM$ and $\CN$ induce a natural descending filtration of $\CM\otimes\CN$ by the formula $C^i(\CM\otimes\CN) \defeq \sum_{j\in\BZ} C^j\CM \otimes C^{i-j}\CN$. The graded subquotients inherit natural isomorphisms $\gr_C^i(\CM\otimes\CN) \ \cong\ \bigoplus_{j\in\BZ} \; \gr_C^j\CM \otimes \gr_C^{i-j}\CN$. Moreover, a descending filtration of $\CM$ induces a descending filtration of $\CM^\vee$ by the formula $C^i(\CM^\vee) \defeq (\CM/C^{1-i}\CM)^\vee$, and the graded subquotients possess natural isomorphisms $\gr_C^i(\CM^\vee) \ \cong\ (\gr_C^{-i}\CM)^\vee$. These notions turn $\DescFilLocFree(S)$ into a rigid tensor category, such that all three functors above are tensor functors. These functors, as well as tensor product and dual, are also exact.


\subsection{Ascending filtrations}
\label{PreAscFilt}

An \emph{ascending filtration} $D_\bullet$ of $\CM$ is a family of subsheaves $D_i\CM$ such that the $D_{-i}\CM$ form a descending filtration of~$\CM$. Thus everything in Subsection \ref{PreDescFilt} has a direct analogue for ascending filtrations. Descending filtrations are generally indexed by upper indices, ascending filtrations by lower indices, while gradings can be indexed in both fashions. In particular the graded subquotients of an ascending filtration are denoted $\gr^D_i\CM := D_i\CM/D_{i-1}\CM$. The category of locally free sheaves of $\CO_S$-modules with an ascending filtration is denoted $\AscFilLocFree(S)$. There are natural exact tensor functors $\gr_\bullet^D\colon \AscFilLocFree(S)\to \GradLocFree(S)$ and $\ascfil\colon\GradLocFree(S)\to\AscFilLocFree(S)$ and $\forget\colon \AscFilLocFree(S)\to \LocFree(S)$.
The functors introduced so far are summarized in the following diagram:
\UseTheoremCounterForNextEquation
\begin{equation}
\label{FilSummaryDiagram}
\vcenter{\xymatrix{
    & \DescFilLocFree(S) \ar@<-.5ex>[ld]_{\gr^\bullet_C} \ar[rd]^\forget & \\
    \GradLocFree(S) \ar@<-.5ex>[ru]_\descfil \ar[rr]^\forget \ar@<.5ex>[rd]^\ascfil &&\LocFree(S)\\
    & \AscFilLocFree(S) \ar@<.5ex>[lu]^{\gr_\bullet^D} \ar[ru]^\forget &
  }}
\end{equation}


\subsection{Types}
\label{PreTypes}

Let $\underline{n} = (n_i)_{i\in\BZ}$ be a family of non-negative integers which vanish for almost all~$i$. We say that a graded locally free sheaf of finite rank $\CM$ is \emph{of type $\underline{n}$} if each $\CM^i$ is locally free of constant rank~$n_i$. We call a locally free sheaf of finite rank endowed with a descending or ascending filtration of type $\underline{n}$ if its associated graded sheaf is of type~$\underline{n}$. In all these cases, the sheaf itself is then locally free of constant rank $\sum_in_i$. If $S$ is connected (and hence non-empty!), every graded or filtered locally free sheaf of finite rank on $S$ possesses a unique type.


\subsection{Pullback}
\label{FilPull}

All the above notions possess evident pullbacks under a morphism $S'\to S$, which are compatible with all the given constructions. We generally denote the pullback of $f\colon\CM\to\CN$ by $f_{S'}\colon\CM_{S'}\to\CN_{S'}$. This defines an exact tensor functor $\DescFilLocFree(S) \to \DescFilLocFree(S')$ and similar functors on the other categories.

Many properties and invariants such as the rank of a locally free sheaf of finite rank are local for the fpqc topology. In particular:

\begin{lemma}
  \label{LocallyAdmissible1}
For a homomorphism of graded, filtered, or naked locally free sheaves of finite rank, the property of being admissible is local for the fpqc topology.
\end{lemma}

\begin{proof}
The subsheaf $f(\CM)\subset\CN$ is a locally direct summand if and only if the quotient $\CN/f(\CM)$ is locally free. Since the latter property is local for the fpqc topology, so is the former, and the lemma follows for naked and graded locally free sheaves of finite rank. For filtered ones observe that the formation of $f(C^i\CM)$ and $f(\CM)\cap C^i\CN$ commutes with flat pullback and their equality is local for the fpqc topology. By the same argument as before their being a locally direct summand is local for the fpqc topology, too, and so the lemma follows in the filtered case.
\end{proof}


\subsection{Alternating and symmetric powers}
\label{Powers}

Consider an object $X$ of an exact additive tensor category $\CC$. For any integer $m\ge0$  let $X^{\otimes m}$ denote the tensor product of $m$ copies of $X$ with itself, which carries a natural action of the symmetric group $S_m$. Thus there is a homomorphism
\UseTheoremCounterForNextEquation
\begin{equation}
\label{AltmXMorph}
A_m(X)\colon\ 
X^{\otimes m} \longto X^{\otimes m}, \quad
x\mapsto \sum_{\sigma\in S_m} \sgn(\sigma)\cdot\sigma(x).
\end{equation}
If this homomorphism is admissible, its image $\Alt^mX \defeq \Image A_m(X)$ is called the \emph{$m$-th alternating, or exterior, power of~$X$}. Likewise, there is a homomorphism
\UseTheoremCounterForNextEquation
\begin{equation}
\label{SymmXMorph}
B_m(X)\colon\ 
\bigoplus_{\sigma\in S_m} X^{\otimes m} \longto X^{\otimes m}, \quad
(x_\sigma)_\sigma \mapsto \sum_{\sigma\in S_m}(\sigma-1)(x_\sigma).
\end{equation}
If this homomorphism is admissible, its cokernel $\Sym^mX \defeq \Coker B_m(X)$ is called the \emph{$m$-th symmetric power of~$X$}.

For any morphism $f\colon X\to Y$ in $\CC$ the above constructions are compatible with the induced morphism $f^{\otimes m}\colon X^{\otimes m} \to Y^{\otimes m}$; hence they induce morphisms $\Alt^mf\colon \allowbreak \Alt^mX \to \Alt^mY$ and $\Sym^mf\colon \allowbreak \Sym^mX \to \Sym^mY$ whenever the respective powers exist. Evidently this sends the identity on $X$ to the identity on $\Alt^mX$ and $\Sym^mX$ and commutes with composition, i.e., it is functorial in~$X$.

As a direct consequence of this construction, alternating and symmetric powers commute with any exact tensor functor between exact additive tensor categories $F\colon\CC\to\CD$. More precisely, if $\Alt^mX$ exists, then $\Alt^mF(X)$ exists and is canonically isomorphic to $F(\Alt^mX)$, and similarly for $\Sym^m$.

In each of the categories $\LocFree(S)$, $\GradLocFree(S)$, $\DescFilLocFree(S)$, and $\AscFilLocFree(S)$ above, all alternating and symmetric powers exist and have the usual local descriptions, essentially because every object is Zariski locally on $S$ a direct sum of objects of rank~$1$. Also, the $m$-th alternating power of an object of constant rank $n$ has constant rank $\binom{n}{m}$, and the $m$-th symmetric power of an object of constant rank $n$ has constant rank $\binom{n+m-1}{m}$. In particular, the $n$-th exterior power of an object $X$ of constant rank $n$ is an object of constant rank $1$, also called the \emph{highest exterior power of~$X$}.



\section{Filtered fiber functors and cocharacters}
\label{FFFAC}


Let $\Ghat$ be a (not necessarily connected) linear algebraic group over an arbitrary field~$k_0$, and let $\GhRep$ denote the tensor category of finite dimensional representations of $\Ghat$ over~$k_0$. 
As usual, by a \emph{fiber functor} over a scheme $S$ over $k_0$ we mean an exact $k_0$-linear tensor functor $\GhRep\to\LocFree(S)$. Similarly, by a \emph{graded fiber functor} we mean an exact $k_0$-linear tensor functor $\GhRep\to\GradLocFree(S)$, and by a \emph{filtered fiber functor} an exact $k_0$-linear tensor functor $\GhRep\to\DescFilLocFree(S)$. 
In this section we collect some results from \cite{Saavedra} and \cite{ZieglerFFF} on graded and filtered fiber functors. We consider only descending filtrations; the corresponding results for ascending filtrations follow directly by renumbering.

\medskip
For any graded fiber functor $\gamma\colon \GhRep\to\GradLocFree(S)$ and any morphism $S'\to S$ we let $\gamma_{S'}\colon \GhRep\to\GradLocFree(S')$ denote the graded fiber functor obtained by pullback. We call two graded fiber functors $\gamma_1$, $\gamma_2\colon \GhRep\to\GradLocFree(S)$ fpqc-locally isomorphic if their pullbacks under some fpqc morphism $S'\to S$ are isomorphic. In general we let $\UIsom^\otimes(\gamma_1,\gamma_2)$ denote the fpqc-sheaf on $(({\Sch}/S))$ sending $S'\to\nobreak S$ to the set of isomorphisms $\gamma_{1,S'}\stackrel{\sim}{\to}\gamma_{2,S'}$. By composition of isomorphisms it carries a natural right action of the sheaf of groups $\UAut^\otimes(\gamma_1) \defeq \UIsom^\otimes(\gamma_1,\gamma_1)$.
The same notation will be used for filtered fiber functors $\psi$, $\psi_1$, $\psi_2\colon \GhRep \to \DescFilLocFree(S)$.

\medskip
We first consider the special case that $S=\Spec k$ for an overfield $k$ of~$k_0$. Since a locally free sheaf of finite rank on $\Spec k$ is just a finite dimensional $k$-vector space, we abbreviate 
\begin{eqnarray*}
\Vec(k)        &\!\! \defeq & \LocFree(\Spec k), \\
\GradVec(k)    &\!\! \defeq & \GradLocFree(\Spec k), \\
\DescFilVec(k) &\!\! \defeq & \DescFilLocFree(\Spec k).
\end{eqnarray*}
Let  $\omega_{0,k}\colon \GhRep\to \Vec(k)$ denote the tautological fiber functor that sends each representation $V$ to the vector space $V_k := V\otimes_{k_0}k$.

Consider a cocharacter $\chi\colon\BG_{m,k}\to \Ghat_{k}$. Let $\hat L$ denote its centralizer in~$\hat G_{k}$, let $U$ denote the unique connected smooth unipotent subgroup of $\Ghat_k$ that is normalized by $\hat L$ and whose Lie algebra is the sum of the weight spaces of weights $>0$ under $\Ad\circ\chi$, and set $\hat P \defeq \hat L\ltimes U$. 
(If the identity component of $\Ghat$ is reductive, the identity component of $\hat P$ is a parabolic subgroup $P$ and the identity component of~$\Lhat$ is a Levi subgroup of $P$.) 

For any representation $V \in\GhRep$, the cocharacter $\chi$ determines a grading $V_k = \bigoplus_{i\in\BZ}V^i_{k}$. This grading is $k_0$-linearly functorial in $V$, exact in short exact sequences, 
and compatible with tensor product, and the same holds for the associated descending filtration. Thus $\chi$ induces a graded fiber functor $\gamma_\chi$
and a filtered fiber functor $\descfil\circ\gamma_\chi$ such that the composite
$$\xymatrix{
\GhRep \ar[r]^-{\gamma_\chi} & \GradVec(k)
\ar[r]^-{\descfil} & \DescFilVec(k)
\ar[r]^-{\forget} & \Vec(k)}$$
is equal to $\omega_{0,k}$.

\begin{proposition} \label{AutOmegaChiIsP}
\begin{enumerate}
\item[(a)] The action of $\hat L$ on $\gamma_\chi$ induces a natural isomorphism
  \begin{equation*}
    \hat L\isoto \UAut^\otimes(\gamma_\chi).
  \end{equation*}
\item[(b)] The action of $\hat P$ on $\descfil\!\circ\!\gamma_\chi$ induces a natural isomorphism
\begin{equation*}
  \hat P\isoto \UAut^\otimes(\descfil\!\circ\!\gamma_\chi). 
\end{equation*}
\end{enumerate}
\end{proposition}

\begin{proof}
Part (a) is \cite[Corollary 3.7]{ZieglerFFF}. 
Part (b) is a consequence of \cite[2.1.4 and 2.1.5]{Saavedra}.
\end{proof}


Now let $\bar k_0$ be an algebraic closure of $k_0$, and let $k_0^s$ denote the separable closure of $k_0$ in~$\bar k_0$.  Let $\CC_{\hat G}$ denote the set of $\hat G(\bar k_0)$-conjugacy classes of cocharacters $\BG_{m,\bar k_0}\to\nobreak \hat G_{\bar k_0}$. The Galois group $\Gal(k_0^s/k_0)\cong \Aut(\bar k_0/k_0)$ acts naturally on $\CC_{\hat G}$. For any $c\in \CC_{\hat G}$ we let $k_c\subset k_0^s$ denote the fixed field of the stabilizer of $c$ in $\Gal(k_0^s/k_0)$. The fact that $c$ contains a cocharacter which is defined over a finite separable extension of $k_0$ implies that $k_c$ is finite separable over $k_0$.

\begin{definition} 
We call $k_c$ the \emph{field of definition} of the conjugacy class $c$.
\end{definition}

Next observe that conjugate cocharacters $\chi$, $\chi'$ give rise to isomorphic functors $\gamma_\chi$,~$\gamma_\chi'$,
so the following definition depends only on the conjugacy class of~$\chi$.

\begin{definition}
\label{DefTypeChi}
Let $c\in \CC_{\hat G}$ and let $S$ be a scheme over $k_c$. A graded fiber functor $\gamma \colon\GhRep\to\GradLocFree(S)$ is called \emph{of type $c$}, or \emph{of type $\chi$} for any $\chi\in c$, if the pullbacks of the functors $\gamma$ and $\gamma_\chi$ to $S\times_{k_c} \bar k_0$ are fpqc-locally isomorphic.
A filtered fiber functor $\psi \colon\GhRep\to\DescFilLocFree(S)$ is called \emph{of type $c$} if the associated graded fiber functor $\descgr \circ\psi\colon\GRep\to\GradLocFree(S)$ is of type~$c$.
\end{definition}


\begin{theorem}[{\cite[Theorem 3.25]{ZieglerFFF}}]
\label{CocharacterConstant1}
Let $S$ be a connected scheme over~$k_0$, and let $\gamma$ be a graded fiber functor $\GhRep\to\GradLocFree(S)$. Then there exist a unique $c\in \CC_{\hat G}$ and a unique morphism $S\to \Spec k_c$ over~$k_0$, such that $\gamma$ is of type~$c$. The same assertion holds for any filtered fiber  functor $\psi\colon \GhRep\to\DescFilLocFree(S)$. 
\end{theorem}


In general, a conjugacy class $c\in \CC_{\hat G}$ does not have a representative which is defined over $k_c$. For the following results, we therefore fix a field extension $k_c \subset k\subset \bar k_0$ and a representative $\chi\in c$ that is defined over~$k$. Let $\hat L$, $U$, and $\hat P = \hat L\ltimes U$ be the associated subgroups of~$\hat G_{k}$. Let $S$ be a scheme over~$k$.

\begin{theorem}[{\cite[Theorem 3.27]{ZieglerFFF}}] \label{GradTorsor1a} 
 There is a natural equivalence of categories from the category of graded fiber functors $\GhRep\to\GradLocFree(S)$ of type $c$ to the category of right $\hat L$-torsors over $S$, given by
$$\gamma \mapsto \UIsom^\otimes(\gamma_{\chi,S},\gamma).$$
\end{theorem}

\begin{theorem}[{\cite[Theorem 4.43]{ZieglerFFF}}] \label{FilIsomTorsor1a}
There is a natural equivalence of categories from the category of filtered fiber functors $\GhRep\to\DescFilLocFree(S)$ of type $c$ to the category of right $\hat P$-torsors over $S$, given by
$$\psi \mapsto \UIsom^\otimes(\descfil\!\circ\gamma_{\chi,S},\psi).$$
\end{theorem}

\begin{theorem}[{\cite[Theorem 4.39]{ZieglerFFF}}]
\label{TorsorQuotient}
For any filtered fiber functor $\psi \colon\GhRep\to\DescFilLocFree(S)$ of type $c$, the functor $\gr^\bullet_C$ induces a natural isomorphism of right $\hat L$-torsors
$$\UIsom^\otimes(\descfil\!\circ\gamma_{\chi,S},\psi)\!\bigm/\!U
\ \cong\ \UIsom^\otimes(\gamma_{\chi,S},\descgr \circ\psi).$$
\end{theorem}

\section{$F$-zips}
\label{Fzips}



\begin{definition}\label{FZipDef}
\begin{itemize}
\item[(a)] An \emph{$F$-zip over $S$} is a tuple $\UCM = (\CM,C^\bullet,D_\bullet,\phi_\bullet)$ consisting of a locally free sheaf of $\CO_S$-modules of finite rank $\CM$ on~$S$, a descending filtration $C^\bullet$ and an ascending filtration $D_\bullet$ of~$\CM$, and an $\CO_S$-linear isomorphism $\phi_i\colon (\gr_C^i\CM)^{(q)} \stackrel{\sim}{\to} \gr^D_i\CM$ for every $i\in\BZ$.
\item[(b)] A \emph{homomorphism} $f\colon\UCM\to\UCN$ of $F$-zips over $S$ is a homomorphism of the underlying sheaves of $\CO_S$-modules $\CM\to\CN$ which for all $i\in\BZ$ satisfies $f(C^i\CM)\subset C^i\CN$ and $f(D_i\CM)\subset D_i\CN$ and makes the following diagram commute:
$$\xymatrix{
(\gr_C^i\CM)^{(q)} \ar[r]_-\sim^-{\phi_i} \ar[d]_{(\gr^i_Cf)^{(q)}} & \gr^D_i\CM \ar[d]^{\gr^D_if} \\
(\gr_C^i\CN)^{(q)} \ar[r]_-\sim^-{\phi_i} & \gr^D_i\CN \rlap{.} \\
}$$
\item[(c)] The resulting \emph{category of $F$-zips over $S$} is denoted $\FZip(S)$.
\end{itemize}
\end{definition}

The category $\FZip(S)$ is additive, but due to the presence of the Frobenius pullback it is only $\BF_q$-linear in general, not $\CO_S$-linear. Easy examples show that it is not abelian if $S\not=\varnothing$.

\begin{definition}\label{FZipAdmis}
A homomorphism of $F$-zips $\UCM\to\UCN$ is \emph{admissible} if the underlying morphisms of filtered locally free sheaves $\CM \to \CN$ for both filtrations~$C^\bullet$ and~$D_\bullet$ is admissible.
\end{definition}

This notion turns $\FZip(S)$ into an exact category that is idempotent complete. 


\begin{definition}\label{FZipTypeDef}
An $F$-zip is called \emph{of rank $n$}, or \emph{of height $n$}, if its underlying sheaf of $\CO_S$-modules is of constant rank~$n$. 
Let $\underline{n} = (n_i)_{i\in\BZ}$ be a family of non-negative integers which vanish for almost all~$i$. An $F$-zip $\UCM$ is \emph{of type $\underline{n}$} if $\gr_C^i\CM$, or equivalently $\gr^D_i\CM$, is locally free of constant rank $n_i$ for all~$i$.
\end{definition}

Any $F$-zip of type $\underline{n}$ is of rank $\sum_in_i$. If its rank is $1$ there exists an integer $d$ such that $n_d = 1$ and $n_i = 0$ for $i \ne d$. In this case we say briefly that the $F$-zip is of type $d$. If $S$ is connected, every $F$-zip over $S$ possesses a unique type.


\begin{definition}\label{FZipTensorDef}
The \emph{tensor product} of $F$-zips $\UCM$ and $\UCN$ over $S$ is the $F$-zip $\UCM\otimes\UCN$ consisting of the tensor product $\CM\otimes\CN$ with the induced descending filtration $C^\bullet$ and the induced ascending filtration $D_\bullet$ of $\CM\otimes\CN$ and the induced isomorphisms
$$\xymatrix@C-10pt{
(\gr_C^i(\CM\otimes\CN))^{(q)} \ar@{=}[r]^-\sim \ar[d]_\cong &
\bigoplus_{j\in\BZ} \; (\gr_C^j\CM)^{(q)} \otimes (\gr_C^{i-j}\CN)^{(q)}
\ar[d]_\cong^{\bigoplus_j (\phi_j\otimes\phi_{i-j})} \\
\gr^D_i(\CM\otimes\CN) \ar@{=}[r]^-\sim &
\bigoplus_{j\in\BZ} \; \gr^D_j\CM \otimes \gr^D_{i-j}\CN \rlap{.} \\}$$
There is also a straightforward definition of tensor product of morphisms of $F$-zips, we leave it to the reader to verify that this is a homomorphism of $F$-zips. The tensor product thus defines a functor $\FZip(S) \times \FZip(S) \to \FZip(S)$, which is $\BF_q$-bilinear and exact. 
\end{definition}

Comparing Definition \ref{FZipAdmis} with the construction in Subsection \ref{Powers} we find that all symmetric powers $\Sym^m\UCM$ and all alternating powers $\Alt^m\UCM$ of $F$-zips exist. They have evident descriptions in terms of the symmetric and alternating powers of the underlying filtered and graded locally free sheaves since symmetric and alternating powers of filtered and graded locally free sheaves are compatible with pullbacks under Frobenius and the functor $\gr$.

\begin{definition}\label{FZipDualDef}
The \emph{dual} of an $F$-zip $\UCM$ over $S$ is the $F$-zip $\UCM^\vee$ consisting of the dual sheaf of $\CO_S$-modules $\CM^\vee$ equipped with the duals of the filtrations $C^\bullet(\CM)$ and $D_\bullet(\CM)$ whose terms are given by $\gr_C^i(\CM^\vee)=(\gr_C^{-i}\CM)^\vee$ and $\gr^D_i(\CM^\vee)=(\gr^D_{-i}\CM)^\vee$, and the induced isomorphisms
$$\xymatrix{
(\gr_C^i(\CM^\vee))^{(q)} \ar@{=}[r] \ar[d]_\cong &
((\gr_C^{-i}\CM)^\vee)^{(q)} \ar[d]_\cong^{(\phi_{-i}^{-1})^\vee} \\
\gr^D_i(\CM^\vee)  \ar@{=}[r]^-\sim &
(\gr^D_{-i}\CM)^\vee \rlap{.} \\}$$
There is an evident notion of the dual of a homomorphism of $F$-zips, so that we obtain a functor $\FZip(S)^\opp \to \FZip(S)$, which is $\BF_q$-linear and exact. 
\end{definition}

As usual the tensor product and the dual yields the notion of an internal Hom of two $F$-zips $\UCM$ and $\UCN$ over $S$ by setting
\[
\UHom(\UCM,\UCN) := \UCM^\vee \otimes \UCN.
\]


\begin{example}\label{FZipTate}
The \emph{Tate $F$-zip of weight $d\in\BZ$} is $\UBOne(d) \defeq (\Oscr_S, C\updot, D\dodot, \varphi\dodot)$, where 
$$C^i =
\begin{cases}
\Oscr_S &\hbox{for}\ i \leq d;\\
0&\hbox{for}\ i > d;
\end{cases}
\qquad
D_i =
\begin{cases}
0&\hbox{for}\ i < d;\\
\Oscr_S&\hbox{for}\  i \geq d;
\end{cases}$$
and $\varphi_d$ is the identity on $\Oscr^{(q)}_S = \Oscr_S$. Thus $\UBOne(d)$ is an $F$-zip of rank $1$ and type $d$.
There are natural isomorphisms $\UBOne(d)\otimes\UBOne(d')\cong\UBOne(d+d')$ and $\UBOne(d)^\vee\cong\UBOne(-d)$.
The \emph{$d$-th Tate twist} of an $F$-zip $\UCM$ is defined as $\UCM(d) := \UCM\otimes\UBOne(d)$, and there is a natural isomorphism $\UCM(0)\cong\UCM$.
\end{example}

With the above tensor product and dual and the unit object $\UBOne(0)$ the category $\FZip(S)$ is a rigid tensor category. It is endowed with the following natural exact $\BF_q$-linear tensor functors:
\begin{eqnarray*}
\forget \colon&\!\! \FZip(S)\to\rlap{$\LocFree(S),$}\phantom{\DescFilLocFree(S),}& 
\UCM\mapsto \CM, \\
\descfil\colon&\!\! \FZip(S)\to\DescFilLocFree(S),& \UCM\mapsto (\CM,C^\bullet), \\
\ascfil \colon&\!\! \FZip(S)\to\AscFilLocFree(S),&  \UCM\mapsto (\CM,D_\bullet).
\end{eqnarray*}
The isomorphism $\phi_\bullet\colon(\gr_C^\bullet\CM)^{(q)}\stackrel{\sim}{\to}\gr^D_\bullet\CM$ that is part of an $F$-zip induces an isomorphism of tensor functors $\phi \circ (\ )^{(q)} \colon \descgr\circ \descfil\to \ascgr\circ \ascfil$. Combined with some of the functors from (\ref{FilSummaryDiagram}) we obtain the following diagram
\UseTheoremCounterForNextEquation
\begin{equation}
\label{FunctorSummaryDiagram}
\vcenter{\xymatrix{
  & \DescFilLocFree(S) \ar[dl]^-\descgr \ar@/^/[drrr]^\forget \\
    \GradLocFree(S) & \phi \circ (\ )^{(q)}\circlearrowright 
  & \FZip(S) \ar[dl]_-\ascfil \ar[ul]^-\descfil \ar[rr]^\forget && \LocFree(S) \\
  & \AscFilLocFree(S) \ar[ul]_-\ascgr \ar@/_/[urrr]_-\forget \rlap{,}
}}
\end{equation}
which is commutative except that the left hand side commutes only up to~$\phi \circ (\ )^{(q)}$.


\medskip

Also, there is an evident notion of pullback of $F$-zips under morphisms $S'\to S$, compatible with everything discussed above. Lemma \ref{LocallyAdmissible1} directly implies:

\begin{lemma}\label{LocallyAdmissible2}
For a homomorphism of $F$-zips, the property of being admissible is local for the fpqc topology.
\end{lemma}



\section{$\hat G$-zip functors}
\label{Functors}


Throughout this section we fix a (not necessarily connected) linear algebraic group $\hat G$ over $\BF_q$. Let $S$ be a scheme over~$\BF_q$. It is known (for example by \cite[Prop.~2.9]{Nori1976}) that giving an exact $\BF_q$-linear tensor functor $\GhRep\to\LocFree(S)$ is equivalent to giving a $\Ghat$-torsor over~$S$. This suggests the idea that an exact $\BF_q$-linear tensor functor from $\GhRep$ to an arbitrary exact $\BF_q$-linear tensor category $\CC$ may be viewed as a ``$\hat G$-torsor in~$\CC$'', which underlies Section 8 of Deligne's article \cite{DeligneGroth}. In the present section we apply this point of view to the category of $F$-zips and describe an equivalence between exact $\BF_q$-linear tensor functors $\GhRep\to\FZip(S)$ and $\hat G$-zips.


\subsection{The stack of $\hat G$-zip functors}
\label{StackFunctors}

\begin{definition}
\label{ZipFunctorDefinition}
\begin{itemize}
\item[(a)] A \emph{$\hat G$-zip functor over $S$} is an exact $\BF_q$-linear tensor functor 
$$\Fz\colon \GhRep\to\FZip(S).$$
\item[(b)] A \emph{morphism} of $\hat G$-zip functors over $S$ is a natural transformation that is compatible with the tensor product.
\item[(c)] The resulting category of $\hat G$-zip functors over $S$ is denoted $\GhZipFunctor(S)$.
\end{itemize}
\end{definition}

With the evident notion of pullback the $\GhZipFunctor(S)$ form a fibered category over the category $(({\Sch}/\BF_q))$ of schemes over~$\BF_q$, which we denote $\GhZipFunctor$. 

\begin{proposition}\label{GZipFStack}
$\GhZipFunctor$ is a stack.
\end{proposition}

\begin{proof}
Since $\GhRep$ and $\FZip(S)$ are rigid tensor categories, any morphism of $\hat G$-zip functors is an isomorphism (see \cite[I.5.2.3]{Saavedra}); hence $\GhZipFunctor$ is fibered in groupoids. It remains to prove that $\GhZipFunctor$ satisfies effective descent for morphisms and objects. For this let $S'\to S$ be an fpqc covering and set $S''\defeq S'\times_S S'$.

First consider objects $\Fz_1$, $\Fz_2\in \GhZipFunctor(S)$ and a morphism $\lambda'\colon \Fz_{1,S'}\to\Fz_{2,S'}$ whose two pullbacks to $S''$ coincide. Since morphisms of $F$-zips satisfy effective descent with respect to the fpqc topology, for any $V\in \GhRep$ the homomorphism $\lambda'(V)\colon \Fz_1(V)_{S'}\to \Fz_2(V)_{S'}$ comes from a unique homomorphism $\lambda(V)\colon \Fz_1(V)\to \Fz_2(V)$. In order for $\lambda$ to be a tensor morphism, certain diagrams in $\FZip(S)$ need to commute. But as $\lambda'$ is a tensor morphism, these diagrams commute after pullback to~$S'$; hence by descent they commute over $S$ and thus $\lambda$ is a tensor morphism. Therefore $\GhZipFunctor$ satisfies effective descent for morphisms.

Now consider an object $\Fz'$ of $\GhZipFunctor(S')$ equipped with a descent datum. For each $V\in \GhRep$ this descent datum induces a descent datum on $\Fz'(V)$; hence it yields an object $\Fz(V)$ of $\FZip(S)$ with $\Fz'(V)=\Fz(V)_{S'}$. Next, the descent datum on $\Fz'(V)$ depends functorially on $V$. Thus for each morphism $f\colon V\to V'$ in $\GhRep$ the two pullbacks of $\Fz'(f)\colon \Fz(V)_{S'}\to \Fz(V')_{S'}$ coincide and therefore come from a unique morphism $\Fz(f)\colon\Fz(V)\to \Fz(V')$. The uniquess of $\Fz(f)$ implies that $\Fz\colon \GhRep\to\FZip(S)$ is a functor. Making $\Fz$ into a tensor functor requires functorial isomorphisms $\Fz(\Bone)\cong \Bone$ and $\Fz(V)\otimes\Fz(V')\cong \Fz(V\otimes V')$  for all $V$, $V'\in \GhRep$ which are compatible with the associativity, commutativity and unit constraints of the tensor category $\FZip(S)$. These are again obtained by descent from the corresponding isomorphisms for~$\Fz'$, and the compatibility with the constraints holds because it holds after pullback to~$S'$. Finally, by Lemma \ref{LocallyAdmissible2} the exactness of $\Fz$ follows from the exactness of~$\Fz'$. Altogether $\Fz$ is an element of $\GhZipFunctor(S)$ which gives rise to $\Fz'$ with its descent datum. Thus $\GhZipFunctor$ satisfies effective descent for objects and we are done.
\end{proof}

To analyze a zip functor we will compose it with the functors $\forget$, $\descfil$, and $\ascfil$ from (\ref{FunctorSummaryDiagram}). First we look at the numerical invariants obtained from the filtrations. 
Let $\bar\BF_q$ be an algebraic closure of~$\BF_q$. As in Section \ref{FFFAC} let $\CC_{\hat G}$ denote the set of $\hat G(\bar\BF_q)$-conjugacy classes of cocharacters $\BG_{m,\bar\BF_q}\to\nobreak \hat G_{\bar\BF_q}$, and let $k_c\subset\bar\BF_q$ denote the field of definition of an element $c\in \CC_{\hat G}$, which is a finite extension of~$\BF_q$.

\begin{definition}
\label{FunctorType}
Let $c\in \CC_{\hat G}$ and let $S$ be a scheme over $k_c$.
\begin{itemize}
\item[(a)] A $\hat G$-zip functor $\Fz$ over $S$ is called \emph{of type $c$}, or \emph{of type $\chi\in c$}, if the associated functor $\descgr \circ\descfil\circ\nobreak\Fz\colon\allowbreak \GhRep\to \GradLocFree(S)$ is of type~$c$ in the sense of Definition~\ref{DefTypeChi}.
\item[(b)] The full subcategory of $\GhZipFunctor(S)$ whose objects are the $G$-zip functors of type $c$ is denoted $\GhZipFunctor^c_{k_c}(S)$.
\end{itemize}
\end{definition}

With the evident notion of pullback the categories $\GhZipFunctor^c_{k_c}(S)$ form a fibered category over the category $(({\Sch}/k_c))$, which we denote $\GhZipFunctor^c_{k_c}$. Since Definition \ref{DefTypeChi} is local for the fpqc topology, Proposition \ref{GZipFStack} and Definition \ref{FunctorType} directly imply:

\begin{proposition}\label{GZipFchiStack}
$\GhZipFunctor^c_{k_c}$ is a substack of $\GhZipFunctor_{k_c}$.
\end{proposition}

The next result says that every zip functor over a connected scheme has a type.

\begin{proposition}
Let $S$ be a connected scheme over $\BF_q$ and $\Fz$ a $\hat G$-zip functor over~$S$. Then there exist a unique $c\in \CC_{\hat G}$ and a unique morphism $S\to \Spec k_c$ over $\BF_q$ such that $\Fz$ is of type $c$.
\end{proposition}


\begin{proof}
Direct consequence of Definition \ref{FunctorType} and Theorem \ref{CocharacterConstant1}.
\end{proof}

\begin{corollary}\label{GZipFDecomp}
\begin{itemize}
\item[(a)] Each $\GhZipFunctor^c_{k_c}$, viewed as a stack over $\BF_q$ by Grothendieck restriction, is an open and closed substack of $\GhZipFunctor$.
\item[(b)] $\GhZipFunctor$ is the disjoint union of the $\GhZipFunctor^c_{k_c}$ taken over all $c\in \CC_{\hat G}$.
\end{itemize}
\end{corollary}

\begin{theorem}[{\cite[Theorem 3.32]{ZieglerFFF}}] \label{InnerForm}
For any $c\in \CC_{\hat G}$ there exists an inner form 
\begin{equation*}
(\hat G', \tau\colon \hat G'_{\bar k}\isoto \hat G_{\bar k})
\end{equation*}
of $\hat G$ defined over $k_c$ and a cocharacter $\chi\colon \BG_{m,k_c}\to \hat G'$ such that $\tau \circ \chi_{\bar k}$ lies in $c$. 
\end{theorem}

\begin{remark}
From the Tannakian viewpoint, replacing $\hat G$ by an inner form does not change the category $\GhRep$; it merely endows it with a different fiber functor. In particular it does not change the stack of $\hat G$-zip functors. Thus Theorem \ref{InnerForm} implies that to study zip functors of a given type~$c$, we may without loss of generality assume that $c$ has a representative $\chi$ which is defined over~$k_c$.
\end{remark}


\subsection{Equivalence with $\hat G$-zips}
\label{FunctorsZips}

We now assume that the identity component $G$ of $\Ghat$ is reductive. We fix a finite field extension $k$ of $\BF_q$ and a cocharacter $\chi\colon \Gm k\to \Ghat_k$. We let $\Lhat\subset \Ghat_k$ denote the centralizer of~$\chi$ and set $\Theta := \pi_0(\Lhat)\subset\pi_0(\Ghat_k)$. Then we are in the situation of Subsection~\ref{notation2} with the maximal possible choice of~$\Theta$. We will use all the pertaining notation from Section \ref{Gzips}. Let $c$, $c^{(q)}\in \CC_{\hat G}$ denote the conjugacy classes of $\chi$, $\chi^{(q)}$. 

\begin{construction}\label{BasicZipFunctor}
For any finite dimensional representation $V$ of $\Ghat$ over~$\BF_q$, the cocharacter $\chi$ determines a grading $V_k = \bigoplus_{i\in\BZ}V_k^i$. This grading induces a descending filtration $C^\bullet(V_k)$. Also, the definition of $V_k$ by base extension induces a natural isomorphism $V_k^{(q)}\cong V_k$. Thus we may consider the decomposition $\bigoplus_{i\in\BZ}(V_k^i)^{(q)}$ as another grading of $V_k$, namely that induced by the cocharacter $\chi^{(q)}$. This grading induces an ascending filtration $D_\bullet(V_k)$. Then for all $i\in\BZ$ we obtain natural isomorphisms $\phi_i(V_k)\colon (\gr^i_C(V_k))^{(q)} \stackrel{\sim}{\to} (V_k^i)^{(q)} \stackrel{\sim}{\to} \gr_i^D(V_k)$. Altogether this data defines an $F$-zip over~$k$, denoted
$$\Fz_1(V) \defeq \bigl( V_k,C^\bullet(V_k),D_\bullet(V_k),\phi_\bullet(V_k) \bigr).$$
Clearly this construction is $\BF_q$-linearly functorial in $V$ and compatible with tensor product. It therefore defines a $\hat G$-zip functor over $k$
$$\Fz_1\colon \GhRep\to\FZip(\Spec k).$$
By pullback we obtain a zip functor $\Fz_{1,S}$ over any scheme $S$ over $k$. We will measure an arbitrary zip functor over $S$ by how it differs from this basic zip functor~$\Fz_{1,S}$.
\end{construction}

\begin{lemma}
  \label{StandardStabilizer}
There are natural isomorphisms 
\begin{itemize}
\item[(a)] $\UAut^\otimes(\forget\!\circ\!\Fz_1) \cong \hat G_k$,
\item[(b)] $\UAut^\otimes(\descfil\!\circ\!\Fz_1) \cong \hat P_+$,
\item[(c)] $\UAut^\otimes(\ascfil\!\circ\!\Fz_1)  \cong \hat P_-^{(q)}$.
\end{itemize}
\end{lemma}

\begin{proof}
(a) is an instance of the main theorem of Tannaka duality \cite[Th.$\,$1.12]{DeligneGroth},
and (b) and (c) are instances of Proposition \ref{AutOmegaChiIsP}.
\end{proof}

\begin{construction}
  \label{GZipFunctorToGZip}
Let $S$ be a scheme over~$k$, and $\Fz$ a $\hat G$-zip functor of type $c$ over~$S$. Then
\begin{itemize}
\item[(a)] $I_\Fz \defeq \UIsom^\otimes (\forget\!\circ\!\Fz_{1,S},\forget\!\circ\!\Fz)$ is a right $\hat G_k$-torsor over~$S$,
\item[(b)] $I_{\Fz,+} \defeq \UIsom^\otimes(\descfil\!\circ\!\Fz_{1,S},\descfil\!\circ\!\Fz)$ is a right $\hat P_+$-torsor over~$S$,
\item[(c)] $I_{\Fz,-} \defeq \UIsom^\otimes(\ascfil\!\circ\!\Fz_{1,S},\ascfil\!\circ\!\Fz)$ is a right $\hat P^{(q)}_-$-torsor over~$S$.
\end{itemize}
Indeed, (a) follows from Lemma \ref{StandardStabilizer} (a) and \cite[Th.$\,$1.12]{DeligneGroth}, and~(b) results from combining Lemma \ref{StandardStabilizer} (b) with Theorem \ref{FilIsomTorsor1a} above. Also, the commutativity of (\ref{FunctorSummaryDiagram}) shows that $(\ )^q \circ \gr^\bullet_C \circ\descfil\circ\Fz \cong \ascgr \circ\ascfil\circ\Fz$ is of type $c^{(q)}$, and so (c) follows from Lemma \ref{StandardStabilizer}~(c) and Theorem \ref{FilIsomTorsor1a}.

Moreover, composition with the functors forgetting the filtration induces a natural $\hat P_+$-equivariant embedding
$$\xymatrix@R-10pt{
I_{\Fz,+}\ \ar@{^{(}->}[r] & \ 
\UIsom^\otimes\bigl(\forget\!\circ\!\descfil\!\circ\!\Fz_{1,S},\forget\!\circ\!\descfil\!\circ\!\Fz\bigr)
\ar@{=}[d]^\wr \\
I_\Fz\ \ar@{=}[r]^-\sim & \ 
\UIsom^\otimes\bigl(\forget\!\circ\!\Fz_{1,S},\forget\!\circ\!\Fz\bigr) 
\\}$$
and likewise a natural $\hat P_-^{(q)}$-equivariant embedding $I_{\Fz,-} \into I_\Fz$. Furthermore, by Theorem \ref{TorsorQuotient} we have natural isomorphisms of $\hat L^{(q)}$-torsors in the rows of the following diagram, where the vertical isomorphism is induced by the isomorphism of tensor functors $\phi \colon (\ )^{(q)} \circ \gr^\bullet_C \circ \descfil\isoto \gr_\bullet ^D\circ \ascfil$ from (\ref{FunctorSummaryDiagram}):
$$\xymatrix@R-10pt{
I_{\Fz,+}^{(q)}/U_+^{(q)}\ \ar@{=}[r]^-\sim & \ 
\UIsom^\otimes((\ )^{(q)} \circ \descgr \!\circ\! \descfil \!\circ\! \Fz_{1,S}, (\ )^{(q)} \circ \descgr \!\circ\! \descfil \!\circ\!\Fz) 
\ar[d]^\wr \\
I_{\Fz,-}/U_-^{(q)}\ \ar@{=}[r]^-\sim & \ 
\UIsom^\otimes(\ascgr\!\circ\!\ascfil\!\circ\!\Fz_{1,S},\ascgr\!\circ\!\ascfil\!\circ\!\Fz) \\}$$
The composite is therefore an isomorphism of $\hat L^{(q)}$-torsors $\iota_\Fz \colon I_{\Fz,+}^{(q)}/U_+^{(q)}\isoto I_{\Fz,-}/U_-^{(q)}$. Together this data defines a $\hat G$-zip of type $(\chi,\Theta)$ over $S$
\[
\UI_\Fz \defeq (I_\Fz,I_{\Fz,+},I_{\Fz,-},\iota_\Fz).
\]
Clearly this construction is $\BF_q$-linearly functorial in $\Fz$ and compatible with pullback. Thus it defines a morphism of stacks
\UseTheoremCounterForNextEquation
\begin{equation}
\label{MorfMorf}
\GhZipFunctor_k^c\to\GhatZip^{(\chi,\Theta)}_k,\ \ \Fz\mapsto\UI_\Fz.
\end{equation}
\end{construction}

\begin{theorem}
\label{GZip=GZipFun}
This morphism is an isomorphism.
\end{theorem}

\begin{proof}
We construct a morphism in the other direction, as follows. Consider a $\hat G$-zip $\UI = (I,I_+,I_-,\iota)$ of type $(\chi,\Theta)$ over $S$. The essential surjectivity in Theorem \ref{FilIsomTorsor1a} shows that $I_- \cong \UIsom^\otimes(\ascfil\circ\Fz_{1,S},\psi_-)$ for some exact $\BF_q$-linear tensor functor
$$\psi_-\colon\GhRep\to\AscFilLocFree(S), \quad V\mapsto (\CM(V),D_\bullet).$$
The embedding $I_-\into I$ and the fullness in Theorem \ref{FilIsomTorsor1a} then yield an isomorphism $I\cong\UIsom^\otimes(\forget\circ\Fz_{1,S},\omega)$ with $\omega \defeq \forget\circ\psi_-\colon V\mapsto \CM(V)$.
The essential surjectivity in Theorem \ref{FilIsomTorsor1a} also shows that $I_+ \cong \UIsom^\otimes(\descfil\circ\Fz_{1,S},\psi_+)$ for some exact $\BF_q$-linear tensor functor 
$\psi_+\colon\GhRep\to\DescFilLocFree(S)$. The embedding $I_+\into I$ and the fullness in Theorem \ref{FilIsomTorsor1a} then yield an isomorphism $\forget\circ\psi_+ \cong \omega$. After replacing $\psi_+$ by an isomorphic functor we may therefore assume that $\psi_+$ has the form $V \mapsto (\CM(V),C^\bullet)$.
Moreover, Theorem \ref{TorsorQuotient} and $\iota$ yield isomorphisms
$$\xymatrix@R-10pt{
I_+^{(q)}/U^{(q)}_+\ \ar@{=}[r]^-\sim \ar[d]^\wr_\iota & \ 
\UIsom^\otimes((\ )^{(q)} \circ \descgr\!\circ\!\descfil\!\circ\!\Fz_{1,S},(\ )^{(q)} \circ \descgr\!\circ\!\psi_+) \\
I_-/U_-^{(q)}\ \ar@{=}[r]^-\sim & \ 
\UIsom^\otimes(\ascgr\!\circ\!\ascfil \!\circ\!\Fz_{1,S},\ascgr\!\circ\!\psi_-) \rlap{.} \\}$$
Thus the isomorphism $\phi_\bullet\colon (\ )^{(q)} \circ \gr\circ\descfil\circ\Fz_{1,S} \isoto \gr\circ\ascfil\circ\Fz_{1,S}$ from Construction \ref{BasicZipFunctor} and the fullness in Theorem \ref{GradTorsor1a} yield an isomorphism $(\ )^{(q)} \circ \descgr\circ\psi_+ \isoto \ascgr\circ\psi_-$. This amounts to graded isomorphisms $\phi_\bullet\colon (\gr_C^\bullet\CM(V))^{(q)} \isoto \gr^D_\bullet\CM(V)$ for all $V\in\GhRep$ that are functorial and compatible with tensor product. The assembled data thus determines a $\hat G$-zip functor
$$\Fz_\UI\colon V\mapsto \bigl( \CM(V), C^\bullet, D_\bullet, \phi_\bullet)$$
of type $c$ over $S$. By definition it satisfies $\descfil\circ\Fz=\psi_+$ and $\ascfil\circ\Fz=\psi_-$; comparing this construction with Construction \ref{GZipFunctorToGZip} therefore yields an isomorphism $\UI_{\Fz_\UI} \cong \UI$.

The faithfulness in Theorems \ref{FilIsomTorsor1a} and \ref{GradTorsor1a} implies that $\Fz$ is unique up to unique isomorphism. It is therefore functorial in $\UI$. As the construction is clearly compatible with pullback, it thus defines a morphism of stacks $\GhatZip^{(\chi,\Theta)}_k \to \GhZipFunctor_k^c$. Again by faithfulness the isomorphism $\UI_{\Fz_\UI} \cong \UI$ is functorial in $\UI$ and compatible with pullback; hence $\UI\mapsto\Fz_\UI$ is a right inverse of $\Fz\mapsto\UI_\Fz$. Moreover, applying the above construction to $\UI_\Fz$ for a $\hat G$-zip functor $\Fz$ one easily shows that $\Fz_{\UI_\Fz} \cong \Fz$, and so the morphism is also a left inverse. Thus the morphism (\ref{MorfMorf}) has a two-sided inverse and is therefore an isomorphism, as desired.
\end{proof}




\section{$F$-zips with additional structure}
\label{Fadd}


An important tool in the study of vector bundles is the equivalence between vector bundles of constant rank $n$ on a scheme $S$ and the associated $\GL_n$-torsors. One also uses the equivalence between vector bundles with a non-degenerate symmetric, alternating, resp.\ hermitian pairing and the associated torsors with respect to the orthogonal, symplectic, resp.\ unitary group. In this section we describe similar equivalences between $F$-zips of constant rank $n$ and $\GL_n$-zips, and between $F$-zips with additional structure such as a pairing and $G$-zips for certain associated linear algebraic groups~$G$.

Let $\Un = (n_i)_{i\in\BZ}$ be a family of non-negative integers which vanish for almost all~$i$, such that $n\defeq \sum_in_i\ge1$. 


\subsection{$F$-zips versus $\GL_n$-zips}
\label{GLnZips}

For any scheme $S$ over $k\defeq\BF_q$ we let $\FZip^\Un_k(S)$ denote the category whose objects are all $F$-zips of type $\Un$ over $S$ according to Definition \ref{FZipTypeDef} and whose morphisms are all isomorphisms. For varying $S$ this defines a category $\FZip^\Un_k$ fibered in groupoids over $(({\Sch}/k))$. Since $F$-zips consist of quasi-coherent sheaves and homomorphisms thereof, they satisfy effective descent with respect to any fpqc morphism $S'\to S$. Therefore $\FZip^\Un_k$ is a stack.

Choose a cocharacter $\chi\colon \BG_{m,k}\to \GL_{n,k}$ whose weights on the standard representation $k^n$ of $\GL_{n,k}$ are $i$ with multiplicity $n_i$ for all~$i$. This determines a grading of~$k^n$, whose associated descending and ascending filtrations we denote by $C^\bullet$ and~$D_\bullet$. Since $k=\BF_q$, there is a natural isomorphism
$\phi_{1\bullet}\colon (\gr_C^\bullet k^n)^{(q)} = \gr_C^\bullet k^n \stackrel{\sim}{\to} \gr^D_\bullet k^n$ turning 
\UseTheoremCounterForNextEquation
\begin{equation}\label{BasicGln}
\UCM_1 := (k^n,C^\bullet,D_\bullet,\phi_{1\bullet})
\end{equation}
into an $F$-zip of type $\Un$ over~$k$. As in Subsection \ref{FunctorsZips}, we compare arbitrary $F$-zips of type $\Un$ with this basic one.

Set $G\defeq\GL_{n,k}$, and let $P_\pm =  L\ltimes U_\pm$ be the parabolics of $G$ associated to~$\chi$, as in Subsection \ref{notation2}. Thus $P_+$ is the stabilizer of the filtration~$C^\bullet$, and $P_-$ is the stabilizer of~$D_\bullet$. Since $\chi$ is defined over $\BF_q$, we have $\chi^{(q)} = \chi$ and $P_\pm^{(q)} = L^{(q)}\ltimes U_\pm^{(q)} = P_\pm = L\ltimes U_\pm$.
Also, since $G$ is connected, we have $\Theta=1$ in this case.

\medskip
In the following, for any graded, filtered, or naked sheaves of $\CO_S$-modules $\CM_1$ and $\CM_2$
we let $\UIsom(\CM_1,\CM_2)$ denote the fpqc-sheaf on $(({\Sch}/S))$ sending $S'\to S$ to the set of (graded, filtered, resp.\ neither) isomorphisms $\CM_{1,S'}\stackrel{\sim}{\to}\CM_{2,S'}$. By composition of isomorphisms it carries a natural right action of the sheaf of groups $\UIsom(\CM_1,\CM_1) = \UAut(\CM_1)$. This sheaf is representable by a smooth affine group scheme over $S$ if $\CM_1$ is locally free of finite rank.

\begin{construction}
\label{FZipToGLnZip}
Let $\UCM = (\CM,C^\bullet,D_\bullet,\phi_\bullet)$ be an $F$-zip of type $\Un$ over $S$. Then $\CM$ is a locally free sheaf of rank $n$, and the filtered sheaves $(\CM,C^\bullet)$ and $(\CM,D_\bullet)$ are Zariski locally isomorphic to $(k^n,C^\bullet)_S$ and $(k^n,D_\bullet)_S$, respectively. Thus 
\begin{itemize}
\item[(a)] $I \defeq \UIsom((k^n)_S,\CM)$ is a right $\GL_n$-torsor over~$S$,
\item[(b)] $I_+ \defeq \UIsom\bigl((k^n,C^\bullet)_S,(\CM,C^\bullet)\bigr)$ is a right $P_+$-torsor over~$S$,
\item[(c)] $I_- \defeq \UIsom\bigl((k^n,D_\bullet)_S,(\CM,D_\bullet)\bigr)$ is a right $P_-^{(q)}$-torsor over~$S$.
\end{itemize}
Forgetting the filtration induces natural equivariant embeddings $I_\pm\into I$. Also, the functors $\descgr$ and $\ascgr$ induce natural isomorphisms $I_+/U_+\cong \UIsom\bigl( (\gr_C^\bullet k^n)_S , \gr_C^\bullet\CM \bigr)$ 
and $I_-/U_-^{(q)} \cong \UIsom\bigl( (\gr^D_\bullet k^n)_S , \gr^D_\bullet\CM \bigr)$. Moreover, the isomorphisms $\phi_{1\bullet}\colon (\gr_C^\bullet k^n)^{(q)} \allowbreak \stackrel{\sim}{\to} \gr^D_\bullet k^n$ and $\phi_\bullet\colon (\gr_C^\bullet\CM)^{(q)}\stackrel{\sim}{\to}\gr^D_\bullet\CM$ induce an isomorphism $\iota$ of $L$-torsors making the following diagram commute:
$$\xymatrix@C-15pt@R-10pt{
I_+^{(q)}/U_+^{(q)}\ \ar@{=}[r]^-\sim \ar[d]_\iota & \ 
\UIsom\bigl( (\gr_C^\bullet k^n)^{(q)}_S , (\gr_C^\bullet\CM)^{(q)} \bigr) 
\ar[d]^\wr \\
I_-/U_-^{(q)}\ \ar@{=}[r]^-\sim & \ 
\UIsom\bigl( (\gr^D_\bullet k^n)_S , \gr^D_\bullet\CM \bigr). \\}$$
Together this data defines a $\GL_n$-zip $\UI \defeq (I,I_+,I_-,\iota)$ of type $\chi$ over $S$.
Clearly this construction is $k$-linearly functorial in $\UCM$ and compatible with pullback. Thus it defines a morphism of stacks
\UseTheoremCounterForNextEquation
\begin{equation}
\label{GLnMorf}
\FZip^\Un_k\longto\GLnZip^\chi_k.
\end{equation}
\end{construction}

\begin{proposition}
This morphism is an isomorphism.
\end{proposition}

\begin{proof}
A morphism in the other direction is obtained by evaluating the inverse of the morphism (\ref{MorfMorf}) at the standard representation $V=k^n$ of $\GL_{n,k}$. It is straightforward to check that these two morphisms are mutually inverse.
\end{proof}

Combined with Theorem \ref{GZip=GZipFun} this shows in particular that giving an $F$-zip of type~$\Un$, or a $\GL_n$-zip of type~$\chi$, or a $\GL_n$-zip functor of type~$\chi$, are all equivalent. Also, combined with Proposition \ref{GZipStack} it shows that the isomorphism classes of $F$-zips of type $\Un$ over an algebraically closed field $K$ containing $\BF_q$ are in bijection with the $E_{\GL_n,\chi}(K)$-orbits on $\GL_n(K)$, which  in turn have a combinatorial description in terms of the Weyl group of $\GL_n$, as in Example \ref{SplitExample}. The analogous remarks apply to the cases treated in the rest of this section.


\subsection{$F$-zips with trivialized determinant versus $\SL_n$-zips}
\label{SLnZips}

Keeping the notations of the preceding subsection, we now assume that $\sum_in_ii=0$. Then the highest exterior power of any $F$-zip of type $\Un$ is an $F$-zip of rank $1$ whose filtrations are concentrated in degree~$0$.
We call a pair $(\UCM,\Delta)$ consisting of an $F$-zip $\UCM$ of type $\Un$ and an isomorphism $\Delta\colon \Alt^n\UCM\isoto\UBOne(0)$ an \emph{$F$-zip of type $\Un$ with trivialized determinant}. For the same reasons as before, the $F$-zips of type $\Un$ with trivialized determinant, together with isomorphisms of such pairs, form a stack over~$k$. 

Let $\Delta_1\colon \Alt^n(k^n)\isoto k$ denote the isomorphism induced by the determinant. With the basic $F$-zip from (\ref{BasicGln}) the pair $(\UCM_1,\Delta_1)$ is then an $F$-zip of type $\Un$ with trivialized determinant over~$k$. As in the preceding subsection, we compare arbitrary $F$-zips of type $\Un$ with trivialized determinant with this basic one.

The relevant linear algebraic group is now $\SL_{n,k}$. Clearly $\chi$ factors through $\SL_{n,k}$, so that we can speak of $\SL_n$-zips of type~$\chi$. The associated parabolics of $\SL_{n,k}$ are now $P_\pm\cap\SL_{n,k}$ with $P_\pm$ as in Subsection \ref{GLnZips}.

Note that $\Delta_1\colon \Alt^n(k^n)\isoto k$ is an isomorphism in the category $\SL_n$-$\Rep$ if the target $k$ is endowed with the trivial representation. The equivalence (\ref{SLnMorf}) below can be interpreted as saying that an $F$-zip with trivialized determinant is a partial $\SL_n$-zip functor containing just enough information to possess a unique extension to a full $\SL_n$-zip functor $\Fz\colon \SL_n$-$\Rep\to\FZip(S)$.

\begin{construction}
\label{FZipDetToSLnZip}
Let $(\UCM,\Delta)$ with $\UCM = (\CM,C^\bullet,D_\bullet,\phi_\bullet)$ be an $F$-zip of type $\Un$ with trivialized determinant over~$S$. Let $\UI \defeq (I,I_+,I_-,\iota)$ be the $\GL_n$-zip associated to $\UCM$ by Construction \ref{FZipToGLnZip}. Let $I' \subset I$ be the subsheaf of all isomorphisms $u\colon(k^n)_{S'}\isoto\CM_{S'}$ for which the composite 
$$\mathstrut\smash{\xymatrix{\Alt^n(k^n)_{S'} \ar[rr]^-{\Alt^nu} &&
  \Alt^n\CM_{S'} \ar[r]^-\Delta & (k)_{S'}}}$$
is equal to $\Delta_{1,S'}$. One easily checks that 
\begin{itemize}
\item[(a)] $I'$ is a right $\SL_n$-torsor over~$S$,
\item[(b)] $I'_+ \defeq I_+ \cap I'$ is a right $P_+\cap\SL_{n,k}$-torsor over~$S$,
\item[(c)] $I'_- \defeq I_- \cap I'$ is a right $P_-^{(q)}\cap\SL_{n,k}$-torsor over~$S$.
\end{itemize}
Moreover, the highest exterior power of a filtered locally free sheaf of finite rank is canonically isomorphic to the highest exterior power of the associated graded sheaf. Thus the isomorphism of $F$-zips $\Delta$ amounts to a commutative diagram of isomorphisms
$$\xymatrix@C-15pt{
{\mathstrut}\Alt^n(\gr^\bullet_C\CM)^{(q)}\, \ar@{=}[r]^-\sim \ar[d]^\wr_{\Alt^n\phi_\bullet} &
{\mathstrut}\, \Alt^n\CM^{(q)} \ar[rrr]^-{\Delta^{(q)}}_-\sim &&&
{\mathstrut}(k)_S^{(q)} \ar@{=}[d] \\
{\mathstrut}\Alt^n\gr_\bullet^D\CM\, \ar@{=}[r]^-\sim &
{\mathstrut}\, \Alt^n\CM \ar[rrr]^-\Delta_-\sim &&&
{\mathstrut}(k)_S \rlap{.} \\}$$
From this one easily deduces that the isomorphism $\iota\colon I^{(q)}_+/U_+^{(q)}\stackrel{\sim}{\to} I_-/U_-^{(q)}$ induces an isomorphism $\iota'\colon (I')^{(q)}_+/U_+^{(q)}\stackrel{\sim}{\to} I'_-/U_-^{(q)}$. Together the assembled data therefore defines an $\SL_n$-zip $\UI' \defeq (I',I'_+,I'_-,\iota')$ of type $\chi$ over~$S$. Clearly this construction is $\BF_q$-linearly functorial in $(\UCM,\Delta)$ and compatible with pullback. Thus it defines a morphism of stacks
\UseTheoremCounterForNextEquation
\begin{equation}
\label{SLnMorf}
\bigl(\!\bigl(\hbox{$F$-zips of type $\Un$ with trivialized determinant}\bigr)\!\bigr)
\longto\SLnZip^\chi_k.
\end{equation}
\end{construction}

\begin{proposition}
This morphism is an isomorphism.
\end{proposition}

\begin{proof}
By the remarks in Subsection \ref{Powers} and exactness, any zip functor $\Fz\colon \SLnRep\to\FZip(S)$ commutes with alternating powers; hence it sends $\Delta_1$ to an isomorphism $\Fz(\Delta_1)\colon \Alt^n(\Fz(k^n)) \cong \Fz(\Alt^n(k^n)) \isoto \Fz(k) = \UBOne(0)$. Therefore $\Fz\mapsto (\Fz(k^n),\Fz(\Delta_1))$ defines a morphism from the stack of $\SL_n$-zips of type $\chi$ to the stack of $F$-zips of type $\Un$ with trivialized determinant. Composed with the inverse of the morphism (\ref{MorfMorf}) we thus obtain a morphism in the other direction. The careful reader will be able to check that this is a two-sided inverse of the morphism (\ref{SLnMorf}).
\end{proof}


\subsection{Symplectic $F$-zips versus $\Sp_n$-zips}
\label{SymplecticZips2}

We call a pair $(\UCM,E)$ consisting of an $F$-zip $\UCM$ of type $\Un$ over~$S$ and an admissible epimorphism $E\colon \Alt^2\UCM\onto \UBOne(0)$, whose underlying alternating pairing $\CM\times\CM\to(k)_S$ is nondegenerate everywhere, a \emph{symplectic $F$-zip of type $\Un$ over~$S$}. For the same reasons as before these pairs, together with compatible isomorphisms, form a stack over $k:=\BF_q$. For this stack to be non-empty we assume that $n\defeq \sum_in_i$ is even and that $n_i=n_{-i}$ for all~$i$.

Fix a non-degenerate alternating pairing $E_1\colon k^n\times k^n\to k$, and let $\Sp_{n,k} \subset \GL_{n,k}$ denote the associated symplectic group. Then $E_1$ can be viewed as an equivariant epimorphism $\Alt^2(k^n)\onto k$, where $\Sp_{n,k}$ acts trivially on the target~$k$.
By the assumptions on $\Un$ there exists a cocharacter $\chi\colon \BG_{m,k}\to \Sp_{n,k}$, unique up to conjugation, whose weights on the standard representation $k^n$ of $\Sp_{n,k}$ are $i$ with multiplicity $n_i$ for all~$i$. Fixing such a cocharacter, we can thus speak of $\Sp_n$-zips of type~$\chi$ over any scheme $S$ over $k$.

To any symplectic $F$-zip $(\UCM,E)$ of type $\Un$ over~$S$ we can associate an $\Sp_n$-zip $\UI'\defeq (I',I_+',I_-',\iota')$ of type $\chi$ over~$S$. Namely, if $\UI \defeq (I,I_+,I_-,\iota)$ denotes the $\GL_n$-zip associated to $\UCM$ by Construction \ref{FZipToGLnZip}, we let $I'\subset I$ be the subsheaf of isomorphisms $(k^n)_{S}\isoto\CM_{S}$ which are compatible with $E_1$ and~$E$. From the fact that any two non-degenerate alternating pairings on the sheaf $\CO_{S}^{\oplus n}$ are conjugate under $\GL_n(S)$ one deduces that this is an $\Sp_{n,k}$-torsor. Also $I_+', I_-' \subset I'$ are the subsheaves of isomorphisms preserving the filtrations $C^\bullet$, respectively $D_\bullet$, and $\iota'$ is constructed from the isomorphism $\phi_\bullet$ in $\UCM$. Together we obtain a morphism of stacks
\UseTheoremCounterForNextEquation
\begin{equation}
\label{SpnMorf}
\bigl(\!\bigl(\hbox{symplectic $F$-zips of type $\Un$}\bigr)\!\bigr)
\longto\SpnZip^\chi_k.
\end{equation}

Conversely, for any $\Sp_n$-zip of type $\chi$ over $S$ we evaluate the associated zip functor $\Fz$ on the standard representation $k^n$ and the homomorphism $E_1$ in \hbox{$\Sp_{n,k}$-$\Rep$}, obtaining a symplectic $F$-zip $(\Fz(k^n),\Fz(E_1))$ of type $\Un$ over~$S$. By showing that this construction yields a two-sided inverse of the first one proves that the morphism (\ref{SpnMorf}) is an isomorphism.


\subsection{Twisted symplectic $F$-zips versus $\CSp_n$-zips}
\label{SymplecticZips1}

We call a triple $(\UCM,\UCL,E)$ consisting of an $F$-zip $\UCM$ of type $\Un$ over~$S$, an $F$-zip $\UCL$ of rank $1$ over~$S$, and an admissible epimorphism $E\colon \Alt^2\UCM\onto\UCL$, whose underlying alternating pairing $\CM\times\CM\to\CL$ is nondegenerate everywhere on~$S$, a \emph{twisted symplectic $F$-zip of type $\Un$ over~$S$}. For the same reasons as before these triples, together with compatible isomorphisms, form a stack over $k:=\BF_q$. For this stack to be non-empty we assume that $n\defeq \sum_in_i$ is even and that there is an integer $d$ satisfying $n_i=n_{d-i}$ for all~$i$. This $d$ is then unique, and the above $\UCL$ must be of type~$d$. 

Fix a non-degenerate alternating pairing $E_1\colon k^n\times k^n\to k$, and let $\CSp_{n,k}$ denote the associated group of symplectic similitudes, i.e., the group of all $g\in\GL_{n,k}$ satisfying $E_1\circ(g\times g)=\mu(g)\cdot E_1$ for a scalar $\mu(g)$.  Then $E_1$ can be viewed as a $\CSp_{n,k}$-equivariant epimorphism $\Alt^2(k^n)\onto k$, where $\CSp_{n,k}$ acts on the target $k$ through the multiplier character $\mu\colon \CSp_{n,k}\onto\BG_{m,k}$.
By the assumptions on $\Un$ there exists a cocharacter $\chi\colon \BG_{m,k}\to \CSp_{n,k}$, unique up to conjugation, whose weights on the standard representation $k^n$ of $\CSp_{n,k}$ are $i$ with multiplicity $n_i$ for all~$i$. Fixing such a cocharacter, 
we can thus speak of $\CSp_n$-zips of type~$\chi$ over any scheme $S$ over~$k$.

Using the same principles as before, to any twisted symplectic $F$-zip $(\UCM,\UCL,E)$ of type $\Un$ over~$S$ we can associate a $\CSp_n$-zip $\UI\defeq (I,I_+,I_-,\iota)$ of type $\chi$ over~$S$. 
In the interest of brevity we only sketch the construction: Here $I$ is the sheaf of pairs of isomorphisms $(k^n)_{S'}\isoto\CM_{S'}$ and $(k)_{S'}\isoto \CL_{S'}$ that are compatible with $E_1$ and~$E$. That this is a $\CSp_{n,k}$-torsor again results from the fact that any two non-degenerate alternating pairings on the sheaf $\CO_{S'}^{\oplus n}$ are conjugate under $\GL_n(S')$. Also $I_+, I_- \subset I$ are the subsheaves of isomorphisms preserving the filtrations~$C^\bullet$, respectively $D_\bullet$, and $\iota$ is constructed from the isomorphisms $\phi_\bullet$ in $\UCM$ and $\UCL$. Together we obtain a morphism of stacks
\UseTheoremCounterForNextEquation
\begin{equation}
\label{CSpnMorf}
\bigl(\!\bigl(\hbox{twisted symplectic $F$-zips of type $\Un$}\bigr)\!\bigr)
\longto\CSpnZip^\chi_k.
\end{equation}

Conversely, for any $\CSp_n$-zip of type $\chi$ over $S$ we evaluate the associated zip functor $\Fz$ on the standard representation $k^n$, the representation $k$ with the multiplier character~$\mu$, and the homomorphism~$E_1$, obtaining a twisted symplectic $F$-zip $(\Fz(k^n),\Fz(k),\Fz(E_1))$ of type $\Un$ over~$S$. By showing that this construction yields a two-sided inverse of the first one proves that the morphism (\ref{CSpnMorf}) is an isomorphism. The details in these arguments follow those in the preceding subsections and are left to the conscientious reader.


\subsection{Orthogonal $F$-zips versus $\groupO_n$-zips}
\label{OrthogonalZips2}

To avoid the usual idiosyncrasies of symmetric bilinear forms in characteristic $2$ we assume that $q$ is odd in this subsection and the next. We call a pair $(\UCM,B)$ consisting of an $F$-zip $\UCM$ of type $\Un$ over~$S$ and an admissible epimorphism $E\colon \Sym^2\UCM\onto \UBOne(0)$, whose underlying symmetric pairing $\CM\times\CM\to(k)_S$ is nondegenerate everywhere, an \emph{orthogonal $F$-zip of type $\Un$ over~$S$}. For the same reasons as before these pairs, together with compatible isomorphisms, form a stack over $k:=\BF_q$. For this stack to be non-empty we assume that $n_i=n_{-i}$ for all~$i$. 

Fix a non-degenerate split symmetric bilinear form $B_1\colon k^n\times k^n\to k$, and let $\groupO_{n,k} \subset \GL_{n,k}$ denote the associated orthogonal group. Then $B_1$ can be viewed as an equivariant epimorphism $\Sym^2(k^n)\onto k$, where $\groupO_{n,k}$ acts trivially on the target~$k$.
Note that $\groupO_{n,k}$ has two connected components and that its identity component is a split special orthogonal group $\SO_{n,k}$.
By the assumptions on $\Un$ there exists a cocharacter $\chi\colon \BG_{m,k}\to \groupO_{n,k}$, unique up to conjugation, whose weights on the standard representation $k^n$ are $i$ with multiplicity $n_i$ for all~$i$. We fix such a cocharacter and set $\Lhat := \Cent_{\groupO_{n,k}}(\chi)$ and $\Theta\defeq\pi_0(\Lhat) \subset \pi_0(\groupO_{n,k})$. A quick calculation shows that $\Lhat \cong \groupO_{n_0,k} \times \prod_{i>0}\GL_{n_i,k}$; hence $\Theta$ is trivial if $n_0=0$, and equal to $\pi_0(\groupO_{n,k})$ if $n_0>0$. According to Definition \ref{GZipDef} we can speak of $\groupO_n$-zips of type $(\chi,\Theta)$ over any scheme $S$ over $k$.

The definition of $\Lhat$ implies that the associated subgroups $\Phat_\pm$ from Subsection \ref{notation2} are precisely the stabilizers of the descending and ascending filtrations of $k^n$ induced by~$\chi$. Also, observe that any two non-degenerate symmetric pairings on the sheaf $\CO_S^{\oplus n}$ are fpqc-locally conjugate under $\GL_n$. Using these facts and the same construction as in Subsection \ref{SymplecticZips2}, to any orthogonal $F$-zip $(\UCM,B)$ of type $\Un$ over~$S$ we can associate an $\groupO_n$-zip of type $(\chi,\Theta)$ over~$S$, obtaining a morphism of stacks
\UseTheoremCounterForNextEquation
\begin{equation}
\label{OnMorf}
\bigl(\!\bigl(\hbox{orthogonal $F$-zips of type $\Un$}\bigr)\!\bigr)
\longto\OnZip^{\chi,\Theta}_k.
\end{equation}

Conversely, for any $\groupO_n$-zip of type $(\chi,\Theta)$ over $S$ we evaluate the associated zip functor $\Fz$ on the standard representation $k^n$ and the homomorphism $B_1$ in \hbox{$\groupO_{n,k}$-$\Rep$}, obtaining an orthogonal $F$-zip $(\Fz(k^n),\Fz(B_1))$ of type $\Un$ over~$S$. This construction yields a two-sided inverse of the first and thereby proves that the morphism (\ref{OnMorf}) is an isomorphism.


\subsection{Twisted orthogonal $F$-zips versus $\groupCO_n$-zips}
\label{OrthogonalZips1}

Again we assume that $q$ is odd. We call a triple $(\UCM,\UCL,B)$ consisting of an $F$-zip $\UCM$ of type $\Un$ over~$S$, an $F$-zip $\UCL$ of rank $1$ over~$S$, and an admissible epimorphism $B\colon \Sym^2\UCM\onto\UCL$, whose underlying symmetric pairing $\CM\times\CM\to\CL$ is nondegenerate everywhere on~$S$, a \emph{twisted orthogonal $F$-zip of type $\Un$ over~$S$}. For the same reasons as before these triples, together with compatible isomorphisms, form a stack over $k:=\BF_q$. For this stack to be non-empty we assume that there is an integer $d$ satisfying $n_i=n_{d-i}$ for all~$i$. This $d$ is then unique, and the above $\UCL$ must be of type~$d$. 

Fix a non-degenerate split symmetric bilinear form $B_1\colon k^n\times k^n\to k$. Let $\groupCO_{n,k}$ denote the associated group of orthogonal similitudes, i.e., the group of all $g\in\GL_{n,k}$ satisfying $B_1\circ(g\times g)=\mu(g)\cdot B_1$ for a scalar $\mu(g)$. Then $B_1$ can be viewed as a $\groupCO_{n,k}$-equivariant epimorphism $\Sym^2(k^n)\onto k$, where $\groupCO_{n,k}$ acts on the target $k$ through the character $\mu\colon \groupCO_{n,k}\onto\BG_{m,k}$.
If $n$ is odd, then $\groupCO_{n,k}$ is connected with a root system of type $B_{(n-1)/2}$ and is therefore split. If $n$ is even, then $\groupCO_{n,k}$ has two connected components and a root system of type $D_{n/2}$. 
In both cases the identity component of $\groupCO_{n,k}$ is split, because $B_1$ is split. Thus there exists a cocharacter $\chi\colon \BG_{m,k}\to \groupCO_{n,k}$, unique up to conjugation, whose weights on the standard representation $k^n$ of $\groupCO_{n,k}$ are $i$ with multiplicity $n_i$ for all~$i$. We fix such a cocharacter and set $\Lhat := \Cent_{\groupCO_{n,k}}(\chi)$ and $\Theta\defeq\pi_0(\Lhat) \subset \pi_0(\groupCO_{n,k})$.
A quick calculation shows that $\Lhat \cong \groupCO_{n_{d/2},k} \times \prod_{i>d/2}\GL_{n_i,k}$ if $d$ is even and $n_{d/2}>0$, and $\Lhat \cong \BG_{m,k} \times \prod_{i>d/2}\GL_{n_i,k}$ otherwise. Thus $\Theta$ is trivial unless $d$ is even and $n_{d/2}$ is even and positive, in which case $\Theta = \pi_0(\groupCO_{n,k})$ of order~$2$. In either case we can speak of $\groupCO_n$-zips of type $(\chi,\Theta)$ over any scheme $S$ over $k$.

In the same way as in Subsection \ref{SymplecticZips1}, to any twisted orthogonal $F$-zip $(\UCM,\UCL,B)$ of type $\Un$ over~$S$ we can associate a $\groupCO_n$-zip of type $(\chi,\Theta)$ over~$S$, obtaining a morphism of stacks
\UseTheoremCounterForNextEquation
\begin{equation}
\label{COnMorf}
\bigl(\!\bigl(\hbox{twisted orthogonal $F$-zips of type $\Un$}\bigr)\!\bigr)
\longto\COnZip^{\chi,\Theta}_k.
\end{equation}

Conversely, for any $\groupCO_n$-zip of type $(\chi,\Theta)$ over $S$ we evaluate the associated zip functor $\Fz$ on the standard representation $k^n$, the representation $k$ with the multiplier character~$\mu$, and the homomorphism~$B_1$, obtaining a twisted orthogonal $F$-zip $(\Fz(k^n),\Fz(k),\Fz(B_1))$ of type $\Un$ over~$S$. By showing that this construction yields a two-sided inverse of the first one proves that the morphism (\ref{COnMorf}) is an isomorphism. 


\subsection{Unitary $F$-zips versus $\groupU_n$-zips}
\label{UnitaryZips2}

Let $\BF_{q^2}$ denote a fixed quadratic extension of $\BF_q$, and let $\sigma$ denote its non-trivial automorphism $x\mapsto x^q$ over~$\BF_q$. 
Let $S$ be a scheme over~$\BF_q$. We call a triple $(\UCM,\rho,H)$ consisting of an $F$-zip $\UCM$ over~$S$, an $\BF_q$-algebra homomorphism $\rho\colon \BF_{q^2} \to \End(\UCM)$, and an admissible epimorphism $H\colon \UCM\otimes\UCM\to \BF_{q^2}\otimes_{\BF_q}\UBOne(0)$, which satisfies
\begin{enumerate}
\item[(a)] $H\circ(\rho(\alpha)\otimes\rho(\beta)) = (\alpha^q\beta\otimes1)\circ H$ for all $\alpha$, $\beta\in \BF_{q^2}$, and 
\item[(b)] $H(m_2,m_1) = (\sigma\otimes1)\circ H(m_1,m_2)$ for all local sections $m_1$, $m_2$ of~$\CM$,
\end{enumerate}
and whose hermitian pairing on the underlying sheaf $\CM\times\CM\to \BF_{q^2}\otimes_{\BF_q}\CO_S$ is nondegenerate everywhere, a \emph{unitary $F$-zip over~$S$}. 
To classify such objects we use base change fro $\BF_q$ to $\BF_{q^2}$:

\medskip
Let $\tilde S$ be a scheme over~$\BF_{q^2}$ and $(\tilde\UCM,\tilde\rho,\tilde H)$ a unitary $F$-zip over $\tilde S$. Then we have a unique decomposition $\tilde \CM=\tilde \CN\oplus\tilde \CN'$, where $\tilde \rho(\alpha)$ acts on $\tilde \CN$ through multiplication by $\alpha$ and on $\tilde \CN'$ through multiplication by~$\alpha^q$, and the hermitian pairing $\tilde H$ amounts to an isomorphism ${\tilde\CN'\stackrel{\sim}{\to}\tilde\CN^\vee}$. Working out the rest of the data we find that giving a unitary $F$-zip over $S$ is equivalent to giving a quadruple $(\tilde \CN,C^\bullet,D_\bullet,\psi_\bullet)$ consisting of a locally free sheaf of $\CO_S$-modules of finite rank $\tilde\CN$ on~$S$, a descending filtration $C^\bullet$ and an ascending filtration $D_\bullet$ of~$\tilde\CN$, and an $\CO_S$-linear isomorphism $\psi_i\colon (\gr_C^i\tilde\CN)^{(q)} \stackrel{\sim}{\to} (\gr^D_{-i}\tilde\CN)^\vee$ for every $i\in\BZ$.
We call $(\tilde\UCM,\tilde\rho,\tilde H)$ \emph{of type $\Un$} if the associated $\gr_C^i\tilde \CN$ is locally free of constant rank $n_i$ for all~$i$. For the same reasons as before the unitary $F$-zips of type~$\Un$, together with compatible isomorphisms, form a stack over~$\BF_{q^2}$. 

Let as above $S$ be a scheme over~$\BF_q$ and $(\UCM,\rho,H)$ a unitary $F$-zip over~$S$. We have $\CM = \proj_{2*}\widetilde\CM$ for a locally free sheaf of $\CO_{\tilde S}$-modules $\widetilde\CM$ on $\tilde S := \Spec\BF_{q^2}\times_{\Spec\BF_q}\nobreak S$, such that the action $\rho$ of $\BF_{q^2}$ is induced from the first factor. The $\BF_{q^2}$-invariant filtration $C^\bullet$ on $\CM$ then comes from a filtration of~$\widetilde\CM$. In this case we call a unitary $F$-zip \emph{of type $\Un$} if the associated $\gr_C^i\widetilde\CM$ is locally free of constant rank $n_i$ for all~$i$. Since the hermitian pairing $H$ induces isomorphisms $(\sigma\times\id)^*\gr_C^i\widetilde\CM \stackrel{\sim}{\to} (\gr_C^{-i}\widetilde\CM)^\vee$, this condition can be satisfied non-trivially only if $n_i=n_{-i}$ for all~$i$. Under this assumption the unitary $F$-zips of type~$\Un$, together with compatible isomorphisms, form a stack over~$\BF_q$. Moreover, a unitary $F$-zip is of type $\Un$ in this sense if and only if its pullback to $\tilde S$ is of type $\Un$ in the previous sense; hence the stack over $\BF_{q^2}$ described before is just the base change of the present stack over~$\BF_q$.

In summary, set $k:=\BF_q$ if $n_i=n_{-i}$ for all~$i$, respectively $k:=\BF_{q^2}$ if not; the unitary $F$-zips of type $\Un$ then form a natural stack over~$k$.

\medskip
Fix a non-degenerate $\sigma$-hermitian form $H_1\colon \BF_{q^2}^n\times \BF_{q^2}^n\to \BF_{q^2}$, and let $\groupU_{n,\BF_q} \subset \CR_{\BF_{q^2}/\BF_q}\GL_{n,\BF_{q^2}}$ denote the associated unitary group. 
The assumptions on $\Un$ and $k$ imply that there exists a cocharacter $\chi\colon \BG_{m,k}\to \groupU_{n,k}$, unique up to conjugation, whose weights on the standard representation $\BF_{q^2}^n$ of $\groupU_{n,\BF_{q^2}}$ are $i$ with multiplicity $n_i$ for all~$i$. Fixing such a cocharacter, we can thus speak of $\groupU_n$-zips of type~$\chi$ over any scheme $S$ over~$k$.

Also, set $M_1 := \BF_{q^2}^n\otimes_{\BF_q}k$ with the descending filtration $C^\bullet$ associated to~$\chi$ and the ascending filtration $D_\bullet$ associated to the Frobenius twist $\chi^{(q)}$, so that there are natural $\BF_{q^2}\otimes_{\BF_q}k$-linear isomorphisms $\phi_i\colon (\gr_C^iM_1)^{(q)} \stackrel{\sim}{\to} \gr^D_iM_1$ for all $i\in\BZ$. Then $\UCM_1 \defeq (M_1,C^\bullet,D_\bullet,\phi_\bullet)$ together with the evident action of $\BF_{q^2}$ and the pairing $H_1$ is a unitary $F$-zip of type $\Un$ over~$k$.

\medskip
Using the same principles as in the preceding subsections, to any unitary $F$-zip $(\UCM,\rho,H)$ of type $\Un$ over~$S$ we can now associate a $\groupU_n$-zip $\UI \defeq (I,I_+,I_-,\iota)$ of type $\chi$ over~$S$. Here $I$ is the sheaf of all $\BF_{q^2}\otimes_{\BF_q}\CO_{S'}$-linear isomorphisms $M_{1,S'}\isoto\CM_{S'}$ which are compatible with $H_1$ and~$H$, for all morphisms $S'\to S$, and $I_\pm$ are the subsheaves of isomorphisms which are in addition compatible with the filtrations~$C^\bullet$, respectively~$D_\bullet$, and $\iota$ is obtained from the graded isomorphisms~$\phi_\bullet$. Together this yields a morphism of stacks
\UseTheoremCounterForNextEquation
\begin{equation}
\label{UnMorf}
\bigl(\!\bigl(\hbox{unitary $F$-zips of type $\Un$}\bigr)\!\bigr)
\longto\UnZip^\chi_k.
\end{equation}

Conversely, for any $\groupU_n$-zip of type $\chi$ over $S$ we evaluate the associated zip functor $\Fz$ on the standard representation $\BF_{q^2}^n$, the obvious homomorphism $\BF_{q^2}\to \End_{\groupU_{n,\BF_q}}(\BF_{q^2}^n)$, and the hermitian pairing~$H_1$ (all of which are objects and morphisms in $\groupU_n$-$\Rep$), obtaining a unitary $F$-zip of type $\Un$ over~$S$. By showing that this construction yields a two-sided inverse of the first one proves that the morphism (\ref{UnMorf}) is an isomorphism.


\subsection{Twisted unitary $F$-zips versus $\groupCU_n$-zips}
\label{UnitaryZips1}

Again let $\BF_{q^2}$ denote a fixed quadratic extension of $\BF_q$, and let $\sigma$ denote its non-trivial automorphism $x\mapsto x^q$ over~$\BF_q$. We call a quadruple $(\UCM,\rho,\UCL,H)$ consisting of an $F$-zip $\UCM$ over~$S$, an $\BF_q$-algebra homomorphism $\rho\colon \BF_{q^2} \to \End(\UCM)$, an $F$-zip $\UCL$ of rank $1$ over~$S$, and an admissible epimorphism $H\colon \UCM\otimes\UCM\to \BF_{q^2}\otimes_{\BF_q}\CL$, which satisfies the same conditions (a) and (b) as in the preceding subsection and whose hermitian pairing on the underlying sheaf is nondegenerate everywhere, a \emph{twisted unitary $F$-zip over~$S$}. 

\medskip
If $S$ is a scheme over~$\BF_{q^2}$, for any twisted unitary $F$-zip over~$S$ there is a unique decomposition $\CM=\CN\oplus\CN'$, where $\rho(\alpha)$ acts on $\CN$ through multiplication by $\alpha$ and on $\CN'$ through multiplication by~$\alpha^q$, and it is compatible with the filtration~$C^\bullet$. In this case we call a twisted unitary $F$-zip over~$S$ \emph{of type $(\Un,d)$} if the associated $\gr_C^i\CN$ is locally free of constant rank $n_i$ for all~$i$ and $\UCL$ is of type~$d$. For the same reasons as before the twisted unitary $F$-zips of type $(\Un,d)$, together with compatible isomorphisms, form a stack over~$\BF_{q^2}$. There is no further condition on $(\Un,d)$ in this case.

If $S$ is only a scheme over~$\BF_q$, we still have $\CM = \proj_{2*}\widetilde\CM$ for a locally free sheaf of $\CO_{\tilde S}$-modules $\widetilde\CM$ on $\tilde S := \Spec\BF_{q^2}\times_{\Spec\BF_q}S$, such that the action $\rho$ of $\BF_{q^2}$ is induced by the first factor. In this case we call a twisted unitary $F$-zip \emph{of type $(\Un,d)$} if the associated $\gr_C^i\widetilde\CM$ is locally free of constant rank $n_i$ for all~$i$ and $\UCL$ is of type~$d$.
Since the hermitian pairing $H$ must induce isomorphisms $(\sigma\times\id)^*\gr_C^i\widetilde\CM \stackrel{\sim}{\to} (\gr_C^{d-i}\widetilde\CM)^\vee\otimes\proj_2^*\CL$, this condition can be satisfied non-trivially only if $n_i=n_{d-i}$ for all~$i$. Under this assumption the twisted unitary $F$-zips of type $(\Un,d)$, together with compatible isomorphisms, form a stack over~$\BF_q$. Moreover, a twisted unitary $F$-zip is of type $(\Un,d)$ in this sense if and only if its pullback to $\tilde S$ is of type $(\Un,d)$ in the previous sense; hence the stack over $\BF_{q^2}$ described before is just the base change of the present stack over~$\BF_q$.

In summary, set $k:=\BF_q$ if $n_i=n_{d-i}$ for all~$i$, respectively $k:=\BF_{q^2}$ if not; the twisted unitary $F$-zips of type $(\Un,d)$ then form a natural stack over~$k$.

\medskip
Fix a non-degenerate $\sigma$-hermitian form $H_1\colon \BF_{q^2}^n\times \BF_{q^2}^n\to \BF_{q^2}$, and let $\groupCU_{n,\BF_q} \subset \CR_{\BF_{q^2}/\BF_q}\GL_{n,\BF_{q^2}}$ denote the associated group of unitary similitudes, i.e., of sections $g$ that satisfy $H_1\circ(g\times g)=\mu(g)\cdot H_1$ for a scalar $\mu(g)$ in $\BG_{m,\BF_q}$. The assumptions on $(\Un,d)$ and $k$ imply that there exists a cocharacter $\chi\colon \BG_{m,k}\to \groupCU_{n,k}$, unique up to conjugation, whose weights on the standard representation $\BF_{q^2}^n$ of $\groupU_{n,\BF_{q^2}}$ are $i$ with multiplicity $n_i$ for all~$i$, and whose weight under the multiplier character $\mu$ is~$d$. Fixing such a cocharacter, we can thus speak of $\groupCU_n$-zips of type~$\chi$ over any scheme $S$ over~$k$.

\medskip
By the same procedure as before, to any twisted unitary $F$-zip $(\UCM,\rho,H)$ of type $(\Un,d)$ over~$S$ we can associate a $\groupCU_n$-zip of type $\chi$ over~$S$, obtaining a morphism of stacks
\UseTheoremCounterForNextEquation
\begin{equation}
\label{CUnMorf}
\bigl(\!\bigl(\hbox{twisted unitary $F$-zips of type $(\Un,d)$}\bigr)\!\bigr)
\longto\CUnZip^\chi_k.
\end{equation}

Conversely, for any $\groupCU_n$-zip of type $\chi$ over $S$ we evaluate the associated zip functor $\Fz$ on the standard representation $\BF_{q^2}^n$, the obvious homomorphism $\BF_{q^2}\to \End_{\groupCU_{n,\BF_q}}(\BF_{q^2}^n)$, the multiplier representation on~$\BF_q$, and the hermitian pairing~$H_1$ (all of which are objects and morphisms in $\groupCU_n$-$\Rep$), obtaining a twisted unitary $F$-zip of type $(\Un,d)$ over~$S$. By showing that this construction yields a two-sided inverse of the first one proves that the morphism (\ref{CUnMorf}) is an isomorphism.


\subsection{Other groups}
\label{OtherZips}

For each of the groups $\Ghat$ above, we have identified a finite subcategory $\CC$ of \hbox{$\Ghat$-$\Rep$} and have shown that any suitable functor $\CC\to \FZip(S)$ extends to a $\Ghat$-zip functor \hbox{$\Ghat$-$\Rep$}${}\to\FZip(S)$. Surely it must be possible to apply the same principle to an arbitrary reductive linear algebraic group $\Ghat$ over~$k$. However, identifying a suitable subcategory $\CC$ becomes tiresome very quickly. 

For instance, it should be possible to describe $\SO_n$-zip functors in terms of triples $(\UCM,B,\Delta)$ consisting of an $F$-zip $\UCM$ of rank~$n$, an everywhere non-degenerate symmetric pairing $B$ on~$\UCM$, and a trivialization $\Delta$ of the highest exterior power of~$\UCM$. However, to guarantee the extendability to \hbox{$\SO_{n,k}$-$\Rep$} one must also impose a certain relation between $B$ and $\Delta$ which is more complicated to describe. Moreover, the identification of the type of an $\SO_n$-zip functor might require some extra data, because $\SO_n$ may possess non-conjugate cocharacters which are conjugate under $\GL_n$.
A similar situation arises for the group $\SU_{n,k}$.



\section{Applications}\label{EO}


\subsection{Zip strata attached to smooth proper morphism with degenerating Hodge spectral sequence}
\label{Appl1}

In the following we give a generalization of a construction from \cite{moonwed}. Let $S$ be a scheme over~$\BF_p$, let $\CX$ be a Deligne-Mumford stack and let $f\colon \CX \to S$ be a morphism of finite type. For every \'etale morphism $U \to \CX$, where $U$ is a scheme, we set $\Omega^{\bullet}_{\CX/S}|U \defeq \Omega^{\bullet}_{U/S}$, where $\Omega^{\bullet}_{U/S}$ is the de Rham complex of $U$ over $S$. As the formation of the de Rham complex $\Omega^{\bullet}_{U/S}$ commutes with \'etale localization on $U$, this defines a complex of quasi-coherent sheaves of $\CO_{\CX}$-modules of finite type on the \'etale site on $\CX$ whose differentials are $f^{-1}\CO_S$-linear.

Attached to the naive and the canonical filtration of the de Rham complex $\Omega^{\bullet}_{\CX/S}$ we obtain two spectral sequences converging to the de Rham cohomology $H^{\bullet}_{{\rm DR}}(\CX/S)= {\BR}^{\bullet}f_*(\Omega^{\bullet}_{\CX/S})$, namely the Hodge-de Rham spectral sequence
\[
{}_HE^{ab}_1 = R^bf_*(\Omega_{\CX/S}^a) \implies H^{a+b}_{{\rm DR}}(\CX/S)
\]
and the conjugate spectral sequence
\[
{}_{\rm conj}E^{ab}_2 = R^af_*\bigl(\Hscr^b(\Omega^{\bullet}_{\CX/S})\bigr)
\implies H^{a+b}_{{\rm DR}}(\CX/S).
\]
In particular these spectral sequences endow $H^d_{{\rm DR}}(\CX/S)$ for $d \geq 0$ with two descending filtrations $({}_HF^iH^d_{{\rm DR}}(\CX/S))_{i\in\BZ}$ and $({}_{\rm conj}F^iH^d_{{\rm DR}}(\CX/S))_{i\in\BZ}$ by sheaves of $\CO_S$-submodules which are called the \emph{Hodge filtration} and the \emph{conjugate filtration}.
 
We denote by $F\colon \CX \to \CX^{(p)}$ the relative Frobenius of $\CX$ over $S$. For \'etale morphisms $g\colon U \to \CX$ the diagram
\[\xymatrix{
U \ar[r]^{\rm Frob} \ar[d]_g & U \ar[d]^g \\
\CX \ar[r]^{\rm Frob} & \CX,
}\]
where the horizontal morphisms are the absolute Frobenii, is cartesian. This shows that the formation of the relative Frobenius also commutes with \'etale base change. In particular, $F$ is representable and finite.

If $f$ is smooth, there is a unique isomorphism of graded sheaves of $\Oscr_{\CX^{(p)}}$-modules
\[
\Cscr^{-1}\colon \bigoplus_{b\geq 0}\Omega^b_{\CX^{(p)}/S} \liso
\bigoplus_{b\geq 0}\Hscr^b\bigl(F_*(\Omega^{\bullet}_{\CX/S})\bigr),
\]
the (inverse) Cartier isomorphism, which satisfies
\[
\begin{aligned}
\Cscr^{-1}(1) &= 1 \\
\Cscr^{-1}(d\sigma^{-1}(x)) &= \text{class of $x^{p-1}dx$} \\
\Cscr^{-1}(\omega \wedge \omega') &= \Cscr^{-1}(\omega) \wedge
\Cscr^{-1}(\omega').
\end{aligned}
\]
To see this, we remark that because of the uniqueness assertion one may work locally for \'etale topology on $\CX^{(p)}$. As the formation of differentials, of $\Hscr^i(\ )$, and of Frobenius is compatible with \'etale base change $U \to \CX$, the unique existence of $\Cscr^{-1}$ for Deligne-Mumford stacks follows from the analogous result for smooth morphisms of schemes.

From now on we assume that $f$ is smooth and proper. We fix an integer $d \geq 0$. We assume that $f$ and $d$ satisfy the following two conditions.
\begin{simplelist}
\item[\textup{(D1)}]
The sheaves of $\Oscr_S$-modules $R^bf_*(\Omega^a_{\CX/S})$ are locally free of finite rank for all
$a,b \geq 0$ with $a + b \leq d$.
\item[\textup{(D2)}]
The Hodge-de Rham spectral sequence ${}_HE^{ab}_1 = R^bf_*(\Omega_{\CX/S}^a) \implies H^{a+b}_{{\rm DR}}(\CX/S)$ degenerates for $a + b \leq d$ (i.e., for all $r \geq 1$ and $a,b$ with $a + b \leq d$ the differentials from and to ${}_HE^{ab}_r$ vanish).
\end{simplelist}
Then the formation of the Hodge-de Rham spectral sequence for $a + b \leq d$ commutes with base change $S' \to S$, and $H^e_{\rm DR}(\CX/S)$ is locally free of finite rank for $e \leq d$.

Now one has $R^af_*^{(p)} \circ F_* = R^af_*$ because $F$ is affine. Hence applying the functor $R^af_*^{(p)}$ to the Cartier isomorphism we obtain an isomorphism
\[
R^af_*^{(p)}(\Omega^b_{\CX^{(p)}/S}) \liso R^af_*\bigl(\Hscr^b(\Omega^{\bullet}_{\CX/S})\bigr).
\]
Because of Condition~(D1) the $\Oscr_S$-modules $R^af_*(\Omega^b_{X/S})$ are flat for all $a,b \geq 0$ with $a + b \leq d$ and we obtain isomorphisms
\UseTheoremCounterForNextEquation
\begin{equation}\label{Cartier2}
\varphi^{ab}\colon R^af_*(\Omega^b_{\CX/S})^{(p)} =
({}_HE^{ba})^{(p)} \iso {}_{\rm conj}E^{ab}_2 =
R^af_*\bigl(\Hscr^b(\Omega^{\bullet}_{\CX/S})\bigr).
\end{equation}
This implies that the conjugate spectral sequence also degenerates for $a + b \leq d$ and that its formation commutes with arbitrary base change for $a + b \leq d$ (see~\cite{Katz_DiffEq}~2 if $\CX$ is a scheme; the arguments for Deligne-Mumford stacks $\CX$ are verbatim the same).

\begin{remark}
We list some examples of morphisms $f$ and integers $d$ that satisfy Conditions~(D1) and~(D2).
\begin{assertionlist}
\item[(a)]
By \cite{moonwed}, conditions~(D1) and~(D2) are satisfied for all $d$ in case $\CX$ is a smooth proper relative curve over $S$, in case $\CX$ is an abelian scheme over $S$, in case $\CX$ is a smooth toric scheme over $S$ and in case $\CX$ is a relative K3-surface over $S$. 
\item[(b)]
Conditions~(D1) and~(D2) are satisfied for all $d \leq p-1$ if there exists a flat scheme $\tilde S$ over $\BZ/p^2\BZ$ satisfying $\tilde S \otimes_{\BZ/p^2\BZ} \BF_p \cong S$ and a smooth proper lift of $\CX$ to $\tilde S$.

This is shown in \cite{DI_Deg} if $\CX$ is a scheme, and the proof carries over verbatim to the case of Deligne-Mumford stacks because the formation of the de Rham complex and the relative Frobenius is compatible with pull back via \'etale morphisms $X' \to X$ (see also \cite[Theorem~3.7]{Satriano_DeRham} by Satriano for a generalization to tame Artin stacks; note that Satriano formulates only the case where $S = \Spec k$ for a perfect field $k$ but combining his proof with the proof over a general base scheme in~\cite{DI_Deg} also shows the general case).
\item[(c)]
Let $S = \Spec k$ for a perfect field $k$. Let $X$ be a smooth proper scheme over $k$ and let $D \in {\rm Div}(X) \otimes \BQ$ be a $\BQ$-divisor whose support has only normal crossings and such that exists an integer $b$ prime to ${\rm char}(k)$ such that $bD$ is integral. Then in~\cite{MatsOls} there is attached a morphism $\CX \to X$, where $\CX$ is a smooth proper Deligne-Mumford stack which is the ``minimal covering'' of $X$ such that $D$ becomes integral. Moreover the authors show that each lift of $(X,D)$ over $W_2(k)$ yields also a smooth proper lift of $\CX$ making it possible to apply~(b). We refer to [loc.~cit.,Theorem~4.1] for details.
\end{assertionlist}
\end{remark}

We associate to~$f$ and $d$ an $F$-zip $(\CM,C^{\bullet},D_{\bullet},\varphi_{\bullet})$ over~$S$ as follows: Set $\CM = H^d_{{\rm DR}}(\CX/S)$. Let $C^{\bullet}$ be the Hodge filtration on~$\CM$, and define the
filtration~$D_{\bullet}$ by $D_i = {}_{\rm conj}F^{d-i}H^d_{{\rm DR}}(\CX/S)$. As the formation of both spectral sequences commutes with arbitrary base change, $C^{\bullet}$ and $D_{\bullet}$ are filtrations by locally direct summands, i.e. they are filtrations in the sense of Subsections~\ref{PreDescFilt} and~\ref{PreAscFilt}. The assumption of the degeneracy of the Hodge spectral sequence and hence of the conjugate spectral sequence shows that one has functorial isomorphisms
\UseTheoremCounterForNextEquation
\begin{equation}\label{GradedCD}
\begin{aligned}
\gr^i_CH^d_{\rm DR}(\CX/S) &\cong R^{d-i}f_*(\Omega^i_{\CX/S}) \\
\gr_i^DH^d_{\rm DR}(\CX/S) &\cong R^{d-i}f_*(\Hscr^i(\Omega^{\bullet}_{\CX/S}))
\end{aligned}
\end{equation}
Finally, let
\[
\varphi_i := \varphi^{d-i,i}\colon (\gr^i_C)^{(p)} =
R^{d-i}f_*(\Omega^i_{\CX/S})^{(p)} \iso  \gr_i^D =
R^{d-i}f_*(\Hscr^i(\Omega^{\bullet}_{\CX/S}))\, ,
\]
where $\varphi^{d-i,i}$ is the isomorphism defined in~\eqref{Cartier2}. We denote this $F$-zip by $\Hline^d_{\rm DR}(\CX/S)$. For $i \in \BZ$ set
\[
n_i := \begin{cases}h^{d-i,i} = {\rm rk}(R^{d-i}f_*(\Omega^i_{X/S})),&\text{for $0 \leq i \leq d$;}\\
0,&\text{otherwise.}
\end{cases}
\]
This is a locally constant function on $S$. If $n_i$ is constant for all $i$, which is automatic if $S$ is connected, then $\Un = (n_i)$ is the type of the $F$-zip $\Hline^d_{\rm DR}(\CX/S)$. Thus for $n\defeq \sum_i n_i$ the isomorphism \eqref{GLnMorf} yields a $\GL_n$-zip $\UI$ of type $\chi$ over $S$, where $\chi$ is the 
cocharacter of $\GL_n$ associated to $\Un$ as in Subsection~\ref{GLnZips}. By~\eqref{GZipStrat} we obtain a decomposition into locally closed subschemes
\[
S = \bigcup_{w \in \leftexp{I}{W}} S^{w}_{\UI}
\]
indexed by
\[
\leftexp{I}{W} = \set{w \in S_h}{\forall\,i \in \BZ: w^{-1}(\sum_{j < i}n_j+1) < \dots < w^{-1}(\sum_{j < i}n_j+n_i)}.
\]
Let $\preceq$ be the partial order on $\leftexp{I}{W}$ given by \eqref{DefPartialOrder}. By the inclusion \eqref{GZipStratClos} and Proposition \ref{Generizing1} one has
\[
\overline{S^{w}_\UI} \subseteq \bigcup_{w' \preceq w} S^{w'}_{\UI}
\]
with equality if the classifying morphism $S \to \GLnZip^{\chi}_{\BF_p}$ of the $\GL_n$-zip $\UI$ is generizing.

\subsection{Cup product and duality}
\label{Appl2}

The cup product in de Rham cohomology yields a bilinear map of $F$-zips, as follows. As the cup product has not yet been worked out for Deligne-Mumford stacks (as far as we know), we restrict ourself to the case that $f\colon X \to S$ is a smooth and proper morphisms of \emph{schemes} over $\BF_p$ satisfying conditions (D1) and (D2) for all $d$. Using the K\"unneth formula for hypercohomology of complexes (\cite{EGAIII_2}~(6.7.8), applicable because $\Omega^a_{X/S}$ and $H^d_{\rm DR}(X/S)$ are $S$-flat for all $a$ and $d$) one sees that the wedge product
\UseTheoremCounterForNextEquation
\begin{equation}\label{WedgeProduct}
\Omega^{\bullet}_{X/S} \otimes_{f^{-1}\Oscr_S} \Omega^{\bullet}_{X/S} \to \Omega^{\bullet}_{X/S}
\end{equation}
induces a homomorphism of locally free graded $\Oscr_S$-modules of finite rank
\UseTheoremCounterForNextEquation
\begin{equation}\label{CupProduct}
\cup\colon H^{\bullet}_{\rm DR}(X/S) \otimes_{\Oscr_S} H^{\bullet}_{\rm DR}(X/S) \lto H^{\bullet}_{\rm DR}(X/S),
\end{equation}
the cup product. This makes $H^{\bullet}(X/S)$ into a graded anti-commutative $\Oscr_S$-algebra. It is easily checked that the wedge product sends the tensor product of the naive (resp.\ canonical) filtrations to the naive (resp.\ canonical) filtration. Thus by functoriality of the spectral sequence associated to a filtered complex the cup product induces for all $d,e \geq 0$ a morphism of filtered locally free modules of finite rank
\UseTheoremCounterForNextEquation
\begin{equation}\label{CupComponent}
\cup\colon H^{d}_{\rm DR}(X/S) \otimes_{\Oscr_S} H^{e}_{\rm DR}(X/S) \lto H^{d+e}_{\rm DR}(X/S).
\end{equation}
In particular we obtain induced pairings on the associated graded pieces. Moreover, using the defining properties of the Cartier isomorphism, one sees that there is a commutative diagram
\UseTheoremCounterForNextEquation
\begin{equation}\label{HodgeGradedDual}
\begin{aligned}
\xymatrix{
\gr^i_CH^d_{\rm DR}(X/S)^{(p)} \otimes \gr^{j}_CH^{e}_{\rm DR}(X/S)^{(p)} \ar[r] \ar[d]_{\varphi^{d-i,i} \otimes \varphi^{e-j,j}} & \gr^{i+j}_CH^{d+e}_{\rm DR}(X/S)^{(p)}\ar[d]^{\varphi^{d+e-i-j,i+j}} \\
\gr_i^DH^d_{\rm DR}(X/S) \otimes \gr_{j}^DH^{e}_{\rm DR}(X/S) \ar[r] & \gr_{i+j}^DH^{d+e}_{\rm DR}(X/S).
}
\end{aligned}
\end{equation}
Hence we obtain a morphism of $F$-zips over $S$
\UseTheoremCounterForNextEquation
\begin{equation}\label{FZipCup1}
\cup\colon \Hline^d_{\rm DR}(X/S) \otimes \Hline^{e}_{\rm DR}(X/S) \to \Hline^{d+e}_{\rm DR}(X/S),
\end{equation}

\begin{example}
Let $A \to S$ be an abelian scheme. Then the cup product yields an isomorphism of graded anti-commutative algebras
\[
\Lambda^{\bullet}H^1_{\rm DR}(A/S) \iso H^{\bullet}_{\rm DR}(A/S)
\]
(see e.g.\ \cite[Prop. 2.5.2]{BBM}). The above arguments show that this is in fact an isomorphism of $F$-zips.
\end{example}

Now assume in addition that $f$ has geometrically connected fibers of fixed dimension $n$. Then we have a trace isomorphism
\UseTheoremCounterForNextEquation
\begin{equation}\label{Trace}
R^nf_*\Omega^n_{X/S} \cong H^{2n}_{\rm DR}(X/S) \lisoover{{\rm tr}} \Oscr_S.
\end{equation}
In other words, we obtain an isomorphism of $F$-zips
\UseTheoremCounterForNextEquation
\begin{equation}\label{Trace2}
\Hline^{2n}_{\rm DR}(X/S) \lisoover{{\rm tr}} \UBOne(n).
\end{equation}
The cup-product pairings of $F$-zips
\UseTheoremCounterForNextEquation
\begin{equation}\label{FZipCup2}
\Hline^d_{\rm DR}(X/S) \otimes_{\Oscr_S} \Hline^{2n-d}_{\rm DR}(X/S) \lto \Hline^{2n}_{\rm DR}(X/S) = \UBOne(n)
\end{equation}
are perfect dualities (\cite{Katz_DiffEq}~(2.3.5.1)), i.e. they yield isomorphisms of $F$-zips
\UseTheoremCounterForNextEquation
\begin{equation}\label{FZipCupPerfect}
\Hline^d_{\rm DR}(X/S) \iso \Hline^{2n-d}_{\rm DR}(X/S)^{\vee}(n).
\end{equation}
For $d = n$ the morphism~\eqref{FZipCup2} factors through
\UseTheoremCounterForNextEquation
\begin{equation}\label{MiddleDuality}
\begin{aligned}
v\colon \Lambda^2\Hline^n_{\rm DR}(X/S) \to \UBOne(n), & \qquad\text{if $n$ is odd;}\\
b\colon S^2\Hline^n_{\rm DR}(X/S) \to \UBOne(n), & \qquad\text{if $n$ is even.}
\end{aligned}
\end{equation}
In other words the pairing is symplectic if $n$ is odd and it is symmetric if $n$ is even.

For $n$ odd we hence obtain a twisted symplectic $F$-zip $(\Hline^n_{\rm DR}(X/S), \UBOne(n), v)$. For $n$ even (and $p > 2$) we obtain a twisted orthogonal $F$-zip $(\Hline^n_{\rm DR}(X/S), \UBOne(n), b)$.

\subsection{Zip strata attached to truncated Barsotti-Tate groups of level $1$}
\label{BT1}

Let $S$ be a scheme over $\BF_p$ and let $X$ be a truncated Barsotti-Tate group of level $1$ over $S$. We denote by $X\vdual$ its Cartier dual. Let $\BD(X)$ be its \emph{covariant} Dieudonn\'e crystal and let $\CM(X)$ be its evaluation at the trivial PD-thickening $(S,S,0)$. Then there is an exact sequence, functorial in $X$ and compatible with base change $S' \to S$
\UseTheoremCounterForNextEquation
\begin{equation}
\label{LieUniversalExt}
0 \to \omega_{X\vdual} \to \CM(X) \to \Lie(X) \to 0,
\end{equation}
where $\omega_{X\vdual} = e^*\Omega_{X\vdual/S}$ is the sheaf of $\CO_S$-modules of invariant differentials of $X\vdual$ (see \cite[3.2]{BBM}). In particular the relative Frobenius $F\colon X \to X^{(p)}$ and the Verschiebung $V\colon X^{(p)} \to X$ induce $\CO_S$-linear homomorphisms
\[
\CF := \CM(V)\colon \CM(X)^{(p)} \to \CM(X), \qquad \CV := \CM(F)\colon \CM(X) \to \CM(X)^{(p)}.
\]
Note that the roles of $F$ and $V$ are switched as we are considering covariant Dieudonn\'e theory. Moreover
\[
(\omega_{X\vdual})^{(p)} = \Ker(\CF) = {\rm Im}(\CV), \qquad \Ker(\CV) = {\rm Im}(\CF)
\]
are locally direct summands of $\CM(X)^{(p)}$ and of $\CM(X)$, respectively.

We attach an $F$-zips $\UCM(X) := (\CM(X),C^\bullet,D_\bullet,\phi_\bullet)$ as follows. Set
\begin{alignat*}{6}
C^0 &:= \CM(X), &\qquad C^1 &:= \omega_{X\vdual}, &\qquad C^2 &:= 0 \\
D_{-1} &:= 0, &\qquad D_0 &:= \Ker(\CV), &\qquad D_1 &:= \CM(X)
\end{alignat*}
and let
\[
\varphi_0\colon \CM(X)^{(p)}/(C^1)^{(p)} \to D_0, \qquad \varphi_1\colon (C^1)^{(p)} \to \CM(X)/D_0
\]
be the $\CO_S$-linear isomorphisms induced by $\CF$ and $\CV^{-1}$, respectively.

Altogether we obtain a functor $X \sends \UCM(X)$ from the category of truncated Barsotti-Tate groups of level $1$ over $S$ to the category of $F$-zips over $S$. Moreover it follows from~\cite[5.2]{BBM} that there is an isomorphism of $F$-zips
\UseTheoremCounterForNextEquation
\begin{equation}\label{DualBT1}
\UCM(X\vdual) \cong \UHom(\UCM(X),\UBOne(1)),
\end{equation}
which is functorial in $X$ and compatible with base change. If $d := {\rm rk}_{\CO_S}(\Lie X)$ is the dimension of $X$ and $n$ its height, then the type of the $F$-zip $\UCM$ is $(n_i)_i$ with
\UseTheoremCounterForNextEquation
\begin{equation}\label{typeBT1}
n_0 = d, \qquad n_1 = n - d, \qquad n_i = 0 \quad \text{for $i \ne 0,1$}.
\end{equation}

Truncated Barsotti-Tate groups of level $1$ of height $n$ and dimension $d$ over schemes over $\BF_p$ form a smooth algebraic stack ${\tt BT}_1^{n,d}$ of finite type over $\BF_p$ (\cite[Prop. 1.8]{Wd_DimOort}) and the above construction yields a morphism of algebraic stacks
\UseTheoremCounterForNextEquation
\begin{equation}\label{BT1FZip}
\Phi\colon {\tt BT}_1^{n,d} \to \FZip^\Un_{\BF_p},
\end{equation}
where $\Un$ is given by~\eqref{typeBT1}. By Dieudonn\'e theory this functor is an equivalence on points with values in a perfect field. In particular, for every algebraically closed field $K$ of characteristic $p$ we obtain a bijection
\UseTheoremCounterForNextEquation
\begin{equation}\label{ClassifyBT1}
\begin{aligned}
&\left\{
\begin{matrix}
\text{isomorphism classes of truncated Barsotti-Tate groups} \\
\text{over $K$ of level $1$, height $n$, and dimension $d$}
\end{matrix}
\right\} \\
\bijective\ 
&\set{w \in S_n}{w^{-1}(1) < \dots < w^{-1}(d), w^{-1}(d+1) < \dots < w^{-1}(n)}
\end{aligned}
\end{equation}
This was first proved by Moonen in \cite{Moonen_Weyl}.

The following results (all due to Eike Lau) show that $\Phi$ \eqref{BT1FZip} is a smooth (non-representable) morphism.

\begin{remark}\label{TruncBTSmooth}
Let $R$ be an $\BF_p$-algebra. Let $\sigma$ be the ring endomorphism $x \sends x^p$, let $I := R_{(\sigma)}$ be the restrictions of scalars of the $R$-module $R$ under $\sigma$, and let $\sigma_1\colon I \to R$ be the $\sigma$-linear map given by the identity of $R$. Lau has defined in \cite{Lau_SmoothTrunc} the notion of a \emph{display of level $1$} over $R$. Recall that this is a tuple ${\mathcal D} = (P,Q,\iota,\eps,F,F_1)$ consisting of $R$-modules $P$ and $Q$ together with $R$-linear maps $I \otimes P \ltoover{\eps} Q \ltoover{\iota} P$ such that $P$ and ${\Coker}(\iota)$ are finitely generated and projective and such that the following sequence is exact
\[
0 \lto I \otimes \Coker(\iota) \ltoover{\eps} Q \ltoover{\iota} P \lto \Coker(\iota) \lto 0.
\]
Finally, $F\colon P \to P$ and $F_1\colon Q \to P$ are $\sigma$-linear maps such that $F_1(Q)$ generates $P$ and such that $F_1 \circ \eps = \sigma_1 \otimes F$. The rank of $P$ is called the \emph{height of ${\mathcal D}$} and the rank of ${\Coker}(\iota)$ is called the \emph{dimension of ${\mathcal D}$}. One has the obvious notion of an isomorphism of displays of level $1$ and of base change for a ring homomorphism $R \to R'$. One obtains the category ${\tt Disp}^{n,d}_1$ of level $1$ displays of height $n$ and dimension $d$ fibered over the category of $\BF_p$-algebras. This is a smooth algebraic stack over $\BF_p$ (\cite{Lau_SmoothTrunc} Proposition~3.15).

To every display ${\mathcal D} = (P,Q,\iota,\eps,F,F_1)$ of level $1$ of height $n$ and dimension $d$ over $R$ one can attach an $F$-zip of type $\Un$ with $\Un$ as in \eqref{typeBT1} as follows. We set $M = P$, $C^1 = {\rm im}(\iota\colon Q \to P)$, $D_0 := {\rm im}(F^{\sharp}\colon P^{(\sigma)} \to P$), where $F^{\sharp}$ denotes the linearization of $F$, and we define $\varphi_0$, $\varphi_1$ as the linearizations of the $\sigma$-linear maps $\varphi_0^{\flat}$, $\varphi_1^{\flat}$ defined by the following commutative diagrams
\[\xymatrix{
P \ar[r]^F \ar[d] & P \ar[d] & & Q \ar[r]^{F_1} \ar[d] & P \ar[d]\\
P/C^1 \ar[r]^{\varphi_0^{\flat}} & D_0 & & C^1 \ar[r]^{\varphi_1^{\flat}} & P/D_0.
}\]
Then it is straight forward to check (by choosing a normal decomposition, see \cite{Lau_SmoothTrunc}~Subsection~3.2) that this contruction defines an equivalence of the category of displays of level $1$ of height $n$ and dimension $d$ over $R$ with the category of $F$-zips of type $\Un$ over $R$. We obtain an equivalence of algebraic stacks $\Theta\colon {\tt Disp}^{n,d}_1 \iso \FZip^\Un_{\BF_p}$.

Lau has defined a morphism $\Psi\colon {\tt BT}_1^{n,d} \to {\tt Disp}^{n,d}_1$ of algebraic stacks such that the composition with the equivalence $\Theta$ is the morphism $\Phi$ (\cite{Lau_SmoothTrunc}~Section.~4). Moreover he proves that $\Psi$ is a smooth morphism (Theorem~A of loc.~cit.) which shows the smoothness of $\Phi$.
\end{remark}

\begin{example}\label{RapZink}
Let $\BX$ be a $p$-divisible group of height $n$ and dimension $d$ over a finite field $\BF_q$ (where $q$ is a power of $p$). Let $\CN$ be the attached Rapoport-Zink space over $\BF_q$ (\cite{RaZi}), i.e., $\CN(S)$ consists for every scheme $S$ over $\BF_q$ of isomorphism classes of pairs $(X,\rho)$, where $X$ is a $p$-divisible group over $S$ and $\rho\colon \BX_S \to X$ is a quasi-isogeny. Then $\CN$ is representable by a formal scheme locally formally of finite type over $\BF_q$. Attaching to $X$ its $p$-torsion $X[p]$ defines a morphism
\[
\epsilon\colon \CN \to ({\tt BT}_1^{n,d}) \otimes \BF_q
\]
which is formally smooth by the main result of~\cite{Illusie_Def} and by Drinfeld's result that quasi-isogenies between $p$-divisible groups can always be deformed uniquely. Composing $\epsilon$ with the smooth morphism $\Phi$ from \eqref{BT1FZip} we obtain a formally smooth (and hence generizing) morphism $\CN \to \FZip^\Un_{\BF_q}$ and hence locally closed formal subschemes $\CN^w$ as in~\eqref{GZipStrat} with
\[
\overline \CN^w = \bigcup_{w' \preceq w} \CN^{w'}.
\]
This can be generalized to other Rapoport-Zink spaces.
\end{example}

For every truncated Barsotti-Tate group $X$ of level $1$ over a scheme $S$ over $\BF_p$
the morphism $\Phi\colon {\tt BT}_1^{n,d} \to \FZip^\Un_{\BF_p}$ from \eqref{BT1FZip} 
yields a homomorphism of the automorphism group schemes
\[
\alpha\colon \UAut(X) \to \UAut(\UCM(X)).
\]
As both stacks are quotient stacks of a linear group acting on a scheme of finite type over $\BF_p$, the two group schemes $\UAut(X)$ and $\UAut(\UCM(X))$ are affine and of finite type over $S$ (Proposition~\ref{AutoStack1}). If $S = \Spec K$ for an algebraically closed field~$K$, it is shown in \cite{Wd_DimOort}~(5.7), that $\alpha$ induces a homeomorphism of the reduced subgroup schemes
\[
\UAut(X)_{\rm red} \liso \UAut(\UCM(X))_{\rm red}.
\]
Hence we can use Proposition~\ref{AutoGZip} to describe $\UAut(X)$.

\begin{proposition}\label{AutBT1}
Let $X$ be a truncated Barsotti-Tate group of level $1$ over an algebraically closed field $K$ of characteristic $p$. Let $n$ be its height and $d$ the dimension of its Lie algebra. Let $w$ be the permutation corresponding to the isomorphism class of $X$ via the bijection~\eqref{ClassifyBT1}.
\begin{itemize}
\item[(a)]
The reduced subgroup scheme of the identity component $\UAut(X)$ is a unipotent linear algebraic group of dimension $d(n-d) - \ell(w)$. In particular
\[
\dim(\UAut(X)) = d(n-d) - \ell(w).
\]
\item[(b)]
The group of connected components of $\UAut(X)$ is isomorphic to the group $\Pi$ defined in Proposition~\ref{AutoGZip}.
\end{itemize}
\end{proposition}

Note that for a permutation $w \in S_n$ the length can be easily computed by
\[
\ell(w) = \#\set{(i,j)}{1 \leq i < j \leq n, w(i) > w(j)}.
\]

\bibliographystyle{plain}
\bibliography{references}

\end{document}